\documentclass[11pt,a4paper,reqno]{article}
\usepackage[margin=1.8cm]{geometry}
\usepackage{amsthm,amsmath,amssymb,mathrsfs}

\usepackage{indentfirst,graphicx,subcaption}
\usepackage{booktabs,float}
\usepackage[hidelinks]{hyperref}
\numberwithin{equation}{section}
\theoremstyle{plain}

\newtheorem{theorem}{Theorem}[section]
\newtheorem{definition}{Definition}[section]
\newtheorem{lemma}{Lemma}[section]
\newtheorem{assumption}{Assumption}[section]
\newtheorem{proposition}{Proposition}[section]
\newtheorem{remark}{Remark}[section]
\newtheorem{corollary}{Corollary}[section]
\title{Topological properties and a multiplicative Bloch-Floquet-Zak transform for scattering by self-similar or fractal media}
\date{}
\author{
    Habib Ammari\thanks{Department of Mathematics, ETH Zürich, Rämistrasse 101, 8092 Zürich, Switzerland and Hong Kong Institute for Advanced Study, City University of Hong Kong, Kowloon Tong, Hong Kong Special Administrative Region. \texttt{habib.ammari@math.ethz.ch}.} 
    \and
    Yat Tin Chow \thanks{Department of Mathematics, University of California, Riverside, Riverside, CA, USA. \texttt{yattinc@ucr.edu}.}
    \and 
    Fuqun Han \thanks{Department of Mathematics, City University of Hong Kong, Kowloon Tong, Hong Kong Special Administrative Region. \texttt{fuqun.han@cityu.edu.hk}  (corresponding author).} 
}

\begin{document}

\maketitle
\begin{abstract}    
We develop the first rigorous topological framework for an auxiliary boundary-integral model motivated by physical scaling identities for wave scattering from self-similar or fractal media. The model is periodic in logarithmic scale, and a multiplicative Bloch-Floquet-Zak transform fiberizes its interscale coupling. For sufficiently small, well-separated components, equilibrium densities define a computable finite-dimensional projected matrix, while Riesz projections select the corresponding invariant spectral subspaces of the full boundary symbol. Under suitable spectral-isolation and point-gap conditions, and after recentering the two families at their respective same-scale reference values, we prove that the determinant winding of the exact Riesz-reduced family agrees with that of the projected matrix family. For regular-simplex configurations, we derive explicit nonzero winding formulas and track the resulting local winding data across finite prefractal levels and along a geometric sequence of wavenumbers. We also formulate conditional winding data for multiple dilation centers and show that the Zak phase of the chiral Hermitianization equals \(\pi\) times the point-gap winding modulo \(2\pi\). Thus, the scale-periodic boundary model admits a rigorous topological reduction.
\end{abstract}

\section{Introduction}
\label{sec:introduction}

Fractal and self-similar structures arise in wave propagation, photonics, metamaterials, and material design. The self-similar structures considered here are characterized not merely by geometric irregularity but by a hierarchy of length scales generated by repeated dilation. Classical examples such as Cantor sets, Sierpiński gaskets, and carpets are described by recursive constructions and dilation symmetries. Mathematically, such structures are studied through Hausdorff dimension and upper and lower Minkowski dimensions, iterated function systems, and more recently fractal zeta functions and complex dimensions \cite{falconer2014fractal,lapidus_van_frankenhuijsen2013,lapidus_radunovic_zubrinic2017distance}. In the latter theory, the poles of Mellin-type zeta functions encode logarithmic oscillations of tube functions and thereby detect multiplicatively periodic oscillations in scale that are not captured by a single real-valued dimension \cite{lapidus_radunovic_zubrinic2017distance,lapidus_radunovic_zubrinic2018tube,lapidus_pearse2010,hoffer2025tube}.

In the auxiliary construction studied here, logarithmic scale plays the role of the spatial coordinate in a periodic medium. The multiplicative BFZ (Bloch-Floquet-Zak) transform introduces a dual scale parameter and organizes the interactions between successive scale layers. Our aim is to define and analyze determinant winding invariants for the resulting scale-periodic boundary-symbol model and to compare the exact invariant reduction selected by Riesz projections with a computable equilibrium-density projection related to the capacitance-matrix approach in subwavelength scattering.

Wave propagation and scattering involving fractal structures, measures, or boundaries have also been studied from several analytic and physical viewpoints. One-dimensional potential scattering with compactly supported Radon-measure potentials has been analyzed, including the recovery of fractal dimensions from the reflection amplitude \cite{guerin1996scattering}. In a different physical setting, multiple-scattering resonances and wave localization have been studied in multifractal arrays of electric dipoles \cite{chen2023enhanced}. A related mathematical literature concerns Laplacians on self-similar fractals and domains with fractal boundaries, including Weyl-Berry asymptotics and spectral zeta functions \cite{brossard1986can,kigami_lapidus1993,lapidus1991fractal, lapidus1994analysis,teplyaev2004spectral}. Wave propagation and inverse problems in domains with irregular or fractal boundaries provide another closely related perspective \cite{wave_fractal_Anna}. These works describe how fractal structure influences scattering amplitudes, localization, spectral asymptotics, and boundary-value problems; the present work instead studies topological winding invariants of a scale-periodic boundary-symbol model.

In parallel, topological wave phenomena have become a central theme in photonics, acoustics, and metamaterials. In periodic media, Floquet-Bloch theory converts translation symmetry into a fibered operator family over the Brillouin zone, and topological invariants of the fibers can predict robust edge or interface modes. For high-contrast subwavelength resonators, the wave problem can often be reduced asymptotically to a finite-dimensional generalized capacitance matrix, from whose Bloch eigenvalues and eigenmodes one computes Zak phases and related winding invariants \cite{cbms,ammari_davies_hiltunen_yu2020,ammari_davies_hiltunen2022,ammari_barandun_cao_feppon2023}. Recent work has also demonstrated topological effects in fractal or non-integer-dimensional structures, including photonic Floquet topological insulators on fractal lattices and topological edge or corner states in bismuth fractal nanostructures \cite{yang2020photonic,canyellas2024topological,fractal-photonic1,fractal-photonic2}. These developments motivate the question of whether self-similar scattering problems admit scale-adapted analogues of Bloch topological invariants.

The central novelty is that, in the auxiliary model, periodicity occurs in logarithmic scale rather than in physical space through the assembly of scatterers at a fixed scale. For a scale ratio \(\lambda>1\), successive layers are indexed by \(\{\lambda^m:m\in\mathbb Z\}\), and the multiplicative BFZ transform is the discrete analogue of Fourier analysis in the logarithmic radial variable. The resulting winding invariants are therefore attached to a scale-periodic boundary-symbol loop generated by dilation, rather than to an ordinary translation-periodic array or merely to a finite collection of coupled scatterers. Since dilation by a factor \(a\) changes the Helmholtz wavenumber from \(k\) to \(ak\), the fixed-wavenumber problem is not itself scale invariant. Exact scaling identities for the Green function, volume potential, and first Born far field nevertheless reveal a scale-frequency recursion and motivate, but do not themselves provide, the auxiliary scale-periodic model studied below. The small, well-separated components introduced later serve to make the reduced loop computable through equilibrium densities; they do not supply the periodic structure, which lies in the scale variable.

Accordingly, the topological analysis is carried out for an auxiliary smooth boundary model obtained by repeating a reference boundary over the discrete dilation group and pulling all scale layers back to one reference cell. The normalized interscale coupling is then a block Laurent operator whose coefficients are bounded for enlarged source layers and decay geometrically for contracted ones. On a suitable exponentially weighted sequence space, the BFZ transform fiberizes this operator into a $C^1$ loop of boundary operators. If a contour uniformly isolates a finite-rank spectral cluster and the resulting Riesz subspaces are identified with a fixed reference space, the invariant restrictions define a finite-dimensional reduced loop; a point gap of this loop then defines an integer determinant winding.

For sufficiently small, well-separated components, equilibrium densities and their normalized constant dual functionals, identified through the Calderón identity, define a computable finite-dimensional projected BFZ matrix. We identify the corresponding self-interaction Riesz space with the equilibrium-density space. On the self-interaction Riesz space, the fixed-space reduction approximates the exact Riesz-reduced family to quadratic order in the norm of the intercomponent and interscale perturbation. Combining this estimate with the comparison to the equilibrium-density projection, and recentering the exact and projected families at their respective same-scale (zeroth Fourier) reference values, we show that their determinant windings agree under the stated assumptions. For regular-simplex configurations of $M\in\{2,3,4\}$ components in $\mathbb R^3$, permutation symmetry decomposes the reduced space into the symmetric line and the $(M-1)$-dimensional sum-zero sector. Under the stated nonvanishing condition and in the regime of small components and sufficiently large scale ratio, the exact scalar loops associated with the two sectors each have winding number $-1$ about their respective same-scale reference values. The corresponding determinant windings of the full reduced matrix are $-1$ and $-(M-1)$ about the symmetric and sum-zero reference points, respectively.

Because determinant winding records eigenvalue motion around a point gap rather than eigenvector transport, we distinguish the determinant-line and bandwise biorthogonal Zak phases from the phase associated with chiral Hermitianization. The bandwise phase additionally requires a globally separated, periodically labelable simple eigenvalue branch. For a $C^1$ reduced loop, the chiral construction requires only a point gap, and we prove that the Zak phase of its negative spectral subspace equals $\pi$ times the determinant winding modulo $2\pi$. For finite prefractals, we record only local winding data obtained from uncoupled reference loops with the multiplicities of the iterative construction.  These data are not asserted to be determinant windings of the fully coupled finite-prefractal boundary operator.

The paper is organized as follows. Section~\ref{sec:selfsimilar_prelim} derives the physical scaling identities and formulates the auxiliary scale-periodic boundary model. Section~\ref{sec:high_freq_gap} constructs the weighted Laurent realization, performs the scale BFZ fiberization, and defines the determinant winding on isolated Riesz clusters. Section~\ref{sec:capacitance_riesz_phases} develops the equilibrium-density projection, proves its comparison with the exact Riesz reduction, computes the regular-simplex windings, and distinguishes the relevant geometric phases.  Section~\ref{sec:num} presents numerical experiments based on spatial discretizations and finite Laurent truncations of the BFZ symbol, records the associated finite-prefractal local winding data, and gives a separate physical far-field diagnostic over one logarithmic scale interval. The final section gives concluding remarks.

\section{Self-similar scattering and the auxiliary scale-periodic boundary model}
\label{sec:selfsimilar_prelim}

This section passes from the physical scaling identities to the auxiliary boundary model used in the BFZ analysis. We first work with the differential-equation and volume-integral formulations, where the Green-function scaling and the recursive far-field pattern are most transparent, and then formulate an exactly scale-periodic boundary model on one reference cell. This auxiliary model is motivated by physical identities but is not equivalent to the fixed-wavenumber penetrable-medium problem.

Throughout this section, fix \(d\ge2\). The self-similar structure is generated by a common contraction ratio
\[
a\in(0,1),\qquad \lambda:=a^{-1}>1.
\]
The parameter \(\lambda\) is the scale ratio that will later define the logarithmic scale cell and the dual variable.

\subsection{Self-similar geometry}

Let \(I=\{1,\dots,M\}\) be a finite index set and let
\[
T_i(x):=a x+b_i,\qquad b_i\in\mathbb R^d,\qquad i\in I.
\]
Let \(\Omega_0\subset\mathbb R^d\) be a bounded Lipschitz domain. We consider a measurable self-similar scatterer \(\Omega_\infty\) satisfying
\begin{equation}
\label{eq:Omega_infty_selfsimilar_prelim}
\Omega_\infty=\Omega_0\cup\bigcup_{i\in I}T_i(\Omega_\infty),
\end{equation}
where the union is disjoint up to sets of Lebesgue measure zero. We assume throughout this section that
\[
|\Omega_\infty|<\infty.
\]
Under the stated disjoint decomposition, Lemma~\ref{lem:OmegaN_decomp_prelim} shows that these hypotheses force \(M a^d<1\) and hence
\[
|\Omega_\infty|=\frac{|\Omega_0|}{1-Ma^d}.
\]

For a word \(\alpha=(\alpha_1,\dots,\alpha_n)\in I^n\), define
\[
T_\alpha:=T_{\alpha_1}\circ\cdots\circ T_{\alpha_n},\qquad |\alpha|:=n.
\]
We also write \(I^0:=\{\emptyset\}\) and \(T_\emptyset:=\mathrm{Id}\). The \(N\)-th level approximation to \(\Omega_\infty\) is
\begin{equation}
\label{eq:OmegaN_prelim}
\Omega_N:=\bigcup_{n=0}^{N}\ \bigcup_{\alpha\in I^n}T_\alpha(\Omega_0).
\end{equation}

The following lemma makes this finite-level decomposition precise and records the volume condition forced by disjoint self-similarity.

\begin{lemma}[Finite-level decomposition]
\label{lem:OmegaN_decomp_prelim}
Assume that \(\Omega_\infty\) satisfies \eqref{eq:Omega_infty_selfsimilar_prelim} up to sets of measure zero, with the union disjoint up to sets of measure zero. Assume also that \(|\Omega_0|>0\) and \(|\Omega_\infty|<\infty\). Then necessarily \(M a^d<1\).
Moreover, for every \(N\ge0\),
\[
\Omega_N=\bigsqcup_{n=0}^{N}\ \bigsqcup_{\alpha\in I^n}T_\alpha(\Omega_0)
\]
up to sets of measure zero, and
\[
\Omega_\infty=\bigcup_{N=0}^{\infty}\Omega_N=\bigsqcup_{n=0}^{\infty}\ \bigsqcup_{\alpha\in I^n}T_\alpha(\Omega_0)
\]
up to sets of measure zero.
\end{lemma}

\begin{proof}
Taking the Lebesgue measure in the self-similar decomposition and using disjointness up to null sets gives
\[
|\Omega_\infty|=|\Omega_0|+\sum_{i\in I}|T_i(\Omega_\infty)|=|\Omega_0|+Ma^d|\Omega_\infty|.
\]
Since \(|\Omega_0|>0\) and \(|\Omega_\infty|<\infty\), it follows that \(Ma^d<1\).

Iterating the self-similar identity gives, for every \(N\ge0\),
\[
\Omega_\infty=\left(\bigsqcup_{n=0}^{N}\ \bigsqcup_{\alpha\in I^n}T_\alpha(\Omega_0)\right)\sqcup\left(\bigsqcup_{\alpha\in I^{N+1}}T_\alpha(\Omega_\infty)\right)
\]
up to sets of measure zero. Hence, the finite-level approximant satisfies
\[
\Omega_N=\bigsqcup_{n=0}^{N}\ \bigsqcup_{\alpha\in I^n}T_\alpha(\Omega_0)
\]
up to sets of measure zero.

It remains only to show that the infinite residual has measure zero. By the preceding decomposition,
\[
\Omega_\infty\setminus\Omega_N\subset \bigsqcup_{\alpha\in I^{N+1}}T_\alpha(\Omega_\infty)
\]
up to a null set. Therefore
\[
|\Omega_\infty\setminus\Omega_N|\le \sum_{\alpha\in I^{N+1}}|T_\alpha(\Omega_\infty)|=(Ma^d)^{N+1}|\Omega_\infty|.
\]
Since \(Ma^d<1\), the right-hand side tends to zero as \(N\to\infty\). Thus
\[
\left|\Omega_\infty\setminus\bigcup_{N=0}^{\infty}\Omega_N\right|=0.
\]
This proves the claimed infinite disjoint decomposition.
\end{proof}

By Lemma~\ref{lem:OmegaN_decomp_prelim}, \(\Omega_\infty\) agrees up to a null set with the union of its finite-level approximants. Accordingly, we henceforth work with the generated representative and retain the notation \(\Omega_\infty\) as 
\[
\Omega_\infty:=\bigcup_{N=0}^{\infty}\Omega_N.
\]
The next proposition shows that the generated representative \(\bigcup_{N=0}^\infty\Omega_N\) is indeed self-similar and that its closures converge in Hausdorff distance.

\begin{proposition}[Generated scatterer and Hausdorff convergence of its closures]
\label{prop:generated_scatterer_hausdorff_hull}
Consider \(\Omega_\infty:=\bigcup_{N=0}^{\infty}\Omega_N.\) Then \(\Omega_\infty\) is open and bounded and satisfies
\[
\Omega_\infty=\Omega_0\cup\bigcup_{i\in I}T_i(\Omega_\infty).
\]
Moreover, with
\[
K_N:=\overline{\Omega_N},\qquad K_\infty:=\overline{\Omega_\infty},
\]
one has
\[
K_\infty=\overline{\Omega_0}\cup\bigcup_{i\in I}T_i(K_\infty),
\qquad
d_H(K_N,K_\infty)\to0,
\]
where \(d_H\) denotes the Hausdorff distance.
\end{proposition}

\begin{proof}
The set \(\Omega_\infty\) is open because it is a union of open sets. Moreover,
\[
\Omega_0\cup\bigcup_{i\in I}T_i(\Omega_\infty)=\Omega_0\cup\bigcup_{i\in I}\bigcup_{n=0}^{\infty}\ \bigcup_{\alpha\in I^n}T_iT_\alpha(\Omega_0).
\]
Since \(T_iT_\alpha=T_{(i,\alpha)}\), the right-hand side is
\[
\Omega_0\cup\bigcup_{n=1}^{\infty}\ \bigcup_{\alpha\in I^n}T_\alpha(\Omega_0)=\bigcup_{n=0}^{\infty}\ \bigcup_{\alpha\in I^n}T_\alpha(\Omega_0)=\Omega_\infty.
\]

It remains to prove compactness and Hausdorff convergence. Set \(K_0:=\overline{\Omega_0}\). Choose \(R>0\) so large that
\[
K_0\subset\overline{B_R(0)},
\qquad
\max_{i\in I}|b_i|\le(1-a)R.
\]
Then \(T_i(\overline{B_R(0)})\subset\overline{B_R(0)}\) for \(i\in I\). Let \(\mathcal K_R\) be the space of nonempty compact subsets of \(\overline{B_R(0)}\), equipped with the Hausdorff distance. Define
\[
F(K):=K_0\cup\bigcup_{i\in I}T_i(K),\qquad K\in\mathcal K_R.
\]
Then \(F:\mathcal K_R\to\mathcal K_R\). For \(K,L\in\mathcal K_R\), set
\[
\delta:=\max_{i\in I}d_H(T_i(K),T_i(L)).
\]
If \(x\in F(K)\), then either \(x\in K_0\subset F(L)\) or \(x\in T_i(K)\) for some \(i\in I\). Hence,
\[
\operatorname{dist}(x,F(L))\le\delta.
\]
Taking the supremum over \(x\in F(K)\), and then repeating the same argument with \(K\) and \(L\) interchanged, gives
\[
d_H(F(K),F(L))\le\delta=\max_{i\in I}d_H(T_i(K),T_i(L))=a\,d_H(K,L),
\]
where the last equality uses the fact that each \(T_i\) is a similarity of ratio \(a\).

Thus, \(F\) is a contraction on the complete metric space \(\mathcal K_R\). By Banach's fixed-point theorem, there exists a unique \(K_\ast\in\mathcal K_R\) such that
\[
K_\ast=K_0\cup\bigcup_{i\in I}T_i(K_\ast).
\]

Indeed, applying \(F\) once adds the zeroth-generation set \(K_0\) and sends \(\alpha\in I^n\) to \((i,\alpha)\in I^{n+1}\), since \(T_iT_\alpha=T_{(i,\alpha)}\). Hence, by induction
\[
F^N(K_0)=\bigcup_{n=0}^{N}\ \bigcup_{\alpha\in I^n}T_\alpha(K_0).
\]
Since \(T_\alpha(K_0)=\overline{T_\alpha(\Omega_0)}\) and the union is finite,
\[
F^N(K_0)=\overline{\bigcup_{n=0}^{N}\ \bigcup_{\alpha\in I^n}T_\alpha(\Omega_0)}=\overline{\Omega_N}=K_N.
\]
The contraction theorem therefore gives \(d_H(K_N,K_\ast)\to0\).

 Next, we identify \(K_\ast\) with \(\overline{\Omega_\infty}\). First, \(K_\ast=F(K_\ast)\) implies \(K_0\subset K_\ast\). If \(F^N(K_0)\subset K_\ast\), then monotonicity gives
\[
F^{N+1}(K_0)=F(F^N(K_0))\subset F(K_\ast)=K_\ast.
\]
Therefore,
\[
K_N=F^N(K_0)\subset K_\ast,\qquad N\ge0.
\]
Hence, \(\overline{\bigcup_{N=0}^{\infty}K_N}\subset K_\ast\).

Conversely, if \(x\in K_\ast\), then
\[
\operatorname{dist}(x,K_N)\le d_H(K_N,K_\ast)\to0.
\]
Thus \(x\in\overline{\bigcup_{N=0}^{\infty}K_N}\), and therefore,
\[
K_\ast=\overline{\bigcup_{N=0}^{\infty}K_N}.
\]
Since \(K_N=\overline{\Omega_N}\), we also have
\[
\overline{\bigcup_{N=0}^{\infty}K_N}=\overline{\bigcup_{N=0}^{\infty}\Omega_N}=\overline{\Omega_\infty}.
\]
Thus \(K_\infty=K_\ast\).
In particular, \(K_\infty\) is compact, \(\Omega_\infty\) is bounded, and
\[
K_\infty=\overline{\Omega_0}\cup\bigcup_{i\in I}T_i(K_\infty),\qquad d_H(K_N,K_\infty)\to0.
\]
\end{proof}

\subsection{Green-function scaling and the volume potential}

We next record the scaling law of the Helmholtz Green function and the resulting covariance of the volume potential.

Let \(m\in L^\infty(\Omega_\infty)\) be the bounded contrast coefficient, extended by zero outside \(\Omega_\infty\) when it appears in integrals over \(\mathbb R^d\). We assume that \(m\) is compatible with the self-similar geometry in the following intrinsic sense: 
\begin{equation} 
\label{eq:m_selfsimilar_prelim} m(T_i x)=m(x)\qquad\text{for a.e. }x\in\Omega_\infty,\quad i\in I.
\end{equation} 
Since \(T_i(\Omega_\infty)\subset\Omega_\infty\), this is the only invariance needed in the following. All identities involving the change of variables \(z=T_i y\) are understood with \(y\) that ranges over subsets of \(\Omega_\infty\).

Let \(G_k\) denote the outgoing Helmholtz Green function, normalized by
\[
(-\Delta-k^2)G_k(\cdot,y)=\delta_y.
\]
Thus,
\begin{equation}
\label{eq:Gk_prelim}
G_k(x,y)=\frac{C_d k^{d-2}i}{(k|x-y|)^{(d-2)/2}}H_{(d-2)/2}^{(1)}(k|x-y|),
\end{equation}
where \(C_d\) is a dimensional constant depending on the chosen normalization.

For every bounded measurable set \(D\subset\mathbb R^d\) and \(\varphi\in L^\infty(D)\), define the volume potential
\begin{equation}
\label{eq:VD_prelim}
(V_D^k\varphi)(x):=k^2\int_D G_k(x,y)m(y)\varphi(y)\,dy.
\end{equation}

The following Green-function identity is the basic scaling relation behind the wavenumber rescaling.

\begin{lemma}[Scaling of the Green function]
\label{lem:G_scaling_prelim}
For every \(x,y\in\mathbb R^d\) away from the corresponding Green-function singularity, every \(i,j\in I\), and every \(k>0\),
\begin{equation}
\label{eq:G_scaling_prelim}
G_k(ax+b_i,ay+b_j)=a^{2-d}G_{ak}\bigl(x,y+a^{-1}(b_j-b_i)\bigr).
\end{equation}
In particular,
\begin{equation}
\label{eq:G_scaling_samecell_prelim}
G_k(ax+b_i,ay+b_i)=a^{2-d}G_{ak}(x,y).
\end{equation}
\end{lemma}

\begin{proof}
We have
\[
ax+b_i-(ay+b_j)=a\bigl(x-y-a^{-1}(b_j-b_i)\bigr).
\]
Hence
\[
|ax+b_i-(ay+b_j)|=a\left|x-\bigl(y+a^{-1}(b_j-b_i)\bigr)\right|.
\]
Substituting this identity into the explicit expression \eqref{eq:Gk_prelim} gives \eqref{eq:G_scaling_prelim}. The special case \(i=j\) gives \eqref{eq:G_scaling_samecell_prelim}.
\end{proof}

For each \(i\in I\), define the pullforward of functions by
\[
(U_i\varphi)(x):=\varphi(T_i^{-1}x),\qquad x\in T_i(D).
\]

The next lemma shows that restriction to one self-similar copy is equivalent, after pullback, to rescaling the wavenumber.

\begin{lemma}[Covariance on one self-similar copy] 
\label{lem:volume_covariance_prelim} 
For every bounded measurable set \(D\subset\Omega_\infty\), 
\begin{equation} \label{eq:volume_covariance_prelim} U_i^{-1}V_{T_i(D)}^kU_i=V_D^{ak},\qquad i\in I. 
\end{equation} 
\end{lemma}

\begin{proof}
Let \(\varphi\in L^\infty(D)\). For \(x\in D\), using the change of variables \(z=ay+b_i\), the contrast invariance \eqref{eq:m_selfsimilar_prelim}, and the same-copy Green scaling \eqref{eq:G_scaling_samecell_prelim}, we obtain
\begin{align*}
&(U_i^{-1}V_{T_i(D)}^kU_i\varphi)(x)=(V_{T_i(D)}^kU_i\varphi)(ax+b_i)=k^2\int_{T_i(D)}G_k(ax+b_i,z)m(z)\varphi(T_i^{-1}z)\,dz\\
=&k^2a^d\int_D G_k(ax+b_i,ay+b_i)m(ay+b_i)\varphi(y)\,dy=k^2a^d a^{2-d}\int_D G_{ak}(x,y)m(y)\varphi(y)\,dy\\
=&(ak)^2\int_D G_{ak}(x,y)m(y)\varphi(y)\,dy=(V_D^{ak}\varphi)(x).
\end{align*}
\end{proof}

For \(\varphi\in L^\infty(\Omega_\infty)\), interpret each term as acting on the restriction of \(\varphi\) to its integration domain. Then the volume potential decomposes pointwise as
\begin{equation}
\label{eq:volume_recursive_prelim}
V_{\Omega_\infty}^k=V_{\Omega_0}^k+\sum_{i\in I}V_{T_i(\Omega_\infty)}^k,
\end{equation}
and each copy term is conjugate to the full potential at the rescaled wavenumber:
\begin{equation}
\label{eq:volume_recursive_copy_prelim}
U_i^{-1}V_{T_i(\Omega_\infty)}^kU_i=V_{\Omega_\infty}^{ak}.
\end{equation}

Equations \eqref{eq:volume_recursive_prelim}-\eqref{eq:volume_recursive_copy_prelim} exhibit the scale recursion that will be organized by the BFZ transform in the logarithmic scale variable.
 
\subsection{The scattering problem and the far-field pattern}

Let \(u^{\mathrm{in}}\) solve
\[
(-\Delta-k^2)u^{\mathrm{in}}=0\qquad\text{in }\mathbb R^d.
\]
Assume that the scattering problem
\begin{equation}
\label{eq:helmholtz_prelim}
\bigl(-\Delta-k^2(1+m\chi_{\Omega_\infty})\bigr)u^k=0\qquad\text{in }\mathbb R^d,
\end{equation}
admits a total field \(u^k\). Set
\[
u^{\mathrm{s},k}:=u^k-u^{\mathrm{in}}.
\]
Assume that \(u^{\mathrm{s},k}\) satisfies the Sommerfeld radiation condition. At the volume-potential level, the total field then satisfies the Lippmann-Schwinger equation
\begin{equation}
\label{eq:LS_prelim}
u^k-V_{\Omega_\infty}^ku^k=u^{\mathrm{in}}.
\end{equation}

For an incident plane wave
\[
u^{\mathrm{in}}(x)=e^{ik\mathbf d\cdot x},\qquad \mathbf d\in\mathbb S^{d-1},
\]
the scattered field has the far-field expansion
\begin{equation}
\label{eq:farfield_asymptotic_prelim}
u^{\mathrm{s},k}(x;\mathbf d)=\gamma_d k^{(d-3)/2}\frac{e^{ik|x|}}{|x|^{(d-1)/2}}A_{\Omega_\infty}^{\infty}(\widehat x,\mathbf d,k)+\mathcal O(|x|^{-(d+1)/2}),\qquad |x|\to\infty,
\end{equation}
where \(\gamma_d\neq0\) is a dimensional constant. The factor \(k^{(d-3)/2}\), which is equal to one in dimension three, follows from the large-argument asymptotics of the Hankel function in \eqref{eq:Gk_prelim}. With the same normalization of the Green function as above,
\begin{equation}
\label{eq:farfield_formula_prelim}
A_{\Omega_\infty}^{\infty}(\widehat x,\mathbf d,k)=c_dk^2\int_{\Omega_\infty}e^{-ik\widehat x\cdot y}m(y)u^k(y;\mathbf d)\,dy.
\end{equation}

The next proposition separates the far-field pattern into scale contributions.

\begin{proposition}[Exact series expansion of the far-field pattern]
\label{prop:farfield_series_prelim}
Assume, in addition, that
\[
u^k(\cdot;\mathbf d)\in L^\infty(\Omega_\infty).
\]
Then the far-field pattern admits the absolutely convergent representation
\begin{equation}
\label{eq:farfield_series_prelim}
A_{\Omega_\infty}^{\infty}(\widehat x,\mathbf d,k)=c_dk^2\sum_{n=0}^{\infty}\ \sum_{\alpha\in I^n}\int_{T_\alpha(\Omega_0)}e^{-ik\widehat x\cdot y}m(y)u^k(y;\mathbf d)\,dy.
\end{equation}
\end{proposition}

\begin{proof}
By Lemma~\ref{lem:OmegaN_decomp_prelim}, the family \(\{ T_\alpha(\Omega_0): \alpha\in I^n, n\geq 0\}\) decomposes \(\Omega_\infty\) up to a set of measure zero. Inserting this decomposition into \eqref{eq:farfield_formula_prelim} gives \eqref{eq:farfield_series_prelim} formally. Moreover,
\begin{align*}
\sum_{n=0}^{\infty}\sum_{\alpha\in I^n}\left|\int_{T_\alpha(\Omega_0)}e^{-ik\widehat x\cdot y}m(y)u^k(y;\mathbf d)\,dy\right|&\le \|m\|_{L^\infty(\Omega_\infty)}\|u^k(\cdot;\mathbf d)\|_{L^\infty(\Omega_\infty)}\sum_{n=0}^{\infty}\sum_{\alpha\in I^n}|T_\alpha(\Omega_0)|\\
&=\|m\|_{L^\infty(\Omega_\infty)}\|u^k(\cdot;\mathbf d)\|_{L^\infty(\Omega_\infty)}|\Omega_\infty|<\infty.
\end{align*}
Thus, the series is absolutely convergent, and the far-field integral may be decomposed term by term.
\end{proof}

The absolute convergence in Proposition~\ref{prop:farfield_series_prelim} uses only \(m\in L^\infty(\Omega_\infty)\), \(u^k(\cdot;\mathbf d)\in L^\infty(\Omega_\infty)\), and \(|\Omega_\infty|<\infty\). No smoothness of the limiting self-similar boundary is needed at this volume-potential level.

To obtain an explicit closed scale recursion, it is useful to pass to the Born approximation, in which the total field inside the scatterer is replaced by the incident field.

\begin{definition}[First Born approximation of the far-field pattern]
\label{def:born_farfield_prelim}
For a bounded measurable set \(\Omega\), define
\begin{equation}
\label{eq:born_farfield_prelim}
A_{\Omega}^{\infty,\mathrm B}(\widehat x,\mathbf d,k):=c_dk^2\int_{\Omega}e^{ik(\mathbf d-\widehat x)\cdot y}m(y)\,dy.
\end{equation}
\end{definition}

The next proposition gives the resulting scale recursion.

\begin{proposition}[Recursive identity for the first Born far-field approximation]
\label{prop:born_recursive_prelim}
The Born far-field pattern of the self-similar scatterer satisfies
\begin{equation}
\label{eq:born_recursive_prelim}
A_{\Omega_\infty}^{\infty,\mathrm B}(\widehat x,\mathbf d,k)=A_{\Omega_0}^{\infty,\mathrm B}(\widehat x,\mathbf d,k)+a^{d-2}\sum_{i\in I}e^{ik(\mathbf d-\widehat x)\cdot b_i}A_{\Omega_\infty}^{\infty,\mathrm B}(\widehat x,\mathbf d,ak).
\end{equation}
\end{proposition}

\begin{proof}
Using the decomposition \eqref{eq:Omega_infty_selfsimilar_prelim}, we write
\[
A_{\Omega_\infty}^{\infty,\mathrm B}(\widehat x,\mathbf d,k)=c_dk^2\int_{\Omega_0}e^{ik(\mathbf d-\widehat x)\cdot y}m(y)\,dy+c_dk^2\sum_{i\in I}\int_{T_i(\Omega_\infty)}e^{ik(\mathbf d-\widehat x)\cdot y}m(y)\,dy.
\]
The first term is \(A_{\Omega_0}^{\infty,\mathrm B}(\widehat x,\mathbf d,k)\). For the \(i\)-th copy, set \(y=az+b_i\). Then \(dy=a^d dz\), and by the contrast invariance \eqref{eq:m_selfsimilar_prelim},
\begin{align*}
c_dk^2\int_{T_i(\Omega_\infty)}e^{ik(\mathbf d-\widehat x)\cdot y}m(y)\,dy&=c_dk^2a^d e^{ik(\mathbf d-\widehat x)\cdot b_i}\int_{\Omega_\infty}e^{i(ak)(\mathbf d-\widehat x)\cdot z}m(z)\,dz\\
&=a^{d-2}e^{ik(\mathbf d-\widehat x)\cdot b_i}A_{\Omega_\infty}^{\infty,\mathrm B}(\widehat x,\mathbf d,ak).
\end{align*}
Summing over \(i\in I\) proves \eqref{eq:born_recursive_prelim}.
\end{proof}

We now pass from the physical scaling identities to the auxiliary boundary setting. The preceding recursion is governed by dilation rather than translation, so the BFZ transform used below is the Fourier transform in the discrete logarithmic scale index. A fixed-wavenumber Helmholtz operator is not invariant under dilation, because spatial scaling also rescales the wavenumber. Accordingly, the construction defines a scale-normalized boundary model; it does not claim to diagonalize the preceding penetrable-medium problem.

From this point onward, we work in dimension \(d=3\), and \(\lambda>1\) denotes the dilation ratio. The continuous Mellin transform is the natural transform for dilations:
\[
(\mathcal Mf)(\xi)=\int_0^\infty f(r)r^{-i\xi}\,\frac{dr}{r}=\int_{\mathbb R}f(e^\tau)e^{-i\xi\tau}\,d\tau,\qquad \tau=\log r.
\]
Thus, it is the Fourier transform in the logarithmic radial variable. In the exact scale-periodic model, the scale factors \(\lambda^m\) correspond to the lattice points \(m\log\lambda\) in that variable. Replacing the continuous logarithmic variable by the scale index gives the discrete Mellin-Fourier series
\[
\sum_{m\in\mathbb Z}u_m e^{-im\theta\log\lambda},
\qquad
\theta\in\mathbb R\Big/\left(\frac{2\pi}{\log\lambda}\mathbb Z\right),
\]
which is the BFZ transform defined below in \eqref{eq:bfz_transform_def}. This converts a shift between adjacent scale layers into multiplication by a phase and converts a convolution in the scale index into an operator-valued BFZ symbol. The boundary formulation introduced next allows every dilated layer to be pulled back to one reference cell and represented in this discrete scale variable.

\subsection{Reference boundary in a fundamental scale cell}
\label{subsec:scale_periodic_setting}

We assume that, after a translation by a fixed vector \(p\in\mathbb R^3\), the self-similar medium admits an exactly scale-periodic extension. 
Assume \((\Omega_\infty + p)\subsetneq\lambda (\Omega_\infty + p)\) and define
\begin{equation}
\label{eq:scale_periodic_extension}
\Omega_\infty^\infty:=\bigcup_{m\in\mathbb Z}\lambda^m (\Omega_\infty + p).
\end{equation}
Then
\[
\lambda\Omega_\infty^\infty=\Omega_\infty^\infty.
\]

The associated volume shell is
\[
\Omega_0^0:=\lambda (\Omega_\infty + p)\setminus (\Omega_\infty + p).
\]
This set does not need to be open.
The passage from the bounded set \(\Omega_\infty\) to \(\Omega_\infty^\infty\) is an idealization analogous to the replacement of a finite crystal by an infinite periodic medium. This motivates the logarithmic scale coordinate but is not used as the boundary of the operator model.

\begin{figure}[H]
    \centering
    \includegraphics[width=0.7\linewidth]{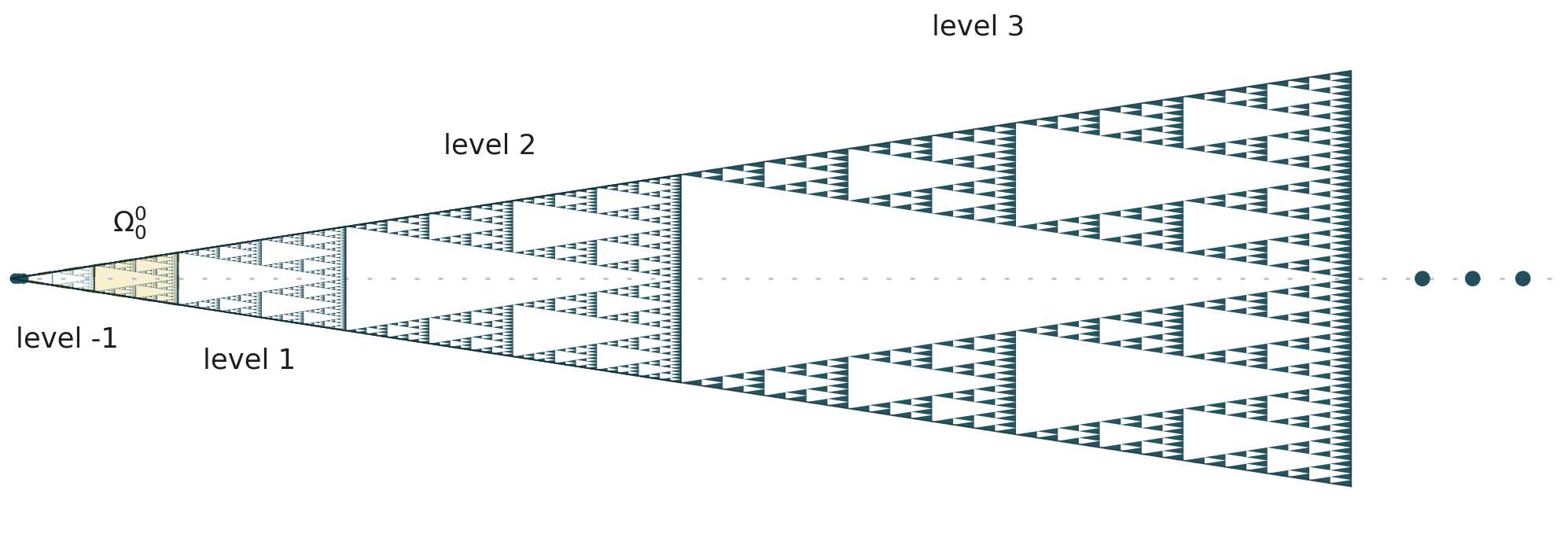}
    \caption{Schematic volume-scale extension \(\Omega_{\infty}^{\infty}\). The reference boundary \(\Gamma_0\) is selected independently inside one annular scale cell.}
    \label{fig:omega_inf_inf}
\end{figure}

We assume that \(\operatorname{int}\Omega_0^0\) contains a bounded \(C^\infty\) reference obstacle \(D_0\Subset\operatorname{int}\Omega_0^0\), possibly with finitely many connected components, and set \(\Gamma_0=\partial D_0\). Its scale-periodic extension is defined independently by
\begin{equation}
\label{eq:boundary_scale_partition}
\Gamma^\infty:=\bigsqcup_{m\in\mathbb Z}\Gamma_m,
\qquad
\Gamma_m:=\lambda^m\Gamma_0.
\end{equation}
In particular, \(\Gamma^\infty\) is not identified with \(\partial\Omega_\infty^\infty\); the interfaces of the nested volume shells are not part of the boundary-integral model.

To keep distinct scale layers uniformly separated, we impose the following geometric condition.

\begin{assumption}[Annular scale separation]
\label{ass:annular_separation}
There exist radii \(0<r_-<r_+\) such that
\[
\Gamma_0\subset\{x\in\mathbb R^3:r_-\le |x|\le r_+\},\qquad r_+<\lambda r_-.
\]
\end{assumption}
This condition is imposed on the smooth scale-normalized boundary model and gives the positive separation needed for the off-diagonal smoothing estimates. The next lemma makes the separation uniform over all nonzero relative scale indices.

\begin{lemma}[Uniform separation of distinct scale boundaries]
\label{lem:scale_separation}
Assume Assumption~\ref{ass:annular_separation}. Then, for every \(n\in\mathbb Z\setminus\{0\}\),
\[
\operatorname{dist}(\Gamma_0,\lambda^n\Gamma_0)\ge \delta_\ast,\qquad \text{where} \qquad \delta_\ast:=r_- - \lambda^{-1}r_+>0.
\]
\end{lemma}

\begin{proof}
If \(n\ge1\), then
\[
\lambda^n\Gamma_0\subset\{|x|\ge \lambda^n r_-\},\qquad \Gamma_0\subset\{|x|\le r_+\}.
\]
Thus
\[
\operatorname{dist}(\Gamma_0,\lambda^n\Gamma_0)\ge \lambda^n r_- - r_+\ge \lambda r_- - r_+.
\]
Since
\[
\lambda r_- - r_+=\lambda(r_- - \lambda^{-1}r_+)\ge r_- - \lambda^{-1}r_+,
\]
this gives the desired bound.

If \(n\le -1\), then
\[
\lambda^n\Gamma_0\subset\{|x|\le \lambda^n r_+\}\subset\{|x|\le \lambda^{-1}r_+\},\qquad\Gamma_0\subset\{|x|\ge r_-\}.
\]
Hence
\[
\operatorname{dist}(\Gamma_0,\lambda^n\Gamma_0)\ge r_- - \lambda^{-1}r_+=\delta_\ast.
\]
\end{proof}

The boundary-integral analysis is carried out on these smooth reference boundaries. We do not define layer-potential operators directly on a rough limiting fractal boundary. The limiting self-similar set enters through its iterative system, scale ratio, and smooth finite-depth approximations.

\subsection{Boundary operator on one scale cell}
\label{subsec:boundary_formulation_one_scale_cell}

Let \(\Gamma\) be a finite disjoint union of closed \(C^\infty\) surfaces, equipped with the outward unit normal. Define the boundary single-layer, double-layer operators, and the adjoint Neumann-Poincar\'e operator by
\[
(S_k^\Gamma\mu)(x):=\int_\Gamma G_k(x,y)\mu(y)\,d\sigma(y),
\]
\[
(K_k^\Gamma\mu)(x):=\operatorname{p.v.}\int_\Gamma \partial_{\nu_y}G_k(x,y)\mu(y)\,d\sigma(y),
\]
and
\[
((K_k^\Gamma)^\ast\mu)(x):=\operatorname{p.v.}\int_\Gamma \partial_{\nu_x}G_k(x,y)\mu(y)\,d\sigma(y).
\]
Here, \(\operatorname{p.v.}\) denotes the Cauchy principal value. Below, we use \({}^\dagger\) to denote Hilbert adjoints. For a coupling parameter \(\eta>0\), define the adjoint combined-field integral operator
\begin{equation}
\label{eq:combined_field_operator}
B_0^{k,\Gamma}:=\frac12I+(K_k^\Gamma)^\ast-i\eta k S_k^\Gamma.
\end{equation}
With the outward normal, the plus sign is the interior normal-trace sign. Thus, \(B_0^{k,\Gamma}\) is not identified here with the single-layer equation for an exterior impedance condition. It is the reference-cell block of the scale-normalized combined-field model used below, written in a form whose static part retains the equilibrium-density kernel described in Subsection~\ref{subsec:calderon_capacitance_pairing}. On smooth surfaces, these layer-potential operators are bounded on \(H^{1/2}(\Gamma)\). We use this realization because it contains the equilibrium densities and ranges of the Riesz projections considered below; no global invertibility of \(B_0^{k,\Gamma}\) is required.

We next define a wavenumber-weighted \(H^{1/2}\) space on \(\Gamma_0\), in which the tangential frequencies and the Helmholtz wavenumber are measured on the same scale.

Let \(-\Delta_{\Gamma_0}\) be the nonnegative Laplace-Beltrami operator. For \(k>0\), define
\begin{equation}
\label{eq:Hk}
\|\varphi\|_{\mathcal H_k}:=\|(k^2-\Delta_{\Gamma_0})^{1/4}\varphi\|_{L^2(\Gamma_0)},\qquad \mathcal H_k:=\left\{\varphi\in L^2(\Gamma_0):\|\varphi\|_{\mathcal H_k}<\infty\right\}.
\end{equation}
Equivalently, if
\[
-\Delta_{\Gamma_0}\phi_j=\mu_j\phi_j,\qquad \mu_j\ge0,
\qquad
\varphi=\sum_j\varphi_j\phi_j,
\]
then
\[
\|\varphi\|_{\mathcal H_k}^2=\sum_j(k^2+\mu_j)^{1/2}|\varphi_j|^2.
\]
Thus, \(\mathcal H_k\) is a wavenumber-weighted \(H^{1/2}\) boundary space: it weights low tangential frequencies by \(k^{1/2}\) and recovers the usual \(H^{1/2}\) weight at high tangential frequencies.

The following comparison transfers standard \(L^2\)- and \(H^{1/2}\)-mapping bounds to \(\mathcal H_k\).

\begin{lemma}[Comparison with \(L^2\) and \(H^{1/2}\)]
\label{lem:Hk_comparison}
For \(k\ge1\),
\[
k^{1/2}\|\varphi\|_{L^2(\Gamma_0)}\le \|\varphi\|_{\mathcal H_k}.
\]
Moreover, there is a constant \(C\), independent of \(k\ge1\), such that
\[
\|\varphi\|_{\mathcal H_k}\le C\left(k^{1/2}\|\varphi\|_{L^2(\Gamma_0)}+\|\varphi\|_{H^{1/2}(\Gamma_0)}\right).
\]
Consequently, if \(A\) is bounded on both \(L^2(\Gamma_0)\) and \(H^{1/2}(\Gamma_0)\), then
\[
\|A\|_{\mathcal L(\mathcal H_k)}\le C\max\left\{\|A\|_{L^2\to L^2},\|A\|_{H^{1/2}\to H^{1/2}}\right\}.
\]
\end{lemma}

\begin{proof}
The first two estimates follow directly from
\[
(k^2+\mu_j)^{1/2}\ge k, \qquad
(k^2+\mu_j)^{1/2}\le C\left(k+(1+\mu_j)^{1/2}\right), \qquad k\ge1.
\]
For the operator estimate, use the equivalent bound
\[
\|\varphi\|_{\mathcal H_k}^2\asymp k\|\varphi\|_{L^2(\Gamma_0)}^2+\|\varphi\|_{H^{1/2}(\Gamma_0)}^2,
\]
with constants independent of \(k\ge1\), and apply the assumed \(L^2\) and \(H^{1/2}\) bounds to \(A\varphi\).
\end{proof}

The space \(\mathcal H_k\) is used after each scale boundary is normalized to \(\Gamma_0\).

\section{Scale-periodic BFZ symbols and determinant winding}
\label{sec:high_freq_gap}

With the reference-cell boundary operator fixed, we now construct the Laurent operator in the scale index, fiberize it by the BFZ transform, and define determinant winding on finite-dimensional invariant spectral subspaces.

\subsection{Block Laurent structure and weighted spaces}
\label{subsec:blocks_Laurent}

We now formulate the boundary-integral operator on the reference cell in the scale index. We use the following construction in the subsequent BFZ analysis. 

On the reference boundary, write
\(B_0^k:=B_0^{k,\Gamma_0}\) for the reference-cell block in
\eqref{eq:combined_field_operator}; since \(\Gamma_0\) is smooth, the standard mapping properties of layer-potential operators give
\(B_0^k\in\mathcal L(\mathcal H_k)\).

To write the interscale combined-field interactions compactly, set
\[
\mathsf K_\eta^k(x,z):=\partial_{\nu_x}G_k(x,z)-i\eta kG_k(x,z).
\]
We use equilibrium-density normalization under dilation: a reference density \(\varphi\) on \(\Gamma_0\) represents a density \(\mu_n\) on \(\Gamma_n\) through
\[
\mu_n(\lambda^ny):=\lambda^{-n}\varphi(y),\qquad y\in\Gamma_0.
\]
Since \(d\sigma(\lambda^ny)=\lambda^{2n}d\sigma(y)\), pulling the source integral back to \(\Gamma_0\) gives
\[
\int_{\Gamma_n}\mathsf K_\eta^k(x,z)\mu_n(z)\,d\sigma(z)=\lambda^n\int_{\Gamma_0}\mathsf K_\eta^k(x,\lambda^ny)\varphi(y)\,d\sigma(y).
\]
Consequently, for \(n\in\mathbb Z\setminus\{0\}\), define the scale-normalized interscale block \(B_n^k:\mathcal H_k\to\mathcal H_k\) by
\begin{equation}
\label{eq:Bn_def_weighted_sec}
(B_n^k\varphi)(x):=\lambda^n\int_{\Gamma_0}\mathsf K_\eta^k(x,\lambda^ny)\varphi(y)\,d\sigma(y),\qquad x\in\Gamma_0.
\end{equation}
We reverse the boundary sequence index: \(\varphi_m\) represents the physical density on \(\Gamma_{-m}=\lambda^{-m}\Gamma_0\). After normalizing the target layer \(\Gamma_{-\ell}\) to \(\Gamma_0\), the source layer \(\Gamma_{-m}\) becomes the relative copy \(\Gamma_{\ell-m}\). Hence, the formal Laurent action is
\begin{equation}
\label{eq:Tk_finitely_supported_def}
\boldsymbol\varphi\longmapsto \left(\sum_{m\in\mathbb Z}B_{\ell-m}^k\varphi_m\right)_{\ell\in\mathbb Z}.
\end{equation}
With the Fourier convention \(e^{-imt}\), this reverse indexing produces the symbol \(\sum_ne^{-int}B_n^k\) and fixes the orientation of all the winding numbers below.
At this stage, this expression is only formal, because the coefficients \(\{B_n^k\}_{n\in\mathbb Z}\) may not be absolutely summable in the positive-scale direction. The actual scale operator used below is therefore defined after introducing the weight \(\lambda^{\beta m}\), for which the weighted Laurent coefficients become summable.
 
The following bounds justify the weight in the Laurent series.

\begin{proposition}[One-sided bounds for interscale blocks]
\label{prop:one_sided_Bn_bounds}
There exists \(C_k<\infty\), depending on $k$, $\eta$, $\lambda$, and $\Gamma_0$ such that
\begin{equation}
\|B_n^k\|_{\mathcal L(\mathcal H_k)}\le C_k,\qquad n\ge1, \qquad \|B_{-m}^k\|_{\mathcal L(\mathcal H_k)}\le C_k\lambda^{-m},\qquad m\ge1.
\end{equation}
Moreover, \(B_n^k:\mathcal H_k\to\mathcal H_k\) is compact for every \(n\neq0\).
\end{proposition}

\begin{proof}
We first prove estimates at the level of kernels. Since \(k\) is fixed and the Helmholtz kernel is smooth away from the diagonal, its tangential \(x\)-derivatives in local surface charts, including derivatives of \(\nu_x\), satisfy
\[
|\nabla_{\Gamma_0,x}^j\mathsf K_\eta^k(x,z)|\le C_{k,c}(1+|x-z|)^{-1},\qquad 0\le j\le1,
\]
whenever \(|x-z|\ge c\).

Let \(n\ge1\) and set \(R:=\lambda^n\). Since \(\Gamma_0\subset\{|x|\le r_+\}\) and \(R\Gamma_0\subset\{|x|\ge Rr_-\}\), we have
\[
|x-Ry|\ge cR,\qquad x,y\in\Gamma_0,
\]
with \(c>0\) independent of \(n\). Therefore
\[
\sup_{x,y\in\Gamma_0}\left(|R\mathsf K_\eta^k(x,Ry)|+|\nabla_{\Gamma_0,x}(R\mathsf K_\eta^k(x,Ry))|\right)\le C_k.
\]
Schur's test and differentiation in the \(x\)-variable imply
\[
\|B_n^k\|_{L^2\to L^2}+\|B_n^k\|_{L^2\to H^1}\le C_k,\qquad n\ge1.
\]
Since \(H^{1/2}(\Gamma_0)\hookrightarrow L^2(\Gamma_0)\) and \(H^1(\Gamma_0)\hookrightarrow H^{1/2}(\Gamma_0)\), this also gives
\[
\|B_n^k\|_{H^{1/2}\to H^{1/2}}\le C_k,\qquad n\ge1.
\]
For \(k\ge1\), Lemma~\ref{lem:Hk_comparison} yields
\[
\|B_n^k\|_{\mathcal L(\mathcal H_k)}\le C_k,\qquad n\ge1.
\]

Now, let \(m\ge1\) and set \(R:=\lambda^m\). Then
\[
B_{-m}^k\varphi(x)=R^{-1}\int_{\Gamma_0}\mathsf K_\eta^k(x,R^{-1}y)\varphi(y)\,d\sigma(y).
\]
By annular separation in Assumption~\ref{ass:annular_separation},
\[
|x-R^{-1}y|\ge \delta_\ast,\qquad x,y\in\Gamma_0.
\]
Therefore
\[
\sup_{x,y\in\Gamma_0}\left(|R^{-1}\mathsf K_\eta^k(x,R^{-1}y)|+|\nabla_{\Gamma_0,x}(R^{-1}\mathsf K_\eta^k(x,R^{-1}y))|\right)\le C_kR^{-1}.
\]
The same Schur and differentiation argument gives
\[
\|B_{-m}^k\|_{L^2\to L^2}+\|B_{-m}^k\|_{L^2\to H^1}\le C_k\lambda^{-m}.
\]
Consequently,
\[
\|B_{-m}^k\|_{H^{1/2}\to H^{1/2}}\le C_k\lambda^{-m},
\]
and, for \(k\ge1\), Lemma~\ref{lem:Hk_comparison} gives
\[
\|B_{-m}^k\|_{\mathcal L(\mathcal H_k)}\le C_k\lambda^{-m}.
\]
For fixed \(0<k<1\), the norms of \(\mathcal H_k\) and \(\mathcal H_1\) are equivalent, with constants depending on \(k\). The same \(L^2\)- and \(H^1\)-bounds therefore give both estimates on \(\mathcal H_k\).

It remains to prove compactness. For fixed \(n\neq0\), the kernel of \(B_n^k\) is \(C^\infty\) on \(\Gamma_0\times\Gamma_0\), because \(\Gamma_0\) and \(\lambda^n\Gamma_0\) are separated. Hence
\[
B_n^k:L^2(\Gamma_0)\to H^1(\Gamma_0)
\]
is bounded. Since \(\mathcal H_k\hookrightarrow L^2(\Gamma_0)\) continuously and \(H^1(\Gamma_0)\hookrightarrow \mathcal H_k\) compactly for fixed \(k\), the map
\[
\mathcal H_k\xrightarrow{B_n^k}H^1(\Gamma_0)\hookrightarrow\mathcal H_k
\]
is compact. Thus, \(B_n^k\) is compact on \(\mathcal H_k\).
\end{proof}
 
The asymmetric estimates of Proposition~\ref{prop:one_sided_Bn_bounds} do not ensure absolute summability of the unweighted off-diagonal Laurent coefficients in the positive-scale direction. We therefore introduce an exponential scale weight adapted to these estimates.

Fix
\[
\beta\in(-1,0).
\]
Define the weighted scale space
\begin{equation}
\mathcal H_{k,\beta}^{\mathrm{sc}}:=\left\{\boldsymbol\varphi=(\varphi_m)_{m\in\mathbb Z}:\sum_{m\in\mathbb Z}\lambda^{2\beta m}\|\varphi_m\|_{\mathcal H_k}^2<\infty\right\},
\end{equation}
where \({\mathrm{sc}}\) denotes the scale-sequence space. Its norm is
\[
\|\boldsymbol\varphi\|_{\mathcal H_{k,\beta}^{\mathrm{sc}}}^2:=\sum_{m\in\mathbb Z}\lambda^{2\beta m}\|\varphi_m\|_{\mathcal H_k}^2.
\]
Let
\[
\mathcal H_k^{\mathrm{sc}}:=\ell^2(\mathbb Z;\mathcal H_k),
\]
and define
\[
\mathcal W_\beta:\mathcal H_{k,\beta}^{\mathrm{sc}}\to\mathcal H_k^{\mathrm{sc}},\qquad(\mathcal W_\beta\boldsymbol\varphi)_m:=\lambda^{\beta m}\varphi_m.
\]
Then \(\mathcal W_\beta\) is an isometric isomorphism.

The next lemma shows that \(-1<\beta<0\) compensates for the one-sided block bounds and yields a bounded Laurent operator.

\begin{lemma}[Weighted Laurent series]
\label{lem:weighted_laurent_realization}
For every \(k>0\) and every \(\beta\in(-1,0)\),
\begin{equation}
\label{eq:weighted_summability}
\sum_{n\in\mathbb Z\setminus\{0\}}(1+|n|)\lambda^{\beta n}\|B_n^k\|_{\mathcal L(\mathcal H_k)}<\infty.
\end{equation}
Consequently, the Laurent operator
\begin{equation}
(\widetilde{\mathcal T}_{k,\beta}\boldsymbol\psi)_\ell:=\sum_{m\in\mathbb Z}\lambda^{\beta(\ell-m)}B_{\ell-m}^k\psi_m
\end{equation}
is a bounded operator on \(\mathcal H_k^{\mathrm{sc}}\), with
\[
\|\widetilde{\mathcal T}_{k,\beta}\|_{\mathcal L(\mathcal H_k^{\mathrm{sc}})} \le \sum_{n\in\mathbb Z}\lambda^{\beta n}\|B_n^k\|_{\mathcal L(\mathcal H_k)}.
\]
The corresponding operator on the weighted scale space is
\[
\mathcal T_k^{(\beta)}:=\mathcal W_\beta^{-1}\widetilde{\mathcal T}_{k,\beta}\mathcal W_\beta \qquad\text{on }\mathcal H_{k,\beta}^{\mathrm{sc}}.
\]
On finitely supported sequences, \(\mathcal T_k^{(\beta)}\) agrees with the formal convolution
\[
(\mathcal T_k^{(\beta)}\boldsymbol\varphi)_\ell=\sum_{m\in\mathbb Z}B_{\ell-m}^k\varphi_m.
\]
\end{lemma}

\begin{proof}
For \(n\ge1\), Proposition~\ref{prop:one_sided_Bn_bounds} gives
\[
\lambda^{\beta n}\|B_n^k\|\le C_k\lambda^{\beta n}.
\]
Since \(\beta<0\),
\[
\sum_{n\ge1}(1+n)\lambda^{\beta n}<\infty.
\]
For \(m\ge1\), the same proposition gives
\[
\lambda^{-\beta m}\|B_{-m}^k\|\le C_k\lambda^{-(1+\beta)m}.
\]
Since \(1+\beta>0\),
\[
\sum_{m\ge1}(1+m)\lambda^{-(1+\beta)m}<\infty.
\]
Thus, \eqref{eq:weighted_summability} holds. Including the \(n=0\) term gives
\[
\sum_{n\in\mathbb Z}\lambda^{\beta n}\|B_n^k\|_{\mathcal L(\mathcal H_k)}<\infty.
\]

Set
\[
a_n:=\lambda^{\beta n}\|B_n^k\|_{\mathcal L(\mathcal H_k)}, \qquad b_m:=\|\psi_m\|_{\mathcal H_k}.
\]
For every \(\ell\in\mathbb Z\),
\[
\|(\widetilde{\mathcal T}_{k,\beta}\boldsymbol\psi)_\ell\|_{\mathcal H_k}\le\sum_{n\in\mathbb Z}a_nb_{\ell-n}.
\]
Young's \(\ell^1*\ell^2\) inequality, therefore, gives
\[
\|\widetilde{\mathcal T}_{k,\beta}\boldsymbol\psi\|_{\mathcal H_k^{\mathrm{sc}}}\le\|a*b\|_{\ell^2} \le \left(\sum_{n\in\mathbb Z}\lambda^{\beta n}\|B_n^k\|_{\mathcal L(\mathcal H_k)}\right)\|\boldsymbol\psi\|_{\mathcal H_k^{\mathrm{sc}}}.
\]
Density extends \(\widetilde{\mathcal T}_{k,\beta}\) to a bounded operator on \(\mathcal H_k^{\mathrm{sc}}\). The definition of \(\mathcal T_k^{(\beta)}\) follows by conjugation with the isometric isomorphism \(\mathcal W_\beta\). Finally, if \(\boldsymbol\varphi\) is finitely supported and \(\boldsymbol\psi=\mathcal W_\beta\boldsymbol\varphi\), then
\[
(\widetilde{\mathcal T}_{k,\beta}\boldsymbol\psi)_\ell=\sum_m \lambda^{\beta(\ell-m)}B_{\ell-m}^k\lambda^{\beta m}\varphi_m=\lambda^{\beta\ell}\sum_m B_{\ell-m}^k\varphi_m.
\]
Applying \(\mathcal W_\beta^{-1}\) gives the formal Laurent series.
\end{proof}

The BFZ parameter lies on the circle dual to the integer scale index. Set
\[
Y^\ast:=\mathbb R\Big/\left(\frac{2\pi}{\log\lambda}\mathbb Z\right),
\]
and identify this compact circle with its half-open representative \([0,2\pi/\log\lambda)\) when writing integrals. The BFZ transform on the reverse-indexed boundary sequence is
\begin{equation}
\label{eq:bfz_transform_def}
(\mathcal F_\lambda \mathbf u)(\theta):=\sum_{m\in\mathbb Z}u_m e^{-im\theta\log\lambda}, \qquad \theta\in Y^\ast.
\end{equation}
It extends to a unitary map
\[
\mathcal F_\lambda:\ell^2(\mathbb Z;X)\to L^2\left(Y^\ast;X,\frac{\log\lambda}{2\pi}d\theta\right)
\]
for every Hilbert space \(X\).

The next theorem diagonalizes the scale convolution.

\begin{theorem}[BFZ fiberization of the weighted boundary operator]
\label{thm:weighted_bfz_fiberization}
The operator \(\widetilde{\mathcal T}_{k,\beta}\) is fiberized by \(\mathcal F_\lambda\):
\[
\mathcal F_\lambda\widetilde{\mathcal T}_{k,\beta}\mathcal F_\lambda^{-1}=\int_{Y^\ast}^{\oplus}\widetilde{\mathcal T}_{k,\beta}(\theta)\,\frac{\log\lambda}{2\pi}d\theta,
\]
where
\begin{equation}
\label{eq:weighted_symbol_def}
\widetilde{\mathcal T}_{k,\beta}(\theta) = \sum_{n\in\mathbb Z}e^{-in\theta\log\lambda}\lambda^{\beta n}B_n^k.
\end{equation}
The series in \eqref{eq:weighted_symbol_def} converges absolutely in the operator norm, uniformly in \(\theta\), and the map \(\theta\mapsto\widetilde{\mathcal T}_{k,\beta}(\theta)\) is \(C^1\) in the operator norm. Moreover,
\[
\sigma_{\mathcal H_{k,\beta}^{\mathrm{sc}}}\bigl(\mathcal T_k^{(\beta)}\bigr)=\sigma_{\mathcal H_k^{\mathrm{sc}}}\bigl(\widetilde{\mathcal T}_{k,\beta}\bigr)=\overline{\bigcup_{\theta\in Y^\ast}\sigma_{\mathcal H_k}\bigl(\widetilde{\mathcal T}_{k,\beta}(\theta)\bigr)}.
\]
\end{theorem}

\begin{proof}
The absolute and uniform norm convergence of \eqref{eq:weighted_symbol_def} follows from Lemma~\ref{lem:weighted_laurent_realization}. The same lemma also gives
\[
\sum_{n\in\mathbb Z}|n|\lambda^{\beta n}\|B_n^k\|_{\mathcal L(\mathcal H_k)}<\infty,
\]
so termwise differentiation is justified and the symbol is \(C^1\).

Indeed,
\begin{align*}
&(\mathcal F_\lambda\widetilde{\mathcal T}_{k,\beta}\mathbf u)(\theta) =\sum_{\ell\in\mathbb Z}\sum_{m\in\mathbb Z} \lambda^{\beta(\ell-m)}B_{\ell-m}^k u_m e^{-i\ell\theta\log\lambda}\\
=&\sum_{n\in\mathbb Z}\lambda^{\beta n}B_n^k e^{-in\theta\log\lambda} \sum_{m\in\mathbb Z}u_m e^{-im\theta\log\lambda}=\widetilde{\mathcal T}_{k,\beta}(\theta)(\mathcal F_\lambda\mathbf u)(\theta).
\end{align*}
By density and boundedness, the identity extends to all of \(\mathcal H_k^{\mathrm{sc}}\). Hence,
\[
\mathcal F_\lambda\widetilde{\mathcal T}_{k,\beta}\mathcal F_\lambda^{-1}=\int_{Y^\ast}^{\oplus}\widetilde{\mathcal T}_{k,\beta}(\theta)\,\frac{\log\lambda}{2\pi}d\theta.
\]
The spectrum formula follows from the standard spectrum formula for decomposable operators. Since the symbol is norm-continuous on the compact set \(Y^\ast\), the essential closure of the fiber spectra equals the closure of the union of the fiber spectra. Finally, the equality of spectra for \(\mathcal T_k^{(\beta)}\) and \(\widetilde{\mathcal T}_{k,\beta}\) follows from the similarity
\[
\mathcal T_k^{(\beta)}=\mathcal W_\beta^{-1}\widetilde{\mathcal T}_{k,\beta}\mathcal W_\beta.
\]
\end{proof}

The weight \(\lambda^{\beta m}\) is part of the bounded realization on \(\mathcal H_{k,\beta}^{\mathrm{sc}}\) and evaluates the symbol at the complexified BFZ parameter \(\theta+i\beta\). 

The periodicity of the BFZ symbol, with period \(2\pi/\log\lambda\), is structurally analogous to the log-periodic factors in tube formulas for self-similar sets \cite{lapidus_radunovic_zubrinic2017distance,lapidus_radunovic_zubrinic2018tube}, which are multiplicatively periodic in scale: both reflect an underlying discrete dilation group. A closer connection can be explored if the scattering blocks \(B_n^k\) are replaced by scale-normalized Laplace-type blocks and suitable heat or resolvent traces are studied. Spectral zeta functions of fractal Laplacians then describe log-periodic spectral asymptotics, whereas distance and tube zeta functions recover the real tube volume
\[
V_F(\varepsilon):=\left|\left\{x:\operatorname{dist}(x,F)<\varepsilon\right\}\right|
\]
and its oscillatory dependence on \(\varepsilon\).
  
\subsection{Riesz reduction and winding of isolated spectral clusters}
\label{subsec:finite_dimensional_scale_bloch_winding}

We extract finite-dimensional topological invariants from invariant spectral subspaces of the BFZ fiber symbol. Throughout this subsection, \(k>0\), \(\eta>0\), \(\lambda>1\), and \(\beta\in(-1,0)\) are fixed. We write \(t:=\theta\log\lambda\in[0,2\pi]\). Thus, on the fundamental cell,
\[
\mathcal T_{k,\beta}^{\Gamma_0}(t) = \widetilde{\mathcal T}_{k,\beta}\!\left(\frac{t}{\log\lambda}\right).
\]

Let \(\Gamma\) be a smooth compact boundary for which there exist \(0<r_-<r_+\) such that
\[
\Gamma\subset\{x\in\mathbb R^3:r_-\le |x|\le r_+\}, \qquad r_+<\lambda r_-.
\]
Let \(\mathcal H_\Gamma:=\mathcal H_k(\Gamma)\) be defined by the same Laplace-Beltrami formula as in \eqref{eq:Hk}, and define \(B_n^{k,\Gamma}\) by the same reference-cell and interscale formulas as \(B_n^k\). Under this annular scale-separation condition, the proof of Proposition~\ref{prop:one_sided_Bn_bounds} gives
\begin{equation}
\label{eq:weighted_bfz_c1_condition_gamma}
\sum_{n\in\mathbb Z}(1+|n|)\lambda^{\beta n}\|B_n^{k,\Gamma}\|_{\mathcal L(\mathcal H_\Gamma)}<\infty.
\end{equation}
The weighted BFZ symbol is
\begin{equation}
\label{eq:weighted_bfz_symbol_gamma}
\mathcal T_{k,\beta}^{\Gamma}(t):=\sum_{n\in\mathbb Z}e^{-int}\lambda^{\beta n}B_n^{k,\Gamma},\qquad t\in\mathbb R/(2\pi\mathbb Z).
\end{equation}
The series converges absolutely in the operator norm, uniformly in \(t\), and defines a \(C^1\) loop on \(\mathcal H_\Gamma\).

To compare parameter-dependent spectral subspaces on a fixed finite-dimensional space, choose a bounded reference operator \(\mathcal T_{\mathrm{ref}}^\Gamma\in\mathcal L(\mathcal H_\Gamma)\) and a positively oriented contour \(\mathscr C\) enclosing an isolated spectral cluster whose Riesz projection has finite rank. Set
\begin{equation}
\label{eq:reference_riesz_projection_gamma}
\Pi^{\mathrm{ref}}:=\frac{1}{2\pi i}\int_{\mathscr C}(\zeta I-\mathcal T_{\mathrm{ref}}^\Gamma)^{-1}\,d\zeta,
\qquad E:=\operatorname{Ran}\Pi^{\mathrm{ref}}, \qquad 1\le\dim_{\mathbb C}E<\infty.
\end{equation}
The reference Riesz projection \(\Pi^{\mathrm{ref}}\) is not assumed orthogonal.
Assume that \(\mathscr C\subset\mathbb C\setminus\sigma(\mathcal T_{k,\beta}^{\Gamma}(t))\) for all \(t\), and define
\begin{equation}
\label{eq:moving_riesz_projection_gamma}
\Pi^\Gamma(t):=\frac{1}{2\pi i}\int_{\mathscr C}(\zeta I-\mathcal T_{k,\beta}^{\Gamma}(t))^{-1}\,d\zeta.
\end{equation}
The parameter-dependent Riesz projection maps the reference space \(E\) to the corresponding spectral subspace. We assume that its restriction
\[
U_E^\Gamma(t):=\Pi^\Gamma(t)|_E:E\longrightarrow\operatorname{Ran}\Pi^\Gamma(t)
\]
is an isomorphism. These conditions follow, in particular, when the BFZ loop is sufficiently close in norm to \(\mathcal T_{\mathrm{ref}}^\Gamma\); this is the situation verified below in Section~\ref{sec:capacitance_riesz_phases}. The Riesz-reduced BFZ symbol is
\begin{equation}
\label{eq:riesz_reduced_bfz_symbol_gamma}
M_E^{\Gamma}(t):=(U_E^\Gamma(t))^{-1}\mathcal T_{k,\beta}^{\Gamma}(t)|_{\operatorname{Ran}\Pi^\Gamma(t)}U_E^\Gamma(t)\in\mathcal L(E).
\end{equation}
Thus, \(M_E^\Gamma(t)\) is similar to the restriction of the full symbol to its invariant spectral subspace. It is a continuous periodic loop and is \(C^1\) under \eqref{eq:weighted_bfz_c1_condition_gamma}. Proposition~\ref{prop:finite_cluster_reduction_nearby_loops} below proves that the computable projected operator on the fixed reference space,
\(\Pi^{\mathrm{ref}}\mathcal T_{k,\beta}^{\Gamma}(t)|_E\)
approximates \(M_E^\Gamma(t)\)  with an error quadratic in the perturbation
\(\mathcal T_{k,\beta}^{\Gamma} - \mathcal T_{\mathrm{ref}}^\Gamma\). The restriction to \(\operatorname{Ran}\Pi^\Gamma(t)\), or equivalently, the similarity class of \(M_E^\Gamma(t)\), is intrinsic to the isolated cluster. In particular, its determinant loop and determinant winding are independent of the chosen frame.

We first fix the winding convention associated with increasing \(t\) from \(0\) to \(2\pi\).

\begin{definition}[Winding number of a scalar loop]
\label{def:scalar_winding_number}
Let \(f:[0,2\pi]\to\mathbb C\setminus\{0\}\) be a continuous loop, that is,  \(f(0)=f(2\pi)\). Its winding number around the origin is denoted by
\[
\operatorname{Wind}(f,0)\in\mathbb Z.
\]
If \(f\) is \(C^1\), then
\[
\operatorname{Wind}(f,0) = \frac{1}{2\pi i}\int_0^{2\pi}\frac{f'(t)}{f(t)}\,dt.
\]
\end{definition}

Applying scalar winding to the determinant loop gives the point-gap invariant used below.

\begin{definition}[Determinant winding]
\label{def:determinant_winding}
Let \(z_\ast\in\mathbb C\). We say that the Riesz-reduced loop \(M_E^\Gamma\) has a point gap at \(z_\ast\) if
\begin{equation}
\label{eq:det_point_gap_condition}
\det_E\bigl(M_E^\Gamma(t)-z_\ast I_E\bigr)\neq0,\qquad 0\le t\le2\pi.
\end{equation}
If \eqref{eq:det_point_gap_condition} holds, the determinant winding of \(M_E^\Gamma\) around \(z_\ast\) is
\[
W(M_E^\Gamma;z_\ast) := \operatorname{Wind}\left(\det_E\bigl(M_E^\Gamma(\cdot)-z_\ast I_E\bigr),0\right)\in\mathbb Z.
\]
Writing \(\sigma_{\min}\) for the smallest singular value, finite dimensionality and continuity imply
\[
\inf_{0\le t\le2\pi}\sigma_{\min}\bigl(M_E^\Gamma(t)-z_\ast I_E\bigr)>0.
\]
\end{definition}

\begin{remark}[Zak phase and point-gap winding]
\label{rem:zak_phase_and_point_gap_winding}
For comparison, suppose that the reduced loop \(M_E^\Gamma(t)\) has a periodic algebraically simple eigenvalue branch and that the BFZ boundary symbol admits corresponding periodic \(C^1\) right and left eigenvectors \(u_j^R(t)\) and \(u_j^L(t)\), normalized by
\[
(u_j^L(t))^\dagger u_j^R(t)=1.
\]
Its biorthogonal BFZ Zak phase is
\begin{equation}
\label{eq:early_bandwise_bfz_zak_phase}
\gamma_j^{\mathrm{BFZ}} :=-\operatorname{Im}\int_0^{2\pi} (u_j^L(t))^\dagger\partial_tu_j^R(t)\,dt \quad \pmod{2\pi}.
\end{equation}
Equivalently, differentiating the normalization gives
\[
\gamma_j^{\mathrm{BFZ}} \equiv\frac{i}{2}\int_0^{2\pi} \left((u_j^L)^\dagger\partial_tu_j^R +(u_j^R)^\dagger\partial_tu_j^L\right)dt \pmod{2\pi}.
\]
This phase measures the transport of an eigenvector line, whereas \(W(M_E^\Gamma;z_\ast)\) measures the winding of the determinant of the Riesz-reduced matrix about the point gap. Thus, the determinant winding does not require simple band or global eigenvector labeling. The relation between these quantities is discussed further in Section~\ref{subsec:symmetrization_zak_phase}.
\end{remark}

The next theorem gives the trace formula and the point-gap stability used in the following.

\begin{theorem}[Finite-dimensional winding and stability] 
\label{thm:finite_dimensional_winding_stability} 
Let \(E\) be a finite-dimensional complex Hilbert space, \(\dim_{\mathbb C}E=q\), and let \(M:[0,2\pi]\to\mathcal L(E)\) be a continuous loop, namely \(M(0)=M(2\pi)\). Fix \(z_\ast\in\mathbb C\). If \(M(t)-z_\ast I_E\) is invertible for all \(t\), define \(W(M;z_\ast)\) as in Definition~\ref{def:determinant_winding}. 
If \(M\) is \(C^1\), then
\begin{equation}
\label{eq:finite_dim_trace_winding_formula}
W(M;z_\ast)=\frac{1}{2\pi i}\int_0^{2\pi}\operatorname{tr}_E\left[(M(t)-z_\ast I_E)^{-1}\partial_tM(t)\right]\,dt.
\end{equation}
If \((s,t)\mapsto M_s(t)\) and \(s\mapsto z_s\) are continuous, each \(M_s\) is periodic, and
\[
M_s(t)-z_sI_E\quad\text{is invertible for every }(s,t)\in[0,1]\times[0,2\pi],
\]
then \(W(M_s;z_s)\) is independent of \(s\). In particular, let \(M_0,M_1:[0,2\pi]\to\mathcal L(E)\) be continuous loops and assume
\begin{equation}
\label{eq:finite_dim_stability_assumption}
\mu_0:=\inf_{0\le t\le2\pi}\sigma_{\min}\bigl(M_0(t)-z_\ast I_E\bigr)>0, \qquad \sup_{0\le t\le2\pi}\|M_1(t)-M_0(t)\|_{\mathcal L(E)}<\mu_0. 
\end{equation}
Then \(M_1(t)-z_\ast I_E\) is invertible for every \(t\), and 
\begin{equation}
\label{eq:finite_dim_stability_conclusion}
W(M_1;z_\ast)=W(M_0;z_\ast). 
\end{equation}
\end{theorem}
 
\begin{proof}
Since \(M(t)-z_\ast I_E\) is invertible for every \(t\), the scalar loop
\[
t\mapsto \det_E\bigl(M(t)-z_\ast I_E\bigr)
\]
takes values in \(\mathbb C\setminus\{0\}\). Hence \(W(M;z_\ast)\) is well-defined.

If \(M\) is \(C^1\), Jacobi's formula gives
\[
\partial_t\det_E\bigl(M(t)-z_\ast I_E\bigr)=\det_E\bigl(M(t)-z_\ast I_E\bigr)\operatorname{tr}_E\left[\bigl(M(t)-z_\ast I_E\bigr)^{-1}\partial_tM(t)\right].
\]
Dividing by the determinant and integrating over one period gives \eqref{eq:finite_dim_trace_winding_formula}.

For a point-gapped family \((M_s,z_s)\), the scalar loops
\(\det_E(M_s(\cdot)-z_sI_E)\) form a homotopy in \(\mathbb C\setminus\{0\}\). Their winding is therefore independent of \(s\).

Under \eqref{eq:finite_dim_stability_assumption}, define
\[
M_s(t):=M_0(t)+s\bigl(M_1(t)-M_0(t)\bigr), \qquad 0\le s\le1.
\]
Then
\[
M_s(t)-z_\ast I_E=\bigl(M_0(t)-z_\ast I_E\bigr)\left[I_E+s\bigl(M_0(t)-z_\ast I_E\bigr)^{-1}\bigl(M_1(t)-M_0(t)\bigr)\right].
\]
The bracketed factor is invertible because, by \eqref{eq:finite_dim_stability_assumption},
\[
s\left\|\bigl(M_0(t)-z_\ast I_E\bigr)^{-1}\bigl(M_1(t)-M_0(t)\bigr)\right\|\le\mu_0^{-1}\sup_t\|M_1(t)-M_0(t)\|<1.
\]
Therefore, \(M_s(t)-z_\ast I_E\) is invertible for every \(s,t\). The determinant loops form a homotopy in \(\mathbb C\setminus\{0\}\), so
\[
W(M_1;z_\ast)=W(M_0;z_\ast).
\]
\end{proof}

Thus, a continuous deformation of the reduced loop preserves the winding, provided that the chosen point gap remains open throughout the deformation. For a shape-induced deformation, the isolated Riesz cluster must also persist.

The symmetry decomposition used later is a special case of the following standard observation.

\begin{remark}[Symmetry-sector winding]
Let a finite group act on a finite-dimensional complex Hilbert space \(E\), with isotypic decomposition
\[
E\cong\bigoplus_\alpha \mathcal V_\alpha\otimes\mathbb C^{m_\alpha}.
\]
If a continuous loop \(L(t)\) commutes with this action, Schur's lemma gives
\[
L(t)\cong\bigoplus_\alpha I_{\mathcal V_\alpha}\otimes L_\alpha(t).
\]
Whenever \(z\) is a point gap, the determinant multiplicativity yields
\[
W(L;z)=\sum_\alpha\dim(\mathcal V_\alpha)W(L_\alpha;z).
\]
For the simplex permutation representation used in the following, the two irreducible subspaces are the symmetric line and the \((M-1)\)-dimensional sum-zero space, namely
\[
\operatorname{span}\{(1,\ldots,1)\}\qquad\text{and}\qquad\left\{c\in\mathbb C^M:\sum_i c_i=0\right\}.
\]
Each irreducible type occurs once: the trivial representation has dimension \(1\), and the standard representation has dimension \(M-1\). Thus, the two reduced sector loops are scalar, which recovers the two-sector factorization without replacing the explicit simplex calculation.
\end{remark}

\subsection{Finite-depth local data and multiple dilation centers}
\label{subsec:finite_winding_vectors}

We now explain how local determinant winding data can be attached to smooth finite-depth approximations generated by an iterative system. The construction is conditional: the local reduced BFZ loops must exist and have point gaps at the chosen reference points.

Let \(I=\{1,\dots,M\}\), and let
\[
T_i(x)=\lambda^{-1}x+b_i,\qquad i\in I,
\]
with common contraction ratio \(\lambda^{-1}\).

Fix \(\beta\in(-1,0)\), and let \(\kappa>0\) denote the effective wavenumber in a reference cell. Let \(\mathcal R\) index a finite collection of local reduced models. For each \(r\in\mathcal R\), choose a smooth reference boundary \(\Gamma_{r,0}\) contained in one annular scale cell, a possibly \(\kappa\)-dependent reference operator with an isolated spectral cluster enclosed by a contour \(\mathscr C_r\), and a reference point \(z_r\in\mathbb C\). We assume that, under reference-boundary identification, the range of the reference Riesz projection is a fixed finite-dimensional space \(E_r\subset C^\infty(\Gamma_{r,0})\). Thus, each index \(r\) specifies one local reduced model that can be copied throughout the prefractal construction.

Let \(\mathcal T_{r,\kappa,\beta}(t)\) be the scale-periodic BFZ symbol associated with \(\Gamma_{r,0}\). Assume that \(\mathscr C_r\) remains in its resolvent set and that its Riesz projection restricts to an isomorphism from \(E_r\) onto the range of the parameter-dependent Riesz projection. Applying the construction of Subsection~\ref{subsec:finite_dimensional_scale_bloch_winding}, denote the resulting reduced loop on \(E_r\) by \(M_r(\kappa,t)\).
Different elements of \(\mathcal R\) may have the same reference boundary but different isolated clusters. Assume that, for each \(r\in\mathcal R\),
\begin{equation}
\label{eq:local_point_gap_condition_general}
\det_{E_r}\bigl(M_r(\kappa,t)-z_r I_{E_r}\bigr)\ne0, \qquad 0\le t\le2\pi.
\end{equation}
Then define the local determinant winding
\begin{equation}
\label{eq:local_determinant_winding_general}
W_r(\kappa;z_r):= \operatorname{Wind}\left( \det_{E_r}\bigl(M_r(\kappa,\cdot)-z_r I_{E_r}\bigr),0 \right).
\end{equation}

For a finite depth \(N\), let \(m_{\ell r}\in\mathbb N_0\) be the number of prescribed copies of local type \(r\) at level \(\ell\). Whenever the local point-gap conditions are satisfied, define
\begin{equation}
\label{eq:prefractal_winding_data_general}
W_\ell^{\mathrm{loc}}(\kappa) :=\operatorname{Wind}\left(\prod_{r\in\mathcal R}\det_{E_r}\bigl(M_r(\kappa,\cdot)-z_rI_{E_r}\bigr)^{m_{\ell r}},0\right),
\qquad
\mathbf W_N^{\mathrm{loc}}(\kappa):=\bigl(W_0^{\mathrm{loc}}(\kappa),\ldots,W_N^{\mathrm{loc}}(\kappa)\bigr).
\end{equation}
Equivalently, with zero-multiplicity summands omitted, the product in \eqref{eq:prefractal_winding_data_general} is the determinant of
\[
\bigoplus_{r\in\mathcal R}I_{m_{\ell r}}\otimes\bigl(M_r(\kappa,t)-z_rI_{E_r}\bigr).
\]
This direct sum represents finitely many uncoupled local reference loops, not a Riesz reduction of the fully coupled prefractal operator; the latter identification requires a separate estimate controlling the coupling among the local blocks while preserving the relevant point gap.

Under \eqref{eq:local_point_gap_condition_general}, the determinant multiplicativity gives the finite-dimensional identity
\begin{equation}
\label{eq:general_prefractal_winding_formula}
\mathbf W_N^{\mathrm{loc}}(\kappa) = \left(\sum_{r\in\mathcal R}m_{\ell r}W_r(\kappa;z_r)\right)_{\ell=0}^{N},
\end{equation}
for the finite direct sum of uncoupled local blocks. 

\begin{remark}[Finite truncations and weighted determinants]
\label{rem:weighted_infinite_depth_determinant}
The finite-dimensional trace formula has a monodromy interpretation. If \(U'(t)=U(t)\mathcal B(t)\), then
\[
\det U(t)=\det U(0)\exp\left(\int_0^t\operatorname{tr}\mathcal B(s)\,ds\right).
\]
For a point-gapped loop \(A(t)=M(t)-z_\ast I\), taking \(\mathcal B=A^{-1}A'\), the trace satisfies
\[
\operatorname{tr}_E(A(t)^{-1}A'(t))=\frac{(\det_E A(t))'}{\det_E A(t)}.
\]
Then the traced Maurer-Cartan form measures the logarithmic monodromy of \(\det A\):
\[
\frac{1}{2\pi i}\int_0^{2\pi}\operatorname{tr}\bigl(A(t)^{-1}A'(t)\bigr)\,dt=W(M;z_\ast).
\]
Thus, the trace formula \eqref{eq:finite_dim_trace_winding_formula} is the period of the traced Maurer-Cartan form.

At infinite depth, one may formally introduce the weighted diagonal functional
\[
\tau_c(B):=\sum_{\ell\ge0}c_\ell\operatorname{tr}_{\mathscr E_\ell}(P_\ell BP_\ell),
\]
whenever every diagonal block is trace class and the series converges absolutely. Writing \(P_{\le N}:=\sum_{\ell=0}^N P_\ell\), the hard cutoff \(c_\ell=\mathbf 1_{\{\ell\le N\}}\) gives
\[
\tau_c(B)=\operatorname{Tr}(P_{\le N}BP_{\le N}),
\]
whenever this compression is trace class. A determinant-like quantity \(\exp\bigl(\tau_c(\log(I+A))\bigr)\) could be introduced only after specifying a class of operators on which \(\tau_c\) is cyclic and for which a chosen logarithm belongs to its domain.  
\end{remark}

The next corollary combines the homotopy invariance of the local loops with the geometric scaling of the wavenumber across levels.

\begin{corollary}[Winding along a geometric wavenumber sequence]
\label{cor:general_frequency_ladder_winding}
Let
\[
\mathcal I_\kappa=[\kappa_-,\kappa_+]\Subset(0,\infty).
\]
Assume that, for every \(r\in\mathcal R\), the map \((\kappa,t)\mapsto M_r(\kappa,t)\) is continuous in the operator norm on \(\mathcal I_\kappa\times[0,2\pi]\), and that
\[
\inf_{\kappa\in\mathcal I_\kappa,\ 0\le t\le2\pi}\sigma_{\min}\bigl(M_r(\kappa,t)-z_rI_{E_r}\bigr)>0.
\]
Then \(W_r(\kappa;z_r)\) is constant for \(\kappa\in\mathcal I_\kappa\).

Fix \(\kappa_\circ\in\mathcal I_\kappa\). For each \(0\le j\le N\), set \(k_j:=\lambda^j\kappa_\circ\). Since a level-\(j\) cell has an effective reference-cell wavenumber \(k_j\lambda^{-j}=\kappa_\circ\), its contribution at the physical wavenumber \(k_j\) is
\[
\sum_{r\in\mathcal R}m_{jr}W_r(\kappa_\circ;z_r).
\]
This is a level-resolved local contribution, not the determinant winding of the full depth-\(N\) prefractal at \(k_j\). Other levels have different effective wavenumbers and are not controlled by this statement.
\end{corollary}

\begin{proof}
The uniform point gaps give a homotopy in \(\kappa\) through nonvanishing determinant loops. To justify the effective wavenumber, put \(s_j:=\lambda^{-j}\) and let \(b\) be the translation locating a level-\(j\) cell. The Green-function identities
\[
G_{k_j}(b+s_jx,b+s_jy)=s_j^{-1}G_{s_jk_j}(x,y), \qquad \partial_{\nu_{b+s_jx}}G_{k_j}(b+s_jx,b+s_jy)=s_j^{-2}\partial_{\nu_x}G_{s_jk_j}(x,y),
\]
together with the equilibrium-density normalization \(\mu(b+s_jy)=s_j^{-1}\varphi(y)\), show that the combined-field block on the level-\(j\) cell at wavenumber \(k_j\) is conjugate to the reference-cell block at wavenumber \(s_jk_j\). The same calculation with \(y\) replaced by \(\lambda^ny\) applies to every relative-scale block, and hence to the BFZ symbol and its Riesz reduction. Since \(s_jk_j=\kappa_\circ\), \eqref{eq:general_prefractal_winding_formula} gives the displayed level formula.
\end{proof}

The local vector above records level multiplicities obtained by copying prescribed local block types along a finite iterative system. A different finite collection of winding data is obtained by fixing several dilation centers. Each center gives a separate scale-periodic extension. For example, in a Sierpiński-type triangle, one may base the construction at any of the three fixed vertices. The resulting loops remain independent one-dimensional BFZ loops; they do not form a higher-dimensional Brillouin torus.

Fix \(\lambda>1\), \(k>0\), \(\eta>0\), and \(\beta\in(-1,0)\). Let \(\xi\in\mathbb R^3\). We call \(\xi\) an admissible dilation center if there exist a smooth compact reference boundary \(\Gamma_{\xi,0}\) and radii \(r_{\xi,-},r_{\xi,+}\) satisfying
\[
\Gamma_{\xi,0}\subset\{x\in\mathbb R^3:r_{\xi,-}\le |x-\xi|\le r_{\xi,+}\}, \qquad 0<r_{\xi,-}<r_{\xi,+}<\lambda r_{\xi,-}.
\]

The corresponding scale-periodic boundary centered at \(\xi\) is
\[
\Gamma_\xi^\infty:=\bigsqcup_{m\in\mathbb Z}\Gamma_{\xi,m},\qquad \Gamma_{\xi,m}:=\xi+\lambda^m(\Gamma_{\xi,0}-\xi).
\]
We use the same equilibrium-density scaling and reverse boundary index as in Subsection~\ref{subsec:blocks_Laurent}.
For \(n\neq0\), define the interscale block associated with \(\xi\) on \(\Gamma_{\xi,0}\) by
\[
(B_{\xi,n}^k\varphi)(x):=\lambda^n\int_{\Gamma_{\xi,0}}\mathsf K_\eta^k\bigl(x,\xi+\lambda^n(y-\xi)\bigr)\varphi(y)\,d\sigma(y), \qquad x\in\Gamma_{\xi,0}.
\]
 
For \(n=0\), set \(B_{\xi,0}^k:=B_0^{k,\Gamma_{\xi,0}}\).
The annular separation condition implies that
\[
\operatorname{dist}(\Gamma_{\xi,0},\Gamma_{\xi,n})>0, \qquad n\neq0,
\]
and the same one-sided estimates as in Proposition~\ref{prop:one_sided_Bn_bounds} give the weighted summability estimate
\[
\sum_{n\in\mathbb Z}\lambda^{\beta n}\|B_{\xi,n}^k\|_{\mathcal L(\mathcal H_k(\Gamma_{\xi,0}))}<\infty.
\]
Therefore, the BFZ symbol associated with the dilation center \(\xi\),
\[
\mathcal T_{\xi,k,\beta}(t):=\sum_{n\in\mathbb Z}e^{-int}\lambda^{\beta n}B_{\xi,n}^k, \qquad 0\le t\le2\pi,
\]
is a norm-continuous operator-valued loop on \(\mathcal H_k(\Gamma_{\xi,0})\).

Let \(\Xi\) be a finite set of dilation centers. For each \(\xi\in\Xi\), choose a reference operator with an isolated spectral cluster inside a contour \(\mathscr C_\xi\), and let \(E_\xi\) be the range of its Riesz projection. Assume that \(\mathscr C_\xi\) remains in the resolvent set of \(\mathcal T_{\xi,k,\beta}(t)\) and that the associated Riesz projection restricts to an isomorphism from \(E_\xi\) onto the range of the parameter-dependent Riesz projection. Applying the construction of Subsection~\ref{subsec:finite_dimensional_scale_bloch_winding}, denote the resulting reduced loop on \(E_\xi\) by \(M_\xi(t)\).

We collect these independent one-dimensional invariants in a finite vector.

\begin{definition}[Windings for prescribed dilation centers]
\label{def:directional_winding_vector}
For each \(\xi\in\Xi\), fix \(z_\xi\in\mathbb C\). Assume that
\[
\det_{E_\xi}\bigl(M_\xi(t)-z_\xi I_{E_\xi}\bigr)\neq0,\qquad 0\le t\le2\pi.
\]
The winding associated with \(\xi\) is
\[
W_\xi(z_\xi):=\operatorname{Wind}\left(\det_{E_\xi}\bigl(M_\xi(\cdot)-z_\xi I_{E_\xi}\bigr),0\right).
\]
The corresponding center-indexed winding vector is
\[
\mathbf W_{\Xi}:=\bigl(W_\xi(z_\xi)\bigr)_{\xi\in\Xi}.
\]
\end{definition}

Theorem~\ref{thm:finite_dimensional_winding_stability} applies to each component of \(\mathbf W_\Xi\). These entries belong to independent one-dimensional BFZ loops based at different dilation centers; they are not coordinates of a higher-dimensional Brillouin torus or a coupled multi-center invariant.

\section{Finite-dimensional projection, Riesz reduction, and BFZ phases}
\label{sec:capacitance_riesz_phases}

The preceding section provides the general BFZ winding framework. We now project the BFZ symbol onto the static kernel spanned by the isolated-component equilibrium densities, compute the winding of the resulting projected matrix for a regular simplex, compare it with the invariant Riesz reduction, and study the associated geometric phases. 

\subsection{Calder\'on identity and the static kernel pairing}
\label{subsec:calderon_capacitance_pairing}

Recall the layer-potential operators \(S_k^\Gamma\), \(K_k^\Gamma\), and \((K_k^\Gamma)^\ast\) defined in Section~\ref{subsec:boundary_formulation_one_scale_cell}. 
For \(f\in H^{1/2}(\Gamma)\) and \(g\in H^{-1/2}(\Gamma)\), let \(\langle f,g\rangle_{\mathrm{bil}}\) denote the complex-bilinear \(H^{1/2}(\Gamma)\)-\(H^{-1/2}(\Gamma)\) duality pairing; for smooth functions it equals \(\int_\Gamma fg\,d\sigma\). If \(T\in\mathcal L(H^{1/2}(\Gamma))\), its bilinear transpose \(T^{\mathsf T}\in\mathcal L(H^{-1/2}(\Gamma))\) is defined by
\[
\langle Tf,g\rangle_{\mathrm{bil}}=\langle f,T^{\mathsf T}g\rangle_{\mathrm{bil}}.
\]
It is distinct from the Hilbert adjoint \({}^{\dagger}\). On a smooth boundary, the layer-potential operators act on all Sobolev orders used below, so identities proved for smooth functions extend by continuity to compatible trace spaces.

The next proposition records the Calder\'on identity used to identify the right kernel and the kernel of the transpose at zero wavenumber.

\begin{proposition}[Calder\'on identity]
\label{prop:one_cell_calderon_identity}
Let \(k\ge0\), \(\eta>0\), and let \(\Gamma\) be a smooth compact boundary. Then
\begin{equation}
\label{eq:one_cell_calderon_identities}
(S_k^\Gamma)^{\mathsf T}=S_k^\Gamma, \qquad \bigl((K_k^\Gamma)^\ast\bigr)^{\mathsf T}=K_k^\Gamma, \qquad S_k^\Gamma(K_k^\Gamma)^\ast=K_k^\Gamma S_k^\Gamma.
\end{equation}
Consequently, the reference-cell operator \(B_0^{k,\Gamma}\) in \eqref{eq:combined_field_operator} satisfies
\begin{equation}
\label{eq:combined_field_calderon_identity}
S_k^\Gamma B_0^{k,\Gamma}=(B_0^{k,\Gamma})^{\mathsf T}S_k^\Gamma.
\end{equation}
\end{proposition}

\begin{proof}
For smooth \(\varphi\) and \(\chi\), Fubini's theorem and the symmetry \(G_k(x,y)=G_k(y,x)\) give
\[
\begin{aligned}
\langle S_k^\Gamma\varphi,\chi\rangle_{\mathrm{bil}}&=\int_\Gamma\!\int_\Gamma G_k(x,y)\varphi(y)\chi(x)\,d\sigma(y)d\sigma(x)=\langle\varphi,S_k^\Gamma\chi\rangle_{\mathrm{bil}},\\
\langle (K_k^\Gamma)^\ast\varphi,\chi\rangle_{\mathrm{bil}} &=\int_\Gamma\!\int_\Gamma \partial_{\nu_x}G_k(x,y)\varphi(y)\chi(x)\,d\sigma(y)d\sigma(x) =\langle\varphi,K_k^\Gamma\chi\rangle_{\mathrm{bil}}.
\end{aligned}
\]
This proves the first two identities in \eqref{eq:one_cell_calderon_identities}. The Calder\'on symmetrization identity gives
\[
S_k^\Gamma(K_k^\Gamma)^\ast=K_k^\Gamma S_k^\Gamma.
\]
Since the bilinear transpose does not conjugate scalar coefficients,
\[
(B_0^{k,\Gamma})^{\mathsf T} =\frac12I+K_k^\Gamma-i\eta kS_k^\Gamma \quad\text{in }\mathcal L(H^{-1/2}(\Gamma)).
\]
Consequently,
\[
S_k^\Gamma B_0^{k,\Gamma} = \frac12S_k^\Gamma+K_k^\Gamma S_k^\Gamma-i\eta k(S_k^\Gamma)^2 = (B_0^{k,\Gamma})^{\mathsf T}S_k^\Gamma.
\]
\end{proof}

At zero wavenumber, the Calder\'on identity, the jump relation, and standard potential-theoretic properties of the single-layer operator identify the Riesz projection onto the static kernel used in the finite-dimensional reduction.

\begin{corollary}[Riesz projection onto the static kernel and dual pairing]
\label{cor:static_capacitance_pairing}
Let \(D\subset\mathbb R^3\) be bounded and smooth with a connected boundary \(\Gamma=\partial D\), and let \(\psi\in H^{-1/2}(\Gamma)\) be determined by \(S_0^\Gamma\psi=\mathbf1_\Gamma\). Smooth-boundary elliptic regularity gives \(\psi\in C^\infty(\Gamma)\). Then
\begin{equation}
\label{eq:static_capacitance_kernel_cokernel}
B_0^{0,\Gamma}\psi=0, \qquad (B_0^{0,\Gamma})^{\mathsf T}\mathbf1_\Gamma=0.
\end{equation}
More precisely,
\[
\ker_{H^{1/2}(\Gamma)}B_0^{0,\Gamma}=\operatorname{span}\{\psi\}, \qquad \ker_{H^{-1/2}(\Gamma)}(B_0^{0,\Gamma})^{\mathsf T}=\operatorname{span}\{\mathbf1_\Gamma\}.
\]
Set \(\operatorname{Cap}(D):=\int_\Gamma\psi\,d\sigma>0\). The right-kernel density \(\psi\) and the normalized element \(\mathbf1_\Gamma/\operatorname{Cap}(D)\) of the transpose kernel have a bilinear pairing equal to one. Now, let \(D_i\), \(i=1,\ldots,M\), be finitely many such domains, set \(\Gamma_i:=\partial D_i\), and let \(S_0^{\Gamma_i}\psi_i=\mathbf1_{\Gamma_i}\). For \(f\in\bigoplus_{i=1}^M H^{1/2}(\Gamma_i)\), define the finite-rank map
\begin{equation}
\label{eq:static_capacitance_projection}
f\longmapsto\sum_{i=1}^M\psi_i \frac{\displaystyle\int_{\Gamma_i}f\,d\sigma} {\operatorname{Cap}(D_i)}.
\end{equation}
Zero is a semisimple eigenvalue of the block-diagonal operator \(\bigoplus_iB_0^{0,\Gamma_i}\), and the map in \eqref{eq:static_capacitance_projection} is its Riesz projection at zero. 
\end{corollary}

\begin{proof}
Let \(u:=S_0^\Gamma\psi\). Its interior trace is \(\mathbf1_\Gamma\), so uniqueness for the interior Dirichlet problem gives \(u\equiv1\) in \(D\). Hence, its interior normal derivative vanishes. With the jump convention used in \eqref{eq:combined_field_operator},
\[
\partial_\nu^-u =\left(\frac12I+(K_0^\Gamma)^\ast\right)\psi =B_0^{0,\Gamma}\psi,
\]
which proves the first identity in \eqref{eq:static_capacitance_kernel_cokernel}. Proposition~\ref{prop:one_cell_calderon_identity} at \(k=0\) then gives
\[
(B_0^{0,\Gamma})^{\mathsf T}\mathbf1_\Gamma =(B_0^{0,\Gamma})^{\mathsf T}S_0^\Gamma\psi =S_0^\Gamma B_0^{0,\Gamma}\psi=0.
\]
Uniqueness of the single-layer equation and the real-valued integral kernel of \(S_0^\Gamma\) imply that \(\psi\) is real. The positive single-layer energy therefore gives
\[
\operatorname{Cap}(D)=\int_\Gamma\psi\,S_0^\Gamma\psi\,d\sigma>0,
\]
and the normalized bilinear pairing is one.
Moreover, the exterior potential is harmonic, equals one on \(\Gamma\), and decays at infinity. The maximum principle and the Hopf lemma give \(\partial_\nu^+u<0\); since the jump relation in our convention is \(\partial_\nu^+u-\partial_\nu^-u=-\psi\), it follows that \(\psi>0\).

For the direct sum, define
\[
\Lambda_i(f):=\frac{1}{\operatorname{Cap}(D_i)} \int_{\Gamma_i}f\,d\sigma.
\]
Because each density is extended by zero off its own component,
\[
\Lambda_i(\psi_j)=\delta_{ij}.
\]
Consequently, if \(P\) denotes the map in \eqref{eq:static_capacitance_projection}, then
\[
P^2f=\sum_{i,j}\psi_i\Lambda_i(\psi_j)\Lambda_j(f)=\sum_j\psi_j\Lambda_j(f)=Pf.
\]
Moreover,
\[
B_0^{0,\Gamma_i}\psi_i=0, \qquad \Lambda_i(B_0^{0,\Gamma_i}f)=0,
\]
by the componentwise kernel and cokernel identities, so the block-diagonal static self-interaction operator annihilates \(P\) on both sides.

It remains to identify the spectral projection. If \(B_0^{0,\Gamma}f=0\), then the interior single-layer potential \(S_0^\Gamma f\) has zero normal derivative and is therefore constant. Uniqueness of the single-layer equation gives \(f\in\operatorname{span}\{\psi\}\). On a \(C^\infty\) boundary, \(K_0^\Gamma\) and \((K_0^\Gamma)^\ast\) improve Sobolev regularity by one order, so the kernel elements of \(B_0^{0,\Gamma}\) on \(H^{1/2}(\Gamma)\) and of \((B_0^{0,\Gamma})^{\mathsf T}\) on \(H^{-1/2}(\Gamma)\) are smooth. Since \((K_0^\Gamma)^\ast\) is compact on \(H^{1/2}(\Gamma)\), \(B_0^{0,\Gamma}\) is Fredholm of index zero. Its cokernel is identified through the \(H^{1/2}\)-\(H^{-1/2}\) bilinear duality with the kernel of \((B_0^{0,\Gamma})^{\mathsf T}\) on \(H^{-1/2}(\Gamma)\). The right kernel is one-dimensional, while \(\mathbf1_\Gamma\) belongs to the transpose kernel by \eqref{eq:static_capacitance_kernel_cokernel}; hence, the latter kernel is \(\operatorname{span}\{\mathbf1_\Gamma\}\). Finally, if
\[
(B_0^{0,\Gamma})^2f=0, \qquad B_0^{0,\Gamma}f=\alpha\psi,
\]
then pairing with \(\mathbf1_\Gamma\) gives
\[
0=\langle B_0^{0,\Gamma}f,\mathbf1_\Gamma\rangle_{\mathrm{bil}}=\alpha\operatorname{Cap}(D),
\]
and therefore \(\alpha=0\). Thus, zero is semisimple on each component and on their direct sum. The range of each component operator is contained in \(\ker\Lambda_i\); both spaces have codimension one by the Fredholm alternative, so they are equal. Hence, the direct-sum range is exactly \(\ker P\). Together with \(\operatorname{Ran}P=\ker(\bigoplus_iB_0^{0,\Gamma_i})\), this proves that \(P\) is the Riesz projection at zero.
\end{proof}

For the scaled components introduced below, \eqref{eq:static_capacitance_projection} is the reference projection used to define the regular-simplex projected matrix; on spheres it is also the \(\mathcal H_k\)-orthogonal projection onto the componentwise constant densities.

Corollary~\ref{cor:static_capacitance_pairing} gives a pairing between the right kernel and the kernel of the transpose for each isolated component, analogous to the equilibrium-density and constant-function pairing used in capacitance-matrix reductions such as \cite{ammari_davies_hiltunen_yu2020}. The equilibrium equation and jump relation give the right-kernel densities, while the Calder\'on identity identifies the normalized constant dual functionals. Identity \eqref{eq:combined_field_calderon_identity} is a bilinear symmetrization of the reference-cell operator, not a singular-value decomposition. The biorthogonal frames for the full BFZ symbol constructed below come instead from its Riesz projection.

\subsection{Projected BFZ matrix in the equilibrium-density basis}
\label{subsec:scalar_scale_winding_reference}

Using the isolated-component static kernel and dual functionals above, we now form the matrix of the projected BFZ operator in the equilibrium-density basis and determine its winding for a regular-simplex configuration.

Let \(M\in\{2,3,4\}\), and let \(p_1,\dots,p_M\in\mathbb R^3\) be the vertices of a regular simplex, normalized by
\begin{equation}
\label{eq:simplex_vertices_normalization}
|p_i|=R,\qquad p_i\cdot p_j=-\frac{R^2}{M-1}\qquad(i\ne j).
\end{equation}
We first construct the reference symbol from the equilibrium densities of isolated copies of a general smooth component. Let \(D\subset\mathbb R^3\) be a bounded smooth domain with connected boundary, containing the origin, and let \(O_i\in O(3)\). For the simplex vertices \(p_i\) in \eqref{eq:simplex_vertices_normalization}, set
\[
D_i^\rho:=p_i+\rho O_iD, \qquad \Gamma_i^\rho:=\partial D_i^\rho, \qquad \Gamma^\rho:=\bigsqcup_{i=1}^M\Gamma_i^\rho.
\]
For sufficiently small \(\rho\), this boundary satisfies the annular scale-separation condition above: it lies between the radii \(R-\rho\sup_{y\in\partial D}|y|\) and \(R+\rho\sup_{y\in\partial D}|y|\), and, after increasing \(\lambda\) if necessary, the latter is smaller than \(\lambda\) times the former.
We call the configuration simplex-equivariant, or simply equivariant, if, for every permutation \(\sigma\) in the symmetric group \(S_M\), there is \(Q_\sigma\in O(3)\) such that
\[
Q_\sigma p_i=p_{\sigma(i)}, \qquad Q_\sigma D_i^\rho=D_{\sigma(i)}^\rho \quad\text{for every }i.
\]
Since \(Q_\sigma(\lambda^n\Gamma^\rho)=\lambda^nQ_\sigma\Gamma^\rho\) and the Green function is orthogonally invariant, the coefficient matrices defined below commute with the corresponding permutation matrices. Thus, the symmetric line and its sum-zero complement in \(\mathbb C^M\) are invariant. This is the symmetry input used below.

Let \(S_0^{\Gamma_i^\rho}\psi_i^\rho=\mathbf1_{\Gamma_i^\rho}\), and write
\[
C_i^\rho:=\operatorname{Cap}(D_i^\rho)=\int_{\Gamma_i^\rho}\psi_i^\rho\,d\sigma, \qquad \Lambda_i^\rho(f):=\frac{1}{C_i^\rho}\int_{\Gamma_i^\rho}f\,d\sigma.
\]
The scaling of the single-layer operator gives the following identities:
\[
\psi_i^\rho(p_i+\rho O_i y)=\rho^{-1}\psi_D(y), \qquad C_i^\rho=\rho\operatorname{Cap}(D) = \rho \int_{\partial D}\psi_D\,d\sigma,
\]
where \(S_0^{\partial D}\psi_D=1\). Define the span of the equilibrium densities and the associated projection by
\begin{equation}
\label{eq:general_capacitance_space_projection}
E_{\mathrm{cap}}^\rho:=\operatorname{span}\{\psi_1^\rho,\ldots,\psi_M^\rho\}, \qquad P_{\mathrm{cap}}^\rho f:=\sum_{i=1}^M\psi_i^\rho\Lambda_i^\rho(f).
\end{equation}
Corollary~\ref{cor:static_capacitance_pairing} shows that \(E_{\mathrm{cap}}^\rho\) is the kernel of the block-diagonal isolated-component operator
\[
\bigoplus_{i=1}^M B_0^{0,\Gamma_i^\rho},
\]
and that \(P_{\mathrm{cap}}^\rho\) is its Riesz projection at zero; it is not, in general, an orthogonal projection. Moreover, the families \(\{\psi_j^\rho\}_{j=1}^M\) and \(\{\Lambda_i^\rho\}_{i=1}^M\) are biorthogonal because \(\Lambda_i^\rho(\psi_j^\rho)=\delta_{ij}\).

Abbreviate \(B_n^{k,\rho}:=B_n^{k,\Gamma^\rho}\) and
\[
\mathcal T^\rho(t):=\sum_{n\in\mathbb Z}e^{-int}\lambda^{\beta n}B_n^{k,\rho}.
\]
The matrix coefficients of the projected BFZ operator in the equilibrium-density basis are
\begin{equation}
\label{eq:general_capacitance_coordinate_matrix}
\bigl(\mathsf C_{\beta,\mathrm{cap}}^\rho(t)\bigr)_{ij} := \Lambda_i^\rho\bigl(\mathcal T^\rho(t)\psi_j^\rho\bigr).
\end{equation}
Thus, \(\mathsf C_{\beta,\mathrm{cap}}^\rho(t)\) is the matrix representation of \(P_{\mathrm{cap}}^\rho\mathcal T^\rho(t)|_{E_{\mathrm{cap}}^\rho}\). By equivariance it is scalar on the symmetric line and the sum-zero subspace. Write \(\mu_{+,\mathrm{cap}}^\rho\) for the symmetric-sector loop and \(\mu_{-,\mathrm{cap}}^\rho\) for the sum-zero-sector loop, and let \(z_{\pm,\mathrm{cap}}^\rho\) be their respective zeroth Fourier coefficients. We use these coefficients as the sector reference points, so subtracting them removes the same-scale term in the selected sector.

The probability measure \(\operatorname{Cap}(D)^{-1}\psi_D\,d\sigma\) is the normalized equilibrium measure. The shape dependence of the leading scale interaction is contained in its Fourier transform,
\begin{equation}
\label{eq:equilibrium_form_factor}
F_D(\xi):=\frac{1}{\operatorname{Cap}(D)} \int_{\partial D}e^{i\xi\cdot y}\psi_D(y)\,d\sigma(y).
\end{equation}
The normalized equilibrium measure is positive, so \(F_D(0)=1\) and
\begin{equation}
\label{eq:equilibrium_form_factor_near_zero}
|F_D(\xi)-1|\le |\xi|\sup_{y\in D}|y|.
\end{equation}
In particular, \(F_D\) is nonzero on every sufficiently small ball about the origin.

The next lemma provides all the coefficient estimates needed for the general-shape winding argument. For an equivariant configuration, the value of \(F_D(sO_i^{\mathsf T}p_i/R)\) is independent of \(i\). Indeed, if a simplex symmetry \(Q\) sends \(p_i\) to \(p_j\), then \(O_j^{\mathsf T}QO_i\) preserves \(D\) and its equilibrium measure, while
\[
(O_j^{\mathsf T}QO_i)O_i^{\mathsf T}\frac{p_i}{R} =O_j^{\mathsf T}\frac{p_j}{R}.
\]
Fix \(\omega:=O_i^{\mathsf T}p_i/R\) for one index \(i\); the preceding equivariance argument makes \(F_D(s\omega)\) independent of this choice.

\begin{lemma}[Coefficients of the projected BFZ matrix for a regular simplex]
\label{lem:general_capacitance_coefficients}
Fix \(\beta\in(-1,0)\), and let
\[
c_{ij,n}^{\rho,\mathrm{cap}} :=\Lambda_i^\rho(B_n^{k,\rho}\psi_j^\rho),
\qquad
a_{+,n}^{\rho,\mathrm{cap}} :=c_{ii,n}^{\rho,\mathrm{cap}}+(M-1)c_{ij,n}^{\rho,\mathrm{cap}},
\qquad
a_{-,n}^{\rho,\mathrm{cap}} :=c_{ii,n}^{\rho,\mathrm{cap}}-c_{ij,n}^{\rho,\mathrm{cap}},
\]
where \(i\ne j\); equivariance makes these quantities independent of the chosen off-diagonal pair. For fixed \(k,R,\eta\) and \(c_\ast>0\), uniformly for small \(\rho\), large \(\lambda\), and \(k\lambda\rho\le c_\ast\),
\[
\mu_{\pm,\mathrm{cap}}^\rho(t) =\sum_{n\in\mathbb Z}e^{-int}\lambda^{\beta n}a_{\pm,n}^{\rho,\mathrm{cap}}, \qquad z_{\pm,\mathrm{cap}}^\rho=a_{\pm,0}^{\rho,\mathrm{cap}},
\]
and
\begin{equation}
\label{eq:general_first_scale_coefficients}
a_{\pm,1}^{\rho,\mathrm{cap}} =-\frac{i\eta k|\partial D|}{4\pi R}\rho^2e^{ik\lambda R} F_D(k\lambda\rho\,\omega)d_\pm(kR) +O\bigl(\rho^2(\rho+\lambda^{-1})\bigr),
\end{equation}
where
\[
d_+(s):=e^{-is}+(M-1)e^{is/(M-1)}, \qquad d_-(s):=e^{-is}-e^{is/(M-1)}.
\]
Here, \(d_+\) and \(d_-\) are the symmetric- and sum-zero-sector interference factors, respectively.
For the same-scale off-diagonal coefficient,
\begin{equation}
\label{eq:general_same_scale_capacitance_coefficient}
c_{ij,0}^{\rho,\mathrm{cap}} =-i\eta k|\partial D|\rho^2G_k(p_i,p_j)+O(\rho^3), \qquad i\ne j.
\end{equation}
Consequently, \(|z_{+,\mathrm{cap}}^\rho-z_{-,\mathrm{cap}}^\rho|\ge c\rho^2\) for all sufficiently small \(\rho\). Moreover,
\begin{equation}
\label{eq:general_capacitance_tail_coefficients}
|a_{\pm,n}^{\rho,\mathrm{cap}}|\le C\rho^2\quad(n\ge2), \qquad |a_{\pm,-m}^{\rho,\mathrm{cap}}|\le C\rho^2\lambda^{-m}\quad(m\ge1).
\end{equation}
\end{lemma}

\begin{proof}
Put \(L=\lambda^n\), write \(x=p_i+\rho O_i\widehat x\), \(y=p_j+\rho O_j\widehat y\), and set \(\omega_j:=O_j^{\mathsf T}p_j/R\). For \(n=1\), the large-distance expansion, uniform under \(k\lambda\rho\le c_\ast\), is
\[
\left|p_i+\rho O_i\widehat x-\lambda(p_j+\rho O_j\widehat y)\right|=\lambda R+\lambda\rho\,\omega_j\cdot\widehat y-\frac{p_i\cdot p_j}{R} +O(\lambda\rho^2+\rho+\lambda^{-1}).
\]
Since \(k\lambda\rho\le c_\ast\), the error in the phase is \(O(\rho+\lambda^{-1})\). It follows that
\[
\lambda G_k\bigl(p_i+\rho O_i\widehat x,\lambda(p_j+\rho O_j\widehat y)\bigr)=\frac{e^{ik\lambda R}}{4\pi R}e^{ik\lambda\rho\,\omega_j\cdot\widehat y}e^{-ikp_i\cdot p_j/R}+O(\rho+\lambda^{-1}).
\]
Integrating the source density produces \(\operatorname{Cap}(D)F_D(k\lambda\rho\,\omega_j)\), which is equal to the common equivariant value \(\operatorname{Cap}(D)F_D(k\lambda\rho\,\omega)\) and cancels the capacity in \(\Lambda_i^\rho\); integrating the target variable produces \(|\partial D|\). Therefore, the single-layer part of the combined kernel gives the leading term in \eqref{eq:general_first_scale_coefficients}. For the double-layer part, a source point \(z\) lies outside \(D_i^\rho\), and
\[
\int_{\Gamma_i^\rho}\partial_{\nu_x}G_k(x,z)\,d\sigma(x)=-k^2\int_{D_i^\rho}G_k(x,z)\,dx.
\]
At \(n=1\), the right-hand side is \(O(\rho^3\lambda^{-1})\). The source density has mass \(O(\rho)\); after the scale-normalization factor \(\lambda\) and the normalization by \(C_i^\rho=O(\rho)\), the double-layer contribution is \(O(\rho^3)\), which is absorbed by the remainder in \eqref{eq:general_first_scale_coefficients}.
The simplex identities in \eqref{eq:simplex_vertices_normalization} give the phases \(e^{-ikR}\) when \(i=j\) and \(e^{ikR/(M-1)}\) when \(i\ne j\), proving \eqref{eq:general_first_scale_coefficients} with the stated sign.

For \(n=0\) and \(i\ne j\), a Taylor expansion on the two separated components gives
\[
\Lambda_i^\rho(S_k^{\Gamma_j^\rho}\psi_j^\rho)=|\partial D|\rho^2G_k(p_i,p_j)+O(\rho^3),
\]
while the averaged target normal derivative is \(O(\rho^3)\). This proves \eqref{eq:general_same_scale_capacitance_coefficient}. Since \(z_{+,\mathrm{cap}}^\rho-z_{-,\mathrm{cap}}^\rho=M c_{ij,0}^{\rho,\mathrm{cap}}\) and \(G_k(p_i,p_j)\ne0\), the separation follows.

For \(n\ne0\), and also for \(n=0\) when \(i\ne j\), both boundaries are separated and the exact pulled-back coefficient formula is
\[
c_{ij,n}^{\rho,\mathrm{cap}}=\frac{L\rho^2}{\operatorname{Cap}(D)}\int_{\partial D}\!\int_{\partial D}\mathsf K_\eta^k\bigl(p_i+\rho O_i\widehat x,L(p_j+\rho O_j\widehat y)\bigr)\psi_D(\widehat y)\,d\sigma(\widehat y)d\sigma(\widehat x).
\]
For \(n\ge2\), separation from the enlarged source gives
\[
L\left|\mathsf K_\eta^k\bigl(p_i+\rho O_i\widehat x, L(p_j+\rho O_j\widehat y)\bigr)\right|\le C.
\]
For \(n=-m\), the contracted source remains a fixed distance from the target and
\[
L\left|\mathsf K_\eta^k\bigl(p_i+\rho O_i\widehat x, L(p_j+\rho O_j\widehat y)\bigr)\right|\le C\lambda^{-m}.
\]
The equilibrium density has a fixed \(L^1(\partial D)\)-norm, so these inequalities prove \eqref{eq:general_capacitance_tail_coefficients}.
\end{proof}

The next theorem converts these estimates into the point-gap winding of the projected matrix family.

\begin{theorem}[Winding of the projected BFZ matrix for a regular simplex]
\label{thm:general_capacitance_simplex_winding}
Fix \(M\in\{2,3,4\}\), \(k,R,\eta>0\), and an equivariant congruent-component configuration as above. Assume
\[
0<kR<\frac{\pi}{2},
\qquad
\beta\in\left(-\frac12,0\right),
\qquad
\inf_{0\le s\le c_\ast}|F_D(s\omega)|>0
\]
for some \(c_\ast>0\). Then there are \(\rho_0>0\), \(\lambda_0>1\), and \(c>0\) such that, whenever \(0<\rho\le\rho_0\), \(\lambda\ge\lambda_0\), and \(k\lambda\rho\le c_\ast\), the matrix \(\mathsf C_{\beta,\mathrm{cap}}^\rho\) has point gaps at \(z_{+,\mathrm{cap}}^\rho\) and \(z_{-,\mathrm{cap}}^\rho\), its two shifted sector loops both have  winding number \(-1\), and
\begin{equation}
\label{eq:general_capacitance_winding_conclusion}
\operatorname{Wind}\left(\det\bigl(\mathsf C_{\beta,\mathrm{cap}}^\rho(\cdot)-z_{+,\mathrm{cap}}^\rho I_M\bigr),0\right)=-1,
\qquad
\operatorname{Wind}\left(\det\bigl(\mathsf C_{\beta,\mathrm{cap}}^\rho(\cdot)-z_{-,\mathrm{cap}}^\rho I_M\bigr),0\right)=-(M-1).
\end{equation}
Moreover,
\begin{equation}
\label{eq:general_capacitance_reference_gap}
\inf_t\sigma_{\min}\bigl(\mathsf C_{\beta,\mathrm{cap}}^\rho(t)-z_{\pm,\mathrm{cap}}^\rho I_M\bigr) \ge c\rho^2\lambda^\beta.
\end{equation}
The nonvanishing hypothesis on \(F_D\) is automatic, for example, if \(c_\ast\sup_{y\in D}|y|<1\).
\end{theorem}

\begin{proof}
For \(0<kR<\pi/2\), both \(d_+(kR)\) and \(d_-(kR)\) are nonzero. Indeed, for \(M\ge3\) the two summands in \(d_+\) have different moduli, while for \(M=2\), \(d_+(kR)=2\cos(kR)\); moreover, \(d_-(kR)=0\) would imply \(MkR/(M-1)\in2\pi\mathbb Z\). The lower bound on \(F_D\) and \eqref{eq:general_first_scale_coefficients} therefore give
\[
|a_{\pm,1}^{\rho,\mathrm{cap}}|\ge c_\pm\rho^2
\]
after decreasing \(\rho_0\) and increasing \(\lambda_0\).

Since \(z_{\pm,\mathrm{cap}}^\rho=a_{\pm,0}^{\rho,\mathrm{cap}}\), write
\[
\mu_{\pm,\mathrm{cap}}^\rho(t)-z_{\pm,\mathrm{cap}}^\rho =e^{-it}\lambda^\beta a_{\pm,1}^{\rho,\mathrm{cap}}+R_\pm^\rho(t).
\]
The tail estimates in \eqref{eq:general_capacitance_tail_coefficients} yield
\[
\sup_t|R_\pm^\rho(t)| \le C\rho^2\left( \frac{\lambda^{2\beta}}{1-\lambda^\beta} +\frac{\lambda^{-(1+\beta)}}{1-\lambda^{-(1+\beta)}} \right).
\]
After division by \(\lambda^\beta|a_{\pm,1}^{\rho,\mathrm{cap}}|\), the right-hand side is bounded by a constant times
\[
\frac{\lambda^\beta}{1-\lambda^\beta}+\frac{\lambda^{-1-2\beta}}{1-\lambda^{-(1+\beta)}},
\]
which tends to zero because \(\beta\in(-1/2,0)\). Each shifted sector loop is consequently homotopic in \(\mathbb C\setminus\{0\}\) to \(e^{-it}\lambda^\beta a_{\pm,1}^{\rho,\mathrm{cap}}\), and therefore has winding number \(-1\). The same dominance gives
\[
\inf_t|\mu_{\pm,\mathrm{cap}}^\rho(t)-z_{\pm,\mathrm{cap}}^\rho| \ge c\rho^2\lambda^\beta.
\]
For either sector, the entire nonzero-scale part also satisfies
\[
\sup_t|\mu_{\pm,\mathrm{cap}}^\rho(t)-z_{\pm,\mathrm{cap}}^\rho| \le C\rho^2\left( \frac{\lambda^\beta}{1-\lambda^\beta} +\frac{\lambda^{-(1+\beta)}}{1-\lambda^{-(1+\beta)}} \right)=o(\rho^2).
\]
At either reference point, this is dominated by
\(|z_{+,\mathrm{cap}}^\rho-z_{-,\mathrm{cap}}^\rho|\ge c\rho^2\), so the complementary sector remains at distance at least \(c\rho^2\) and has winding number zero. The two symmetry sectors are orthogonal, and the matrix is scalar on each of them; since \(\lambda^\beta\le1\), taking the smaller sector distance proves \eqref{eq:general_capacitance_reference_gap}. Finally,
\[
\det\bigl(\mathsf C_{\beta,\mathrm{cap}}^\rho(t)-zI_M\bigr) =\bigl(\mu_{+,\mathrm{cap}}^\rho(t)-z\bigr) \bigl(\mu_{-,\mathrm{cap}}^\rho(t)-z\bigr)^{M-1}.
\]
The two sector windings therefore prove \eqref{eq:general_capacitance_winding_conclusion}. The final sufficient condition follows from \eqref{eq:equilibrium_form_factor_near_zero}.
\end{proof}

The hypothesis in Theorem~\ref{thm:general_capacitance_simplex_winding} on \(F_D\) is not restricted to spheres. For example, if \(D=\mathsf A B_1\) with \(\mathsf A\in\mathrm{GL}(3,\mathbb R)\), its normalized equilibrium measure is the pushforward of normalized surface measure on \(\mathbb S^2\) under \(y=\mathsf A\omega\). Consequently, \(F_D(\xi)=j_0(|\mathsf A^{\mathsf T}\xi|)\), where \(j_0(s)=\sin(s)/s\) \cite{ebenfelt_khavinson_shapiro2002}. Hence, the hypothesis holds whenever \(c_\ast|\mathsf A^{\mathsf T}\omega|<\pi\).

For use in the Riesz comparison, we record the spherical specialization in the normalization used below. We let
\[
\Gamma_i^\rho:=\partial B_\rho(p_i), \qquad \Gamma^\rho:=\bigsqcup_{i=1}^M\Gamma_i^\rho, \qquad A_\rho:=4\pi\rho^2,
\]
and retain the abbreviations \(B_n^{k,\rho}:=B_n^{k,\Gamma^\rho}\) and \(\mathcal T^\rho(t):=\mathcal T_{k,\beta}^{\Gamma^\rho}(t)\). Since
\[
\psi_i^\rho=\rho^{-1}\mathbf1_{\Gamma_i^\rho}, \qquad \operatorname{Cap}(B_\rho)=4\pi\rho, \qquad F_{B_1}(\xi)=j_0(|\xi|),
\]
the vectors
\[
e_i^\rho:=\frac{\mathbf1_{\Gamma_i^\rho}}{\sqrt{kA_\rho}} =\frac{\psi_i^\rho}{\sqrt{4\pi k}}, \qquad i=1,\ldots,M,
\]
form an orthonormal basis of \(E_{\mathrm{cap}}^\rho\) in \(\mathcal H_k(\Gamma^\rho)\), and \(P_{\mathrm{cap}}^\rho\) is the corresponding orthogonal projection. In this basis, the projected operator
\[
\mathsf C_\beta^\rho(t):= P_{\mathrm{cap}}^\rho\mathcal T^\rho(t)\big|_{E_{\mathrm{cap}}^\rho}
\]
is exactly the matrix \(\mathsf C_{\beta,\mathrm{cap}}^\rho(t)\) defined in \eqref{eq:general_capacitance_coordinate_matrix}, because
\[
\left\langle e_i^\rho,\mathcal T^\rho(t)e_j^\rho\right\rangle_{\mathcal H_k(\Gamma^\rho)} =\Lambda_i^\rho\bigl(\mathcal T^\rho(t)\psi_j^\rho\bigr).
\]

Set
\[
e_+^\rho:=\sum_{i=1}^M e_i^\rho, \qquad E_+^\rho:=\operatorname{span}\{e_+^\rho\}, \qquad E_-^\rho:=\left\{\sum_{i=1}^M c_i e_i^\rho:\sum_{i=1}^M c_i=0\right\}.
\]
Then \(E_{\mathrm{cap}}^\rho=E_+^\rho\oplus E_-^\rho\), and simplex symmetry gives scalar loops \(\mu_\pm^\rho\) and same-scale values \(z_\pm^\rho\) characterized by
\begin{equation}
\mathsf C_\beta^\rho(t)\big|_{E_\pm^\rho} = \mu_\pm^\rho(t)I_{E_\pm^\rho},
 \qquad
P_{\mathrm{cap}}^\rho B_0^{k,\rho}\big|_{E_\pm^\rho}=z_\pm^\rho I_{E_\pm^\rho}.
\label{eq:simplex_same_scale_values}
\end{equation}
Thus, \(\mu_\pm^\rho-z_\pm^\rho\) contain only the nonzero-scale BFZ terms. Since \(F_{B_1}(s\omega)=j_0(s)>0\) for \(0\le s\le c_\ast<\pi\), Theorem~\ref{thm:general_capacitance_simplex_winding} applies directly to this projected matrix family.

\subsection{Comparison with the invariant Riesz reduction}
\label{subsec:isolated_cluster_component_constant_winding}

The preceding winding is computed for the fixed matrix family obtained by projecting onto the static kernel. We now prove that, under the same small-component and point-gap conditions, it agrees with the winding of the exact restriction to an invariant spectral subspace of the full BFZ boundary symbol. For a general shape, the Riesz projection of the block-diagonal self-interaction operator at nonzero wavenumber maps the static kernel isomorphically onto its nearby spectral subspace; a second Riesz projection selects the corresponding invariant subspace of the full boundary symbol. For spheres, the first Riesz range is the static kernel itself. We retain the abstract notation of Subsection~\ref{subsec:finite_dimensional_scale_bloch_winding} without redefining those objects.

Writing \(\iota_E:E\hookrightarrow\mathcal H_\Gamma\) for inclusion, the previously defined frame of the Riesz range is \(U_E^\Gamma=\Pi^\Gamma\iota_E\). Whenever it is an isomorphism onto \(\operatorname{Ran}\Pi^\Gamma(t)\), define its dual frame by
\begin{equation}
\label{eq:riesz_biorthogonal_frames}
\bigl(V_E^\Gamma(t)\bigr)^\dagger := \bigl(U_E^\Gamma(t)\bigr)^{-1}\Pi^\Gamma(t).
\end{equation}
Here, \((U_E^\Gamma)^{-1}\) is used only on \(\operatorname{Ran}\Pi^\Gamma\). Since \((\Pi^\Gamma)^2=\Pi^\Gamma\), these frames satisfy the exact identities
\begin{equation}
\label{eq:riesz_biorthogonal_identity}
\bigl(V_E^\Gamma(t)\bigr)^\dagger U_E^\Gamma(t)=I_E, \qquad \Pi^\Gamma(t)=U_E^\Gamma(t)\bigl(V_E^\Gamma(t)\bigr)^\dagger,
\end{equation}
and, because \(\mathcal T_{k,\beta}^\Gamma(t)\) commutes with its Riesz projection,
\begin{equation}
\label{eq:riesz_biorthogonal_reduction}
M_E^\Gamma(t)=\bigl(V_E^\Gamma(t)\bigr)^\dagger\mathcal T_{k,\beta}^\Gamma(t)U_E^\Gamma(t).
\end{equation}
Thus, the Riesz reduction has an exact biorthogonal frame representation, although the frame vectors need not be individual eigenvectors.

The next proposition quantifies this comparison: the fixed-space projected family and the exact Riesz-reduced family differ quadratically in the perturbation, whereas their frames vary linearly.

\begin{proposition}[Quadratic Riesz reduction and frame approximation]
\label{prop:finite_cluster_reduction_nearby_loops}
Set
\[
\varepsilon:=\sup_{0\le t\le2\pi} \|\mathcal T_{k,\beta}^\Gamma(t)-\mathcal T_{\mathrm{ref}}^\Gamma\|_{\mathcal L(\mathcal H_\Gamma)}.
\]
There exist \(\varepsilon_0>0\) and \(C<\infty\), depending only on \(\mathcal T_{\mathrm{ref}}^\Gamma\), \(\Pi^{\mathrm{ref}}\), and \(\mathscr C\), such that, if \(\varepsilon\le\varepsilon_0\), then \(\mathscr C\) remains in the resolvent set and encloses a spectral cluster whose Riesz projection has rank \(\dim E\), \(U_E^\Gamma(t)\) is an isomorphism for every \(t\), and
\begin{equation}
\label{eq:riesz_frame_approximation}
\sup_t\left(\|\Pi^\Gamma(t)-\Pi^{\mathrm{ref}}\|+\|U_E^\Gamma(t)-\iota_E\|+\|V_E^\Gamma(t)-(\Pi^{\mathrm{ref}})^\dagger\iota_E\|\right) \le C\varepsilon.
\end{equation}
Moreover,
\begin{equation}
\label{eq:finite_cluster_reduction_expansion}
\sup_t\left\|M_E^\Gamma(t)-\Pi^{\mathrm{ref}}\mathcal T_{k,\beta}^\Gamma(t)|_E\right\|_{\mathcal L(E)}\le C\varepsilon^2.
\end{equation}
If
\[
M_E^\Gamma(t)r=\nu r, \qquad \bigl(M_E^\Gamma(t)\bigr)^\dagger\ell=\overline\nu\,\ell, \qquad \ell^\dagger r=1,
\]
then \(U_E^\Gamma(t)r\) and \(V_E^\Gamma(t)\ell\) are exact right and left eigenvectors of \(\mathcal T_{k,\beta}^\Gamma(t)\), are biorthogonally normalized, and satisfy
\[
\|(U_E^\Gamma(t)-\iota_E)r\|\le C\varepsilon\|r\|,
\qquad
\|\bigl(V_E^\Gamma(t)-(\Pi^{\mathrm{ref}})^\dagger\iota_E\bigr)\ell\| \le C\varepsilon\|\ell\|.
\]
\end{proposition}

\begin{proof}
Fix \(t\) and write the perturbed operator as \(\mathcal T_{\mathrm{ref}}^\Gamma+Q\). Differentiating the resulting fixed-space reduction at \(Q=0\) identifies its first-order term. For a fixed \(Q\in\mathcal L(\mathcal H_\Gamma)\), let
\[
\Pi_Q:=\frac{1}{2\pi i}\int_{\mathscr C}(\zeta I-\mathcal T_{\mathrm{ref}}^\Gamma-Q)^{-1}\,d\zeta, \qquad U_Q:=\Pi_Q|_E, \qquad L_Q:=\Pi^{\mathrm{ref}}\Pi_Q|_E.
\]
The resolvent Neumann series is uniform on \(\mathscr C\) for \(\|Q\|\) small. Hence \(\Pi_Q=\Pi^{\mathrm{ref}}+O(\|Q\|)\), \(U_Q\) is an isomorphism onto \(\operatorname{Ran}\Pi_Q\), and \(L_Q=I_E+O(\|Q\|)\) is analytically invertible. Since
\[
U_Q^{-1}\Pi_Q=L_Q^{-1}\Pi^{\mathrm{ref}}\Pi_Q,
\]
the reduced operator is the fixed-space analytic family
\[
\mathsf M(Q) :=L_Q^{-1}\Pi^{\mathrm{ref}} (\mathcal T_{\mathrm{ref}}^\Gamma+Q)\Pi_Q|_E.
\]
At \(Q=0\),
\[
\mathsf M(0)=\Pi^{\mathrm{ref}} \mathcal T_{\mathrm{ref}}^\Gamma|_E, \qquad D\mathsf M(0)[Q]=\Pi^{\mathrm{ref}} Q|_E.
\]
Indeed, let \(D\Pi(0)[Q]\) denote the derivative at zero of the map
\(Q\mapsto\Pi_Q\). Differentiating \(\Pi_Q^2=\Pi_Q\) gives
\[
\Pi^{\mathrm{ref}}D\Pi(0)[Q]\Pi^{\mathrm{ref}}=0.
\]
Since \(U_Q=\Pi_Q|_E\) and \(E=\operatorname{Ran}\Pi^{\mathrm{ref}}\), this implies
\[
\Pi^{\mathrm{ref}}DU(0)[Q]=0.
\]
Differentiate
\((\mathcal T_{\mathrm{ref}}^\Gamma+Q)U_Q=U_Q\mathsf M(Q)\),
apply \(\Pi^{\mathrm{ref}}\), and use the commutation of
\(\mathcal T_{\mathrm{ref}}^\Gamma\) with its Riesz projection. This yields the stated formula for \(D\mathsf M(0)[Q]\), and therefore
\[
\mathsf M(Q)=\Pi^{\mathrm{ref}}(\mathcal T_{\mathrm{ref}}^\Gamma+Q)|_E+O(\|Q\|^2).
\]
Taking \(Q=\mathcal T_{k,\beta}^\Gamma(t)-\mathcal T_{\mathrm{ref}}^\Gamma\) proves \eqref{eq:finite_cluster_reduction_expansion}.

The resolvent estimate also gives \(\Pi^\Gamma-\Pi^{\mathrm{ref}}=O(\varepsilon)\), and hence \(U_E^\Gamma-\iota_E=O(\varepsilon)\). With \(Q=\mathcal T_{k,\beta}^\Gamma(t)-\mathcal T_{\mathrm{ref}}^\Gamma\), the fixed-space formula \((V_E^\Gamma)^\dagger=L_Q^{-1}\Pi^{\mathrm{ref}}\Pi^\Gamma\) gives \(V_E^\Gamma-(\Pi^{\mathrm{ref}})^\dagger\iota_E=O(\varepsilon)\), proving \eqref{eq:riesz_frame_approximation}. Finally, \eqref{eq:riesz_biorthogonal_identity}-\eqref{eq:riesz_biorthogonal_reduction} give
\[
\mathcal T_{k,\beta}^\Gamma U_E^\Gamma=U_E^\Gamma M_E^\Gamma, \qquad (\mathcal T_{k,\beta}^\Gamma)^\dagger V_E^\Gamma =V_E^\Gamma(M_E^\Gamma)^\dagger,
\]
which proves the eigenvector assertions.
\end{proof}

We now return to a general reference component \(D\) and compare its static kernel with the nearby range of the Riesz projection. To remove the shrinking geometry from the operator estimates, pull every component back to \(\partial D\). Set
\[
X_\rho:=\bigoplus_{i=1}^M H^{1/2}(\Gamma_i^\rho).
\]
For \(f=(f_1,\ldots,f_M)\in X_\rho\), let
\[
(\mathscr U_i^\rho f_i)(y):=f_i(p_i+\rho O_i y), \qquad y\in\partial D,
\]
and use the pullback Hilbert norm
\begin{equation}
\label{eq:general_capacitance_pullback_norm}
\|f\|_{X_\rho}^2:=\sum_{i=1}^M \|\mathscr U_i^\rho f_i\|_{H^{1/2}(\partial D)}^2.
\end{equation}
In particular,
\[
\left\|\sum_{i=1}^M c_i\psi_i^\rho\right\|_{X_\rho}=\rho^{-1}\|\psi_D\|_{H^{1/2}(\partial D)}|c|_{\ell^2},
\]
so identification through the equilibrium-density basis has a condition number independent of \(\rho\). Unless stated otherwise, all operator norms and Hilbert adjoints in the following general-shape argument, including those on its finite-dimensional subspaces, are induced by \(X_\rho\).

To isolate the nonzero-wavenumber correction before including intercomponent and interscale interactions, let
\[
B_{\mathrm{self}}^\rho:=\bigoplus_{i=1}^M B_0^{k,\Gamma_i^\rho},
\]
and choose a positively oriented, shape-dependent contour \(\mathscr C_D\) enclosing zero and no other spectral point of \(\frac12I+(K_0^{\partial D})^\ast\). Thus, \(\mathscr C_D\) isolates the static zero cluster after pullback to \(\partial D\).

The next lemma replaces the static equilibrium-density space by the nearby invariant self-interaction cluster and compares their projected BFZ families.

\begin{lemma}[Static kernel and the near-zero self-interaction cluster]
\label{lem:general_capacitance_cluster_coupling}
Fix \(k,\eta>0\) and \(\beta\in(-1,0)\). There exist \(\rho_\ast>0\), \(\lambda_\ast>1\), and \(C<\infty\) such that, whenever \(0<\rho\le\rho_\ast\) and \(\lambda\ge\lambda_\ast\), the following statements hold.

The contour \(\mathscr C_D\) encloses an \(M\)-dimensional spectral cluster of \(B_{\mathrm{self}}^\rho\). Its Riesz projection
\[
P_{\mathrm{self}}^\rho:=\frac{1}{2\pi i}\int_{\mathscr C_D}(\zeta I-B_{\mathrm{self}}^\rho)^{-1}\,d\zeta
\]
satisfies
\begin{equation}
\label{eq:general_self_capacitance_projection_comparison}
\|P_{\mathrm{self}}^\rho-P_{\mathrm{cap}}^\rho\|_{\mathcal L(X_\rho)}\le Ck\rho,
\end{equation}
and
\begin{equation}
\label{eq:general_self_cluster_resolvent}
\|B_{\mathrm{self}}^\rho\|+\sup_{\zeta\in\mathscr C_D}\|(\zeta I-B_{\mathrm{self}}^\rho)^{-1}\|\le C.
\end{equation}

The intercomponent and interscale interactions satisfy
\begin{equation}
\label{eq:general_bfz_self_coupling}
\sup_{0\le t\le2\pi}\left(\|\mathcal T^\rho(t)-B_{\mathrm{self}}^\rho\|+\|\partial_t\mathcal T^\rho(t)\|\right)\le C\rho^2.
\end{equation}

Set
\[
E_{\mathrm{self}}^\rho:=\operatorname{Ran}P_{\mathrm{self}}^\rho,\qquad \mathcal J^\rho:=P_{\mathrm{self}}^\rho|_{E_{\mathrm{cap}}^\rho}:E_{\mathrm{cap}}^\rho\longrightarrow E_{\mathrm{self}}^\rho.
\]
Then \(\mathcal J^\rho\) is an isomorphism. Using the equilibrium-density basis to identify \(E_{\mathrm{cap}}^\rho\) with \(\mathbb C^M\), define
\[
\mathsf C_{\beta,\mathrm{self}}^\rho(t):=(\mathcal J^\rho)^{-1}P_{\mathrm{self}}^\rho\mathcal T^\rho(t)|_{E_{\mathrm{self}}^\rho}\mathcal J^\rho.
\]
There exists \(q_\rho\in\mathbb C\), independent of \(t\), such that
\begin{equation}
\label{eq:general_static_dynamic_matrix_comparison}
|q_\rho|\le C(k\rho)^2,\qquad \sup_{0\le t\le2\pi}\|\mathsf C_{\beta,\mathrm{self}}^\rho(t)-\mathsf C_{\beta,\mathrm{cap}}^\rho(t)-q_\rho I_M\|\le Ck\rho^3.
\end{equation}
\end{lemma}

\begin{proof}
\emph{The self-interaction cluster.}
Under the pullback \(\mathscr U_i^\rho\), the static projection is independent of \(\rho\):
\[
\mathscr U_i^\rho P_{\mathrm{cap}}^\rho(\mathscr U_i^\rho)^{-1}g=\psi_D\frac{\displaystyle\int_{\partial D}g\,d\sigma}{\operatorname{Cap}(D)}.
\]
By Corollary~\ref{cor:static_capacitance_pairing}, this is the Riesz projection at zero of \(\frac12I+(K_0^{\partial D})^\ast\).

Set \(\kappa:=k\rho\). On the fixed boundary \(\partial D\), the standard low-wavenumber kernel expansions give
\[
\|(K_\kappa^{\partial D})^\ast-(K_0^{\partial D})^\ast\|_{\mathcal L(H^{1/2}(\partial D))}+\|\kappa(S_\kappa^{\partial D}-S_0^{\partial D})\|_{\mathcal L(H^{1/2}(\partial D))}\le C\kappa^2.
\]
Hence, with \(A_\kappa:=\frac12I+(K_\kappa^{\partial D})^\ast-i\eta\kappa S_\kappa^{\partial D}\) and \(A_0:=\frac12I+(K_0^{\partial D})^\ast\),
\[
A_\kappa=A_0-i\eta\kappa S_0^{\partial D}+O_{\mathcal L(H^{1/2}(\partial D))}(\kappa^2),
\]
so \(\|A_\kappa-A_0\|\le C\kappa\). Since \(\mathscr C_D\subset\rho(A_0)\), the Neumann series and the resolvent identity imply, for sufficiently small \(\kappa\),
\[
\sup_{\zeta\in\mathscr C_D}\|(\zeta I-A_\kappa)^{-1}\|\le C,\qquad \sup_{\zeta\in\mathscr C_D}\|(\zeta I-A_\kappa)^{-1}-(\zeta I-A_0)^{-1}\|\le C\kappa.
\]
Integrating the second estimate over \(\mathscr C_D\) and taking the direct sum over the components gives
\[
\|P_{\mathrm{self}}^\rho-P_{\mathrm{cap}}^\rho\|_{\mathcal L(X_\rho)}\le Ck\rho.
\]
The rank of a Riesz projection is locally constant under norm-continuous perturbations; hence, \(\mathscr C_D\) encloses exactly one eigenvalue, counted algebraically, from each component, and the enclosed cluster has dimension \(M\). The preceding bounds also give \eqref{eq:general_self_cluster_resolvent}.

\emph{The BFZ coupling.}
Let \(B_{ij,n}^{k,\rho}\) denote the block from component \(j\) to component \(i\). If \(n\neq0\), or if \(n=0\) and \(i\neq j\), separation gives
\[
\mathscr U_i^\rho B_{ij,n}^{k,\rho}(\mathscr U_j^\rho)^{-1}g(x)=\rho^2\int_{\partial D}L_{ij,n}^\rho(x,y)g(y)\,d\sigma(y).
\]
Annular separation and scale normalization imply
\[
\sup_{n\ge1}\sup_{|\alpha|\le1}\|\partial_x^\alpha L_{ij,n}^\rho\|_{L^\infty}\le C,\qquad \sup_{|\alpha|\le1}\|\partial_x^\alpha L_{ij,-m}^\rho\|_{L^\infty}\le C\lambda^{-m}\quad(m\ge1),
\]
and the first estimate also holds when \(n=0\) and \(i\neq j\). The corresponding \(H^1\)-output estimate, together with \(H^{1/2}\hookrightarrow L^2\) and \(H^1\hookrightarrow H^{1/2}\), yields
\[
\|B_{ij,n}^{k,\rho}\|_{\mathcal L(X_\rho)}\le C\rho^2\quad(n\ge1),\qquad \|B_{ij,-m}^{k,\rho}\|_{\mathcal L(X_\rho)}\le C\rho^2\lambda^{-m}\quad(m\ge1),
\]
while \(\|B_{ij,0}^{k,\rho}\|_{\mathcal L(X_\rho)}\le C\rho^2\) for \(i\neq j\). Since \(\beta\in(-1,0)\), the sums
\[
\sum_{n\ge1}(1+n)\lambda^{\beta n},\qquad \sum_{m\ge1}(1+m)\lambda^{-(1+\beta)m}
\]
are uniformly bounded for \(\lambda\ge\lambda_\ast\). The terms without the factors \(n\) and \(m\) control \(\mathcal T^\rho(t)-B_{\mathrm{self}}^\rho\), while those containing these factors control \(\partial_t\mathcal T^\rho(t)\). This proves \eqref{eq:general_bfz_self_coupling}.

\emph{Identification of the spectral subspaces.}
Let \(\iota_{\mathrm{cap}}^\rho:E_{\mathrm{cap}}^\rho\hookrightarrow X_\rho\) denote the inclusion. For sufficiently small \(\rho\), \eqref{eq:general_self_capacitance_projection_comparison} gives
\[
\|P_{\mathrm{self}}^\rho-P_{\mathrm{cap}}^\rho\|<1.
\]
If \(u\in E_{\mathrm{cap}}^\rho\), then
\[
\|\mathcal J^\rho u\|=\|P_{\mathrm{self}}^\rho u\|\ge\left(1-\|P_{\mathrm{self}}^\rho-P_{\mathrm{cap}}^\rho\|\right)\|u\|.
\]
Thus, \(\mathcal J^\rho\) is injective. Its domain and codomain both have dimension \(M\), so it is an isomorphism with uniformly bounded inverse.

The identities
\[
\mathcal J^\rho-\iota_{\mathrm{cap}}^\rho=(P_{\mathrm{self}}^\rho-P_{\mathrm{cap}}^\rho)\iota_{\mathrm{cap}}^\rho
\]
and
\[
(\mathcal J^\rho)^{-1}P_{\mathrm{self}}^\rho-P_{\mathrm{cap}}^\rho=(\mathcal J^\rho)^{-1}(P_{\mathrm{self}}^\rho-P_{\mathrm{cap}}^\rho)(I-P_{\mathrm{cap}}^\rho)
\]
gives the intermediate estimate
\[
\|\mathcal J^\rho-\iota_{\mathrm{cap}}^\rho\|+\|(\mathcal J^\rho)^{-1}P_{\mathrm{self}}^\rho-P_{\mathrm{cap}}^\rho\|\le Ck\rho.
\]

\emph{Comparison of the reduced matrices.}
The operators \(B_{\mathrm{self}}^\rho\), \(P_{\mathrm{self}}^\rho\), and \(P_{\mathrm{cap}}^\rho\) are block diagonal, and their diagonal blocks are identical after the component pullbacks. Consequently,
\[
(\mathcal J^\rho)^{-1}B_{\mathrm{self}}^\rho\mathcal J^\rho-P_{\mathrm{cap}}^\rho B_{\mathrm{self}}^\rho|_{E_{\mathrm{cap}}^\rho}=q_\rho I_M
\]
for some scalar \(q_\rho\in\mathbb C\).

Apply Proposition~\ref{prop:finite_cluster_reduction_nearby_loops} to the constant family \(B_{\mathrm{self}}^\rho\), taking \(\bigoplus_iB_0^{0,\Gamma_i^\rho}\) as the reference operator and \(P_{\mathrm{cap}}^\rho\) as its Riesz projection. Its quadratic reduction estimate gives
\[
|q_\rho|\le C\left\|B_{\mathrm{self}}^\rho-\bigoplus_iB_0^{0,\Gamma_i^\rho}\right\|^2\le C(k\rho)^2.
\]
The constant is uniform because, after the component pullbacks, the static reference operator, its Riesz projection, and the contour are independent of \(\rho\).

Finally, set
\[
\Delta^\rho(t):=\mathcal T^\rho(t)-B_{\mathrm{self}}^\rho.
\]
Using the definition of \(q_\rho\), a direct expansion gives
\[
\mathsf C_{\beta,\mathrm{self}}^\rho(t)-\mathsf C_{\beta,\mathrm{cap}}^\rho(t)-q_\rho I_M=\bigl((\mathcal J^\rho)^{-1}P_{\mathrm{self}}^\rho-P_{\mathrm{cap}}^\rho\bigr)\Delta^\rho(t)\mathcal J^\rho+P_{\mathrm{cap}}^\rho\Delta^\rho(t)(\mathcal J^\rho-\iota_{\mathrm{cap}}^\rho).
\]
The two identification errors are \(O(k\rho)\), \(\mathcal J^\rho\) is uniformly bounded, and \eqref{eq:general_bfz_self_coupling} gives \(\|\Delta^\rho(t)\|=O(\rho^2)\) uniformly in \(t\). Therefore
\[
\sup_{0\le t\le2\pi}\|\mathsf C_{\beta,\mathrm{self}}^\rho(t)-\mathsf C_{\beta,\mathrm{cap}}^\rho(t)-q_\rho I_M\|\le Ck\rho^3,
\]
which completes the proof.
\end{proof}

Equivariance makes \(\mathsf C_{\beta,\mathrm{self}}^\rho\) scalar on the same symmetric and sum-zero sectors as \(\mathsf C_{\beta,\mathrm{cap}}^\rho\). Let \(z_{\pm,\mathrm{self}}^\rho\) be the two sector values of its zeroth Fourier coefficient. The matrix comparison above gives
\[
z_{\pm,\mathrm{self}}^\rho-z_{\pm,\mathrm{cap}}^\rho=q_\rho+O(k\rho^3),
\]
and, for either sign,
\begin{equation}
\label{eq:general_centered_full_matrix_comparison}
\sup_t\left\|\bigl(\mathsf C_{\beta,\mathrm{self}}^\rho(t)-z_{\pm,\mathrm{self}}^\rho I_M\bigr)-\bigl(\mathsf C_{\beta,\mathrm{cap}}^\rho(t)-z_{\pm,\mathrm{cap}}^\rho I_M\bigr)\right\|\le Ck\rho^3.
\end{equation}
In particular,
\[
(z_{+,\mathrm{self}}^\rho-z_{-,\mathrm{self}}^\rho)-(z_{+,\mathrm{cap}}^\rho-z_{-,\mathrm{cap}}^\rho)=O(k\rho^3).
\]
The estimate \(q_\rho=O((k\rho)^2)\) does not ensure that this common shift is smaller than the point-gap lower bound \(O(\rho^2\lambda^\beta)\). This is why the general result uses the nonzero-wavenumber reference points \(z_{\pm,\mathrm{self}}^\rho\), defined as the sector values of the zeroth Fourier coefficient above.

We now include the \(O(\rho^2)\) intercomponent and interscale interactions. Let \(\Pi^\rho(t)\) denote the Riesz projection of the full BFZ boundary symbol on the same contour:
\[
\Pi^\rho(t):=\frac{1}{2\pi i}\int_{\mathscr C_D}(\zeta I-\mathcal T^\rho(t))^{-1}\,d\zeta,
\]
so that its range is the corresponding invariant spectral subspace. Using \(E_{\mathrm{self}}^\rho\) as the fixed reference space, let \(\iota_{\mathrm{self}}^\rho:E_{\mathrm{self}}^\rho\hookrightarrow X_\rho\) be the inclusion and define the frame and its dual by
\[
U_{\mathrm{self}}^\rho(t):=\Pi^\rho(t)\iota_{\mathrm{self}}^\rho, \qquad \bigl(V_{\mathrm{self}}^\rho(t)\bigr)^\dagger :=\bigl(U_{\mathrm{self}}^\rho(t)\bigr)^{-1}\Pi^\rho(t).
\]
Transporting the invariant restriction through \(\mathcal J^\rho\) gives the following matrix on the fixed space \(E_{\mathrm{cap}}^\rho\):
\[
\widehat{\mathsf C}_{\beta}^\rho(t):=(\mathcal J^\rho)^{-1}\bigl(U_{\mathrm{self}}^\rho(t)\bigr)^{-1}\mathcal T^\rho(t)|_{\operatorname{Ran}\Pi^\rho(t)}U_{\mathrm{self}}^\rho(t)\mathcal J^\rho.
\]
Thus, \(\widehat{\mathsf C}_{\beta}^\rho\) is similar to the restriction of the BFZ symbol to its invariant spectral subspace.

The following corollary transfers the winding of the projected BFZ matrix to the exact Riesz-reduced family.

\begin{corollary}[Winding of the Riesz-reduced family for a regular-simplex configuration]
\label{cor:general_capacitance_riesz_winding}
Under the hypotheses of Theorem~\ref{thm:general_capacitance_simplex_winding}, after decreasing \(\rho_0\) and increasing \(\lambda_0\) if necessary, while retaining \(k\lambda\rho\le c_\ast\), the exact Riesz-reduced family \(\widehat{\mathsf C}_{\beta}^\rho(t)\) is well-defined. Under the equilibrium-density identification of \(E_{\mathrm{cap}}^\rho\) with \(\mathbb C^M\), it is scalar on the symmetric line and the sum-zero subspace; denote the corresponding scalar loops by \(\widehat\mu_+^\rho\) and \(\widehat\mu_-^\rho\).

The reduced family has point gaps at \(z_{+,\mathrm{self}}^\rho\) and \(z_{-,\mathrm{self}}^\rho\), and, for both choices of sign,
\[
\operatorname{Wind}\bigl(\widehat\mu_\pm^\rho-z_{\pm,\mathrm{self}}^\rho,0\bigr)=-1.
\]
Moreover,
\begin{equation}
\label{eq:general_exact_capacitance_windings}
\begin{aligned}
\operatorname{Wind}\left(\det\bigl(\widehat{\mathsf C}_{\beta}^\rho(\cdot)-z_{+,\mathrm{self}}^\rho I_M\bigr),0\right)&=-1,\qquad
\operatorname{Wind}\left(\det\bigl(\widehat{\mathsf C}_{\beta}^\rho(\cdot)-z_{-,\mathrm{self}}^\rho I_M\bigr),0\right)&=-(M-1).
\end{aligned}
\end{equation}
\end{corollary}

\begin{proof}
Set
\[
\Delta^\rho(t):=\mathcal T^\rho(t)-B_{\mathrm{self}}^\rho.
\]
Equations~\eqref{eq:general_self_cluster_resolvent} and \eqref{eq:general_bfz_self_coupling} give
\[
\sup_t\|\Delta^\rho(t)\|\le C\rho^2
\]
together with uniform bounds for \(B_{\mathrm{self}}^\rho\), its Riesz projection \(P_{\mathrm{self}}^\rho\), and its resolvent on \(\mathscr C_D\). Proposition~\ref{prop:finite_cluster_reduction_nearby_loops} therefore holds uniformly with
\[
\mathcal T_{\mathrm{ref}}^\Gamma=B_{\mathrm{self}}^\rho, \qquad \Pi^{\mathrm{ref}}=P_{\mathrm{self}}^\rho, \qquad E=E_{\mathrm{self}}^\rho.
\]
After decreasing \(\rho_0\), the contour \(\mathscr C_D\) remains in the resolvent set of \(\mathcal T^\rho(t)\), the corresponding Riesz projection \(\Pi^\rho(t)\) has rank \(M\), and
\[
U_{\mathrm{self}}^\rho(t)=\Pi^\rho(t)|_{E_{\mathrm{self}}^\rho}
\]
is an isomorphism onto \(\operatorname{Ran}\Pi^\rho(t)\). Hence \(V_{\mathrm{self}}^\rho(t)\) and \(\widehat{\mathsf C}_{\beta}^\rho(t)\) are well-defined.

The frame and quadratic reduction estimates in Proposition~\ref{prop:finite_cluster_reduction_nearby_loops} give
\begin{equation}
\label{eq:general_riesz_frame_estimates}
\sup_t\left( \|\Pi^\rho(t)-P_{\mathrm{self}}^\rho\| +\|U_{\mathrm{self}}^\rho(t)-\iota_{\mathrm{self}}^\rho\| +\|V_{\mathrm{self}}^\rho(t) -(P_{\mathrm{self}}^\rho)^\dagger\iota_{\mathrm{self}}^\rho\|\right) \le C\rho^2
\end{equation}
and, after conjugation by the uniformly bounded maps \(\mathcal J^\rho\) and \((\mathcal J^\rho)^{-1}\),
\begin{equation}
\label{eq:general_exact_dynamic_reduction_comparison}
\sup_t \|\widehat{\mathsf C}_{\beta}^\rho(t) -\mathsf C_{\beta,\mathrm{self}}^\rho(t)\| \le C\rho^4.
\end{equation}

For completeness, \eqref{eq:general_self_capacitance_projection_comparison} and the identities
\[
\begin{aligned}
\Pi^\rho(t)-P_{\mathrm{cap}}^\rho&=\Pi^\rho(t)-P_{\mathrm{self}}^\rho+P_{\mathrm{self}}^\rho-P_{\mathrm{cap}}^\rho,\\
U_{\mathrm{self}}^\rho(t)\mathcal J^\rho-\iota_{\mathrm{cap}}^\rho&=\bigl(U_{\mathrm{self}}^\rho(t)-\iota_{\mathrm{self}}^\rho\bigr)\mathcal J^\rho+\bigl(P_{\mathrm{self}}^\rho-P_{\mathrm{cap}}^\rho\bigr)\iota_{\mathrm{cap}}^\rho
\end{aligned}
\]
yield the intermediate estimate
\[
\sup_t\|\Pi^\rho(t)-P_{\mathrm{cap}}^\rho\|+\sup_t\|U_{\mathrm{self}}^\rho(t)\mathcal J^\rho-\iota_{\mathrm{cap}}^\rho
\|
\le C(k\rho+\rho^2).
\]
Thus, the exact invariant spectral subspace is approximated by combinations of the static equilibrium densities.

The definition of \(\widehat{\mathsf C}_{\beta}^\rho(t)\) gives the exact intertwining identity
\[
\mathcal T^\rho(t)U_{\mathrm{self}}^\rho(t)\mathcal J^\rho=U_{\mathrm{self}}^\rho(t)\mathcal J^\rho\widehat{\mathsf C}_{\beta}^\rho(t).
\]
For every simplex symmetry \(Q_\sigma\), the orthogonal invariance of the Green function and the surface measure imply that the induced action \(\mathcal R_\sigma\) commutes with \(B_{\mathrm{self}}^\rho\) and \(\mathcal T^\rho(t)\), and hence with \(P_{\mathrm{self}}^\rho\) and \(\Pi^\rho(t)\). Consequently, \(\mathcal J^\rho\) and \(U_{\mathrm{self}}^\rho(t)\) intertwine the corresponding actions, so \(\widehat{\mathsf C}_{\beta}^\rho(t)\) commutes with the simplex permutation representation on \(E_{\mathrm{cap}}^\rho\). Schur's lemma then shows that it is scalar on the symmetric line and the irreducible sum-zero subspace, with scalar loops \(\widehat\mu_+^\rho\) and \(\widehat\mu_-^\rho\). The intertwining identity also shows that the corresponding vectors transported by \(U_{\mathrm{self}}^\rho(t)\mathcal J^\rho\) are exact right eigenvectors of \(\mathcal T^\rho(t)\).

It remains to transfer the point gaps and windings. Define
\[
\begin{aligned}
A_{\pm,\mathrm{cap}}^\rho(t)&:=\mathsf C_{\beta,\mathrm{cap}}^\rho(t)-z_{\pm,\mathrm{cap}}^\rho I_M,\quad
A_{\pm,\mathrm{self}}^\rho(t)&:=\mathsf C_{\beta,\mathrm{self}}^\rho(t)-z_{\pm,\mathrm{self}}^\rho I_M,\quad \widehat A_\pm^\rho(t)&:=\widehat{\mathsf C}_{\beta}^\rho(t)-z_{\pm,\mathrm{self}}^\rho I_M.
\end{aligned}
\]
Theorem~\ref{thm:general_capacitance_simplex_winding} gives
\[
\inf_t\sigma_{\min}\bigl(A_{\pm,\mathrm{cap}}^\rho(t)\bigr) \ge c\rho^2\lambda^\beta.
\]
Moreover, \eqref{eq:general_centered_full_matrix_comparison} and \eqref{eq:general_exact_dynamic_reduction_comparison} imply
\[
\sup_t\|A_{\pm,\mathrm{self}}^\rho(t)-A_{\pm,\mathrm{cap}}^\rho(t)\| \le Ck\rho^3,
\qquad
\sup_t\| \widehat A_\pm^\rho(t) -A_{\pm,\mathrm{self}}^\rho(t)\| \le C\rho^4.
\]
Using \(k\lambda\rho\le c_\ast\), the ratios of these errors to the projected point-gap bound satisfy
\[
\frac{Ck\rho^3}{\rho^2\lambda^\beta} \le Cc_\ast\lambda^{-(1+\beta)}, \qquad \frac{C\rho^4}{\rho^2\lambda^\beta} \le \frac{Cc_\ast^2}{k^2}\lambda^{-(2+\beta)}.
\]
Both tend to zero as \(\lambda\to\infty\). We may therefore choose \(\lambda_0\) sufficiently large and \(\rho_0\) sufficiently small so that
\[
Ck\rho^3+C\rho^4<c\rho^2\lambda^\beta.
\]

For either sign, consider
\[
\begin{aligned}
H_\pm^{(1)}(s,t)&:=(1-s)A_{\pm,\mathrm{cap}}^\rho(t)+sA_{\pm,\mathrm{self}}^\rho(t),\qquad H_\pm^{(2)}(s,t)&:=(1-s)A_{\pm,\mathrm{self}}^\rho(t)+s\widehat A_\pm^\rho(t).
\end{aligned}
\]
The singular-value perturbation inequality shows that both homotopies are invertible for every \((s,t)\in[0,1]\times[0,2\pi]\). Thus, \(\widehat{\mathsf C}_{\beta}^\rho\) has the asserted point gaps.

The homotopies commute with the simplex action and therefore preserve both symmetry sectors. Homotopy invariance transfers the sector winding \(-1\) from Theorem~\ref{thm:general_capacitance_simplex_winding} to \(\widehat\mu_\pm^\rho-z_{\pm,\mathrm{self}}^\rho\). Applying the same argument to the full determinant loops and using \eqref{eq:general_capacitance_winding_conclusion} gives \eqref{eq:general_exact_capacitance_windings}.
\end{proof}

The preceding conclusions are uniform for a family \(D_\tau\), \(0\le\tau\le1\), compact in \(C^{2,\alpha}\), provided that separation, equivariance, isolation of the static zero cluster, and the coefficient and resolvent bounds hold uniformly. If \(c_\ast\sup_{0\le\tau\le1}\sup_{y\in D_\tau}|y|<1\), then the thresholds can be chosen independently of \(\tau\), the pulled-back reduced loops and corrected reference points vary continuously, and the two determinant windings remain \(-1\) and \(-(M-1)\).

For spheres, rotational invariance gives \(P_{\mathrm{self}}^\rho=P_{\mathrm{cap}}^\rho\) exactly. Hence \(\mathcal J^\rho=I\), \(q_\rho=0\), and \(z_{\pm,\mathrm{self}}^\rho=z_\pm^\rho\). Consider the previously defined \(B_{\mathrm{self}}^\rho\) as an operator on \(\mathcal H_k(\Gamma^\rho)\), and set
\[
\mathscr C_0:=\{\zeta\in\mathbb C:|\zeta|=1/6\}.
\]
In the preceding Riesz construction, take \(\mathscr C_D=\mathscr C_0\), write
\(\iota^\rho:E_{\mathrm{cap}}^\rho\hookrightarrow\mathcal H_k(\Gamma^\rho)\) for the inclusion, and abbreviate \(U_{\mathrm{self}}^\rho,V_{\mathrm{self}}^\rho\) by \(U^\rho,V^\rho\). Then \(\widehat{\mathsf C}_\beta^\rho\) is the exact reduction in \eqref{eq:riesz_biorthogonal_frames}-\eqref{eq:riesz_biorthogonal_reduction}, while \(\mathsf C_\beta^\rho\) is the corresponding computable projected BFZ family on the fixed space. The following corollary records the resulting spherical winding specialization; the sphere-specific norm estimates enter only in its proof.

\begin{corollary}[Spherical regular-simplex winding of the Riesz-reduced family]
\label{cor:spherical_simplex_riesz_specialization}
Fix \(M\in\{2,3,4\}\), \(k,R,\eta>0\), \(\beta\in(-1/2,0)\), and \(c_\ast\in(0,\pi)\), and assume that \(0<kR<\pi/2\). There exist \(\rho_1>0\) and \(\lambda_1>1\) such that, whenever \(0<\rho\le\rho_1\), \(\lambda\ge\lambda_1\), and \(k\lambda\rho\le c_\ast\), the exact Riesz-reduced family \(\widehat{\mathsf C}_\beta^\rho\) is well-defined and scalar on the two simplex sectors:
\[
\widehat{\mathsf C}_\beta^\rho(t)|_{E_\pm^\rho}=\widehat\mu_\pm^\rho(t)I_{E_\pm^\rho}.
\]
It has point gaps at \(z_+^\rho\) and \(z_-^\rho\), and
\[
\operatorname{Wind}(\widehat\mu_\pm^\rho-z_\pm^\rho,0)=-1.
\]
Moreover,
\begin{equation}
\label{eq:spherical_exact_windings}
\begin{aligned}
\operatorname{Wind}\left(\det\bigl(\widehat{\mathsf C}_\beta^\rho(\cdot)-z_+^\rho I_M\bigr),0\right)&=-1,\qquad \operatorname{Wind}\left(\det\bigl(\widehat{\mathsf C}_\beta^\rho(\cdot)-z_-^\rho I_M\bigr),0\right)&=-(M-1).
\end{aligned}
\end{equation}
\end{corollary}

\begin{proof}
After pullback to the unit sphere, the spherical harmonics diagonalize each self-interaction:
\[
\left(\frac12I+K_0^\ast\right)Y_{\ell m} =\frac{\ell}{2\ell+1}Y_{\ell m}, \qquad \ell=0,1,2,\ldots.
\]
Thus, \(\mathscr C_0\) separates the constant harmonic from the nonconstant eigenvalues, which lie in \([1/3,1/2)\). After pullback, each diagonal combined-field block is
\[
\frac12I+(K_{k\rho}^{\mathbb S^2})^\ast-i\eta(k\rho)S_{k\rho}^{\mathbb S^2},
\]
whose difference from the static block is \(O(k\rho)\). Rotational invariance preserves every spherical-harmonic sector. For sufficiently small \(k\rho\), the resolvent identity therefore shows that \(\mathscr C_0\) exactly encloses the constant harmonic of each block and excludes all sectors with \(\ell\ge1\). Consequently, the enclosed Riesz range of \(B_{\mathrm{self}}^\rho\) is \(E_{\mathrm{cap}}^\rho\), its Riesz projection is \(P_{\mathrm{cap}}^\rho\), and its contour resolvent is uniformly bounded.

For a separated interaction block \(B\), let
\[
g(y):=f(p_j+\rho y)=\sum_{\ell,m}g_{\ell m}Y_{\ell m}(y),\qquad q(x):=(Bf)(p_i+\rho x).
\]
After pullback to \(\mathbb S^2\),
\[
q(x)=\rho^2\int_{\mathbb S^2}L_{ij,n}^\rho(x,y)g(y)\,d\sigma(y).
\]
The uniform kernel bounds, including the first \(x\)-derivatives, give
\[
\|q\|_{H^1(\mathbb S^2)}\le C\rho^2\|g\|_{L^2(\mathbb S^2)}
\]
for \(n\ge1\) and for the same-scale off-diagonal blocks; when \(n=-m\), the right-hand side acquires the factor \(\lambda^{-m}\). Moreover,
\[
\|f\|_{\mathcal H_k(\Gamma_j^\rho)}^2=\rho\sum_{\ell,m}\bigl((k\rho)^2+\ell(\ell+1)\bigr)^{1/2}|g_{\ell m}|^2.
\]
If \(q=\sum_{\ell,m}q_{\ell m}Y_{\ell m}\), then \(k\rho\le c_\ast\) implies
\[
\|Bf\|_{\mathcal H_k(\Gamma_i^\rho)}^2\le C\rho\|q\|_{H^1(\mathbb S^2)}^2\le C\rho^5\|g\|_{L^2(\mathbb S^2)}^2.
\]
Write \(g=g_{\mathrm c}+g_\perp\), where \(g_{\mathrm c}=g_{00}Y_{00}\), and let \(f=f_{\mathrm c}+f_\perp\) be the corresponding decomposition. Since \(\ell(\ell+1)\ge2\) for \(\ell\ge1\),
\[
\|f_{\mathrm c}\|_{\mathcal H_k}^2=k\rho^2|g_{00}|^2,\qquad \|f_\perp\|_{\mathcal H_k}^2\ge\sqrt2\,\rho\|g_\perp\|_{L^2(\mathbb S^2)}^2.
\]
Consequently,
\[
\|Bf_{\mathrm c}\|_{\mathcal H_k}\le C\rho^{3/2}\|f_{\mathrm c}\|_{\mathcal H_k},\qquad \|Bf_\perp\|_{\mathcal H_k}\le C\rho^2\|f_\perp\|_{\mathcal H_k}.
\]
For \(n=-m\), both right-hand sides acquire the factor \(\lambda^{-m}\). Since the two input sectors are orthogonal and \(0<\rho\le1\),
\[
\begin{aligned}
\|B_0^{k,\rho}-B_{\mathrm{self}}^\rho\|&\le C\rho^{3/2},\qquad
\|B_n^{k,\rho}\|&\le C\rho^{3/2}\quad(n\ge1),\qquad
\|B_{-m}^{k,\rho}\|&\le C\rho^{3/2}\lambda^{-m}\quad(m\ge1).
\end{aligned}
\]

The corresponding weighted geometric series are uniformly summable for fixed \(\beta\in(-1,0)\) and \(\lambda\ge\lambda_1\). It follows that
\[
\sup_t\|\mathcal T^\rho(t)-B_{\mathrm{self}}^\rho\|_{\mathcal L(\mathcal H_k(\Gamma^\rho))} \le C\rho^{3/2}.
\]

The uniform contour-resolvent bound above controls the constants in Proposition~\ref{prop:finite_cluster_reduction_nearby_loops}. Applying that proposition in the ambient \(\mathcal H_k(\Gamma^\rho)\)-norm with the reference operator \(B_{\mathrm{self}}^\rho\), the reference projection \(P_{\mathrm{cap}}^\rho\), and the reference space \(E_{\mathrm{cap}}^\rho\) proves that \(\Pi^\rho(t)\), \(U^\rho(t)\), \(V^\rho(t)\), and \(\widehat{\mathsf C}_\beta^\rho(t)\) are well-defined. The same application gives the intermediate estimate
\[
\begin{aligned}
\sup_t\bigl(&\|\Pi^\rho(t)-P_{\mathrm{cap}}^\rho\| +\|U^\rho(t)-\iota^\rho\|+\|V^\rho(t)-\iota^\rho\|\bigr) \le C\rho^{3/2}.
\end{aligned}
\]

For the matrix comparison, use the pullback norm \(X_\rho \) instead. Lemma~\ref{lem:general_capacitance_cluster_coupling} gives
\[
\sup_t\|\mathcal T^\rho(t)-B_{\mathrm{self}}^\rho\|_{\mathcal L(X_\rho)} \le C\rho^2.
\]
The quadratic estimate in Proposition~\ref{prop:finite_cluster_reduction_nearby_loops} therefore yields
\[
\sup_t\|\widehat{\mathsf C}_\beta^\rho(t)-\mathsf C_\beta^\rho(t)\| \le C\rho^4.
\]
Once the equilibrium-density basis is fixed, these finite matrices are independent of the choice of an equivalent ambient norm.

Simplex equivariance shows that \(\widehat{\mathsf C}_\beta^\rho(t)\) is scalar on \(E_+^\rho\) and \(E_-^\rho\). Therefore,
\[
\begin{aligned}
\widehat{\mathsf C}_\beta^\rho(t)|_{E_\pm^\rho}&=\widehat\mu_\pm^\rho(t)I_{E_\pm^\rho},\qquad \sup_t|\widehat\mu_\pm^\rho(t)-\mu_\pm^\rho(t)|&\le C\rho^4.
\end{aligned}
\]

Since \(F_{B_1}(s\omega)=j_0(s)>0\) for \(0\le s\le c_\ast<\pi\), Theorem~\ref{thm:general_capacitance_simplex_winding} applies to \(\mathsf C_\beta^\rho\) and gives
\[
\inf_t\sigma_{\min}\bigl(\mathsf C_\beta^\rho(t)-z_\pm^\rho I_M\bigr) \ge c\rho^2\lambda^\beta.
\]
The ratio of the reduction error to this point-gap bound satisfies
\[
\frac{C\rho^4}{\rho^2\lambda^\beta} =C\rho^2\lambda^{-\beta} \le \frac{Cc_\ast^2}{k^2}\lambda^{-(2+\beta)},
\]
which tends to zero as \(\lambda\to\infty\). After increasing \(\lambda_1\) and decreasing \(\rho_1\), the straight-line homotopy
\[
\begin{aligned}
H_\pm(s,t):={}&(1-s)\bigl(\mathsf C_\beta^\rho(t)-z_\pm^\rho I_M\bigr)+s\bigl(\widehat{\mathsf C}_\beta^\rho(t)-z_\pm^\rho I_M\bigr)
\end{aligned}
\]
is invertible for every \(s\in[0,1]\) and \(t\in[0,2\pi]\). It also preserves the two simplex sectors. Homotopy invariance therefore transfers the selected-sector windings and the determinant windings of \(\mathsf C_\beta^\rho\) to \(\widehat{\mathsf C}_\beta^\rho\). Theorem~\ref{thm:general_capacitance_simplex_winding} now gives \eqref{eq:spherical_exact_windings}.
\end{proof}

We finally organize the two invariant sector windings by the recursive simplex multiplicities and record their behavior under geometric rescaling of the wavenumber.

For each \(\kappa>0\) for which the relevant Riesz projection and both point gaps exist, repeat the construction with \(k=\kappa\), and denote the exact sector loops and the corresponding reference points by \(\widehat\mu_\pm^\rho(\kappa,t)\) and \(z_{\pm,\mathrm{self}}^\rho(\kappa)\), respectively. For spheres, \(z_{\pm,\mathrm{self}}^\rho(\kappa)=z_\pm^\rho(\kappa)\), with the latter given by \eqref{eq:simplex_same_scale_values}. Whenever the point gap conditions are satisfied, define
\[
w_+^\rho(\kappa):=\operatorname{Wind}\bigl(\widehat\mu_+^\rho(\kappa,t)-z_{+,\mathrm{self}}^\rho(\kappa),0\bigr),
\qquad
w_-^\rho(\kappa):=\operatorname{Wind}\bigl(\widehat\mu_-^\rho(\kappa,t)-z_{-,\mathrm{self}}^\rho(\kappa),0\bigr).
\]
The quantities \(w_\pm^\rho(\kappa)\) record the two exact sector windings. At the fixed wavenumber \(k\), abbreviate \(w_\pm^\rho:=w_\pm^\rho(k)\). For \(N\ge1\), the finite-depth simplex data follow the recursive multiplicities. There is one symmetric root sector. For every word \(v\in I^{\ell-1}\), \(1\le\ell\le N\), its \(M\) children contribute one \((M-1)\)-dimensional sum-zero sector. Hence, level \(\ell\) has multiplicity \(M^{\ell-1}(M-1)\), and
\[
1+\sum_{\ell=1}^NM^{\ell-1}(M-1)=M^N.
\]
Only these multiplicities are used; the copied sector spaces are not asserted to be invariant subspaces of the fully coupled prefractal boundary operator. Define the levelwise local winding vector
\begin{equation}
\label{eq:simplex_component_level_vector_definition}
\mathbf W_N^{\mathrm{loc},\rho} :=\bigl(w_+^\rho,(M-1)w_-^\rho,(M-1)Mw_-^\rho,\ldots,(M-1)M^{N-1}w_-^\rho\bigr).
\end{equation}

The next proposition collects the finite-depth multiplicities and their realization along a geometric sequence of wavenumbers.

\begin{proposition}[Finite-depth simplex data and geometric wavenumber scaling]
\label{prop:localized_component_cell_vector_separated_ifs}
For parameters satisfying the hypotheses and threshold conditions of Corollary~\ref{cor:general_capacitance_riesz_winding}, for every \(N\ge1\),
\begin{equation}
\label{eq:finite_depth_simplex_data}
\mathbf W_N^{\mathrm{loc},\rho} =\bigl(-1,-(M-1),-(M-1)M,\ldots,-(M-1)M^{N-1}\bigr).
\end{equation}
The conclusion is uniform over compact wavenumber intervals. Fix the geometric data and \(\eta>0\) as in Theorem~\ref{thm:general_capacitance_simplex_winding}, let \(\beta\in(-1/2,0)\), and choose
\[
\mathcal I_\kappa=[\kappa_-,\kappa_+]\Subset\left(0,\frac{\pi}{2R}\right),
\]
together with \(c_\ast>0\) such that
\[
\inf_{0\le s\le c_\ast}|F_D(s\omega)|>0.
\]
There exist \(\rho_0>0\) and \(\lambda_0>1\) such that the two sector-winding conclusions for the Riesz-reduced family hold simultaneously for every \(\kappa\in\mathcal I_\kappa\) whenever
\[
\lambda\ge\lambda_0, \qquad 0<\rho\le\rho_0, \qquad \kappa_+\lambda\rho\le c_\ast.
\]
Fix one such pair \((\rho,\lambda)\), choose \(\kappa_\circ\in\mathcal I_\kappa\), and set \(k_j:=\lambda^j\kappa_\circ\), \(0\le j\le N-1\). The root contribution is \(-1\), and the generation-\(j\) local contribution at physical wavenumber \(k_j\) is \(-(M-1)M^j\), so successive level contributions differ by the factor \(M\). These are level-resolved local contributions after rescaling, not determinant windings of the fully coupled depth-\(N\) prefractal at the wavenumbers \(k_j\).
\end{proposition}

\begin{proof}
Corollary~\ref{cor:general_capacitance_riesz_winding} gives \(w_+^\rho=w_-^\rho=-1\), so the first formula follows from the displayed multiplicities. The continuity and nonvanishing of \(d_\pm(\kappa R)\) on the compact interval \(\mathcal I_\kappa\) give a positive uniform lower bound for both leading sector coefficients. Moreover, the pulled-back layer-potential expansions and all coefficient remainders used in Theorem~\ref{thm:general_capacitance_simplex_winding} and Lemma~\ref{lem:general_capacitance_cluster_coupling} are uniform in \(\kappa\). The static contour is independent of \(\kappa\), and the uniform self-interaction estimates give a uniform contour-resolvent bound. Consequently, the constants in Proposition~\ref{prop:finite_cluster_reduction_nearby_loops}, the Riesz reductions, and the point-gap comparisons are uniform on \(\mathcal I_\kappa\), so \(\rho_0\) and \(\lambda_0\) may be chosen independently of \(\kappa\). The reduced loops and their reference points vary continuously with \(\kappa\), and Theorem~\ref{thm:finite_dimensional_winding_stability} therefore preserves both sector windings throughout \(\mathcal I_\kappa\). The boundary-operator covariance proved in Corollary~\ref{cor:general_frequency_ladder_winding}, together with \(k_j\lambda^{-j}=\kappa_\circ\), gives the stated level contributions.
\end{proof}

\subsection{Biorthogonal Zak phases and point-gap winding}
\label{subsec:symmetrization_zak_phase}

With the invariant spectral subspace selected by the Riesz projection in place, we now distinguish three geometric quantities: the biorthogonal Zak phase of that isolated subspace, the Zak phase of a globally isolated simple band, and the Zak phase of the negative spectral subspace of the chiral Hermitianization associated with a point gap. The equilibrium equation and jump relation give the static right-kernel basis, the Calder\'on identity identifies its constant dual functionals, and the parameter-dependent biorthogonal frames come from the Riesz projection.

More precisely, the connection defined by the biorthogonal Riesz frames is \((V_E^\Gamma)^\dagger\partial_tU_E^\Gamma\,dt\), whereas the point-gap logarithmic derivative is \((M_E^\Gamma-z_\ast I_E)^{-1}\partial_tM_E^\Gamma\,dt\). Their traces define scalar one-forms, but they are generally unequal. Thus, the Calder\'on identity does not by itself identify the biorthogonal BFZ Zak phase with the point-gap winding; Theorem~\ref{thm:point_gap_winding_as_zak_phase} gives the exact relation only for the chiral Hermitianization.

We return to the abstract Riesz setting of Subsection~\ref{subsec:finite_dimensional_scale_bloch_winding}. Suppose that \(t\mapsto\Pi^\Gamma(t)\) is \(C^1\) and periodic, and abbreviate
\[
U:=U_E^\Gamma, \qquad V:=V_E^\Gamma.
\]
Then \(U\) and \(V\) are \(C^1\), periodic, and satisfy \(V^\dagger U=U^\dagger V=I_E\).

The trace of the biorthogonal connection defines a phase for the full isolated spectral subspace. Geometrically, this is the Zak phase of its determinant line; when the subspace splits into globally separated simple bands, it is the sum of their bandwise phases.

\begin{proposition}[Biorthogonal Zak phase of an isolated spectral subspace]
\label{prop:determinant_line_bfz_cluster_phase}
Define the determinant-line biorthogonal BFZ Zak phase of the isolated spectral subspace by
\begin{equation}
\label{eq:determinant_line_bfz_cluster_phase}
\gamma_E^{\mathrm{BFZ}} :=-\operatorname{Im}\int_0^{2\pi}\operatorname{tr}_{E}\bigl(V(t)^\dagger\partial_tU(t)\bigr)\,dt \quad\pmod{2\pi}.
\end{equation}
It has the equivalent symmetrized form
\begin{equation}
\label{eq:symmetrized_biorthogonal_phase_identity}
\gamma_E^{\mathrm{BFZ}} \equiv\frac{i}{2}\int_0^{2\pi}\operatorname{tr}_{E}\!\left(V^\dagger\partial_tU+U^\dagger\partial_tV\right)dt \pmod{2\pi}.
\end{equation}
Both formulas are invariant modulo \(2\pi\) under every periodic \(C^1\) change of frame
\[
U\mapsto U\mathcal G, \qquad V\mapsto V(\mathcal G^{-1})^\dagger, \qquad \mathcal G:S^1\to\mathrm{GL}(E).
\]
\end{proposition}

\begin{proof}
Differentiating \(U^\dagger V=I_E\) gives
\[
\operatorname{tr}_E(U^\dagger V') =-\operatorname{tr}_E((U')^\dagger V) =-\overline{\operatorname{tr}_E(V^\dagger U')}.
\]
Therefore
\[
\frac{i}{2}\operatorname{tr}_E(V^\dagger U'+U^\dagger V') =-\operatorname{Im}\operatorname{tr}_E(V^\dagger U'),
\]
which proves \eqref{eq:symmetrized_biorthogonal_phase_identity}.

Under the stated frame change,
\[
\bigl(V(\mathcal G^{-1})^\dagger\bigr)^\dagger(U\mathcal G)' =\mathcal G^{-1}(V^\dagger U')\mathcal G+\mathcal G^{-1}\mathcal G'.
\]
Taking traces and applying Jacobi's formula to \(\det\mathcal G\) gives
\[
\int_0^{2\pi}\operatorname{tr}_E(\mathcal G^{-1}\mathcal G')\,dt=2\pi i\,\operatorname{Wind}(\det\mathcal G,0).
\]
Thus, \eqref{eq:determinant_line_bfz_cluster_phase} changes by an integer multiple of \(2\pi\), proving the frame invariance.
\end{proof}

An early formulation of biorthogonal geometric phase appears in \cite{garrison_wright1988}. For rank one, \eqref{eq:symmetrized_biorthogonal_phase_identity} reduces to the symmetrized biorthogonal connection used for bands of generalized capacitance matrices in \cite{ammari_barandun_cao_feppon2023}; here it is attached to boundary-density frames of the invariant spectral subspace selected by the Riesz projection around an isolated BFZ spectral cluster.

Under the hypotheses and thresholds of Corollary~\ref{cor:general_capacitance_riesz_winding}, the Zak phase of the small-component invariant spectral subspace is perturbatively small despite the nonzero point-gap winding:
\begin{equation}
\label{eq:general_cluster_phase_bound}
\operatorname{dist}\bigl(\gamma_{E_{\mathrm{self}}^\rho}^{\mathrm{BFZ}},2\pi\mathbb Z\bigr)\le C\rho^4.
\end{equation}
Indeed, the contour bound in Lemma~\ref{lem:general_capacitance_cluster_coupling} and \eqref{eq:general_bfz_self_coupling} give a uniform resolvent for \(\mathcal T^\rho(t)\). Differentiating the Riesz formula and using \eqref{eq:general_riesz_frame_estimates} then gives
\[
\sup_t\|\partial_t\Pi^\rho(t)\|+\sup_t\|\Pi^\rho(t)-P_{\mathrm{self}}^\rho\|\le C\rho^2.
\]
Moreover, \(U^{-1}\) is uniformly bounded. Writing \(U=U_{\mathrm{self}}^\rho\), \(V=V_{\mathrm{self}}^\rho\), \(\Pi=\Pi^\rho\), and \(\iota=\iota_{\mathrm{self}}^\rho\), the identities \(U=\Pi\iota\), \(V^\dagger=U^{-1}\Pi\), and \(\Pi\Pi'\Pi=0\) yield
\[
V^\dagger U'=U^{-1}\Pi\Pi'(P_{\mathrm{self}}^\rho-\Pi)\iota=O(\rho^4)
\]
uniformly in \(t\). Formula~\eqref{eq:determinant_line_bfz_cluster_phase} proves \eqref{eq:general_cluster_phase_bound}; a constant change of equilibrium-density basis leaves the trace unchanged. Thus, this Zak phase lies within \(O(\rho^4)\) of zero modulo \(2\pi\), while the point-gap windings in \eqref{eq:general_exact_capacitance_windings} remain nonzero.

The next corollary verifies the bandwise definition for a uniformly isolated simple branch and relates the band phases to the phase of the isolated subspace.

\begin{corollary}[Bandwise BFZ Zak phase for an isolated simple band]
\label{cor:bandwise_bfz_zak_phase}
Let \(M(t):=M_E^\Gamma(t)\) be a \(C^1\) reduced loop. Suppose that \(\nu_j(t)\) is a periodic algebraically simple eigenvalue branch satisfying
\[
\inf_{0\le t\le2\pi}\operatorname{dist}\!\left(\nu_j(t), \sigma(M(t))\setminus\{\nu_j(t)\}\right)>0.
\]
Then its right and left eigenlines admit periodic \(C^1\) eigenvectors \(r_j,\ell_j\in E\) normalized by
\[
Mr_j=\nu_jr_j, \qquad M^\dagger\ell_j=\overline{\nu_j}\ell_j, \qquad \ell_j^\dagger r_j=1.
\]
Set \(u_j^R:=Ur_j\) and \(u_j^L:=V\ell_j\). Then the Zak phase in \eqref{eq:early_bandwise_bfz_zak_phase} is given by
\begin{equation}
\label{eq:bandwise_bfz_zak_phase}
\gamma_j^{\mathrm{BFZ}} :=-\operatorname{Im}\int_0^{2\pi} (u_j^L)^\dagger\partial_tu_j^R\,dt =-\operatorname{Im}\int_0^{2\pi} \left[\ell_j^\dagger r_j' +\ell_j^\dagger(V^\dagger U')r_j\right]dt \quad\pmod{2\pi}.
\end{equation}
This phase is independent of the periodic biorthogonal gauge. If all \(q=\dim E\) eigenvalue branches satisfy the same hypothesis, then
\begin{equation}
\label{eq:sum_bandwise_cluster_zak_phase}
\sum_{j=1}^q\gamma_j^{\mathrm{BFZ}} \equiv\gamma_E^{\mathrm{BFZ}}\pmod{2\pi}.
\end{equation}
\end{corollary}

\begin{proof}
The uniformly separated right and left eigenline projections are periodic and \(C^1\). Their line bundles over \(S^1\) are trivial, while algebraic simplicity makes the left-right pairing nonzero; hence periodic eigenvectors may be chosen with \(\ell_j^\dagger r_j=1\). The Riesz intertwining identities show that \(u_j^R,u_j^L\) are exact biorthogonal boundary-density eigenvectors of the BFZ symbol, and
\[
(u_j^L)^\dagger(u_j^R)'=\ell_j^\dagger V^\dagger(U'r_j+Ur_j')=\ell_j^\dagger(V^\dagger U')r_j+\ell_j^\dagger r_j',
\]
which proves \eqref{eq:bandwise_bfz_zak_phase}. A periodic gauge \(r_j\mapsto gr_j\), \(\ell_j\mapsto\ell_j/\overline g\) adds \(g^{-1}g'\) to the connection, whose integral lies in \(2\pi i\mathbb Z\).

If all branches satisfy the hypothesis, their eigenvectors are biorthogonal. With \(\mathsf X=(r_1,\ldots,r_q)\), the left eigenvector matrix is \((\mathsf X^{-1})^\dagger\), and hence
\[
\sum_{j=1}^q(u_j^L)^\dagger(u_j^R)' =\operatorname{tr}(V^\dagger U')+\operatorname{tr}(\mathsf X^{-1}\mathsf X').
\]
The last term integrates to \(2\pi i\operatorname{Wind}(\det\mathsf X,0)\), proving \eqref{eq:sum_bandwise_cluster_zak_phase}.
\end{proof}

Point-gap winding records eigenvalue motion, whereas \eqref{eq:bandwise_bfz_zak_phase} records eigenvector transport. For example, \(M(t)=\operatorname{diag}(z_\ast+e^{it},z_\ast+3)\) has \(W(M;z_\ast)=1\) but zero band phases. More generally, when every branch is periodic,
\[
W(M;z_\ast)=\sum_{j=1}^q\operatorname{Wind}(\nu_j-z_\ast,0),
\]
so the \(C^0\) comparison used to transfer winding does not transfer a band phase. That would require a uniform band gap and a \(C^1\) comparison, as well as a periodic \(C^1\) reconstruction if a physical-field phase is sought.

The final theorem produces a quantized phase from a point gap without assuming simple bands, diagonalizability, or global eigenvector labels. It combines the unitary polar factor with the standard chiral Hermitianization used for point-gapped non-Hermitian operators in \cite{gong_ashida_kawabata_et_al2018,kawabata_shiozaki_ueda_sato2019}; the Zak-phase identity follows from an explicit frame of its negative spectral subspace. For earlier Berry-like and spectral windings around exceptional points, see \cite{leykam_bliokh_huang_chong_nori2017}.

\begin{theorem}[Point-gap winding and Zak phase of the chiral Hermitianization]
\label{thm:point_gap_winding_as_zak_phase}
Let \(E\) be a finite-dimensional complex Hilbert space and let \(M:[0,2\pi]\to\mathcal L(E)\) be a \(C^1\) periodic loop having a point gap at \(z_\ast\). Set
\begin{equation}
\label{eq:point_gap_polar_unitary}
A_{z_\ast}(t):=M(t)-z_\ast I_E, \qquad \mathcal Q_{z_\ast}(t):=A_{z_\ast}(t) \bigl(A_{z_\ast}(t)^\dagger A_{z_\ast}(t)\bigr)^{-1/2}.
\end{equation}
Then \(\mathcal Q_{z_\ast}\) is the unitary polar factor of \(A_{z_\ast}\) and is a \(C^1\) periodic loop. The chiral Hermitianization, also called the doubled Hermitian family,
\begin{equation}
\label{eq:point_gap_hermitianization}
\mathbb H_{z_\ast}(t):=
\begin{pmatrix}
0&A_{z_\ast}(t)\\
A_{z_\ast}(t)^\dagger&0
\end{pmatrix}
\end{equation}
has a uniform spectral gap about zero, and its family of negative spectral subspaces has the periodic orthonormal frame
\begin{equation}
\label{eq:negative_chiral_frame}
\Psi_-(t):=\frac1{\sqrt2}
\begin{pmatrix}
-\mathcal Q_{z_\ast}(t)\\ I_E
\end{pmatrix}.
\end{equation}
The winding of the polar unitary is exactly the determinant point-gap winding:
\begin{equation}
\label{eq:point_gap_chiral_winding}
\frac1{2\pi i}\int_0^{2\pi}\operatorname{tr}_E\bigl( \mathcal Q_{z_\ast}(t)^\dagger\partial_t\mathcal Q_{z_\ast}(t)\bigr)\,dt =W(M;z_\ast).
\end{equation}
The Zak phase of the negative spectral subspace satisfies
\begin{equation}
\label{eq:point_gap_zak_winding_identity}
\gamma_{\mathrm{Herm}}(z_\ast) :=i\int_0^{2\pi}\operatorname{tr}_E\bigl(\Psi_-(t)^\dagger\partial_t\Psi_-(t)\bigr)\,dt =-\pi W(M;z_\ast) \equiv\pi W(M;z_\ast)\pmod{2\pi}.
\end{equation}
\end{theorem}

\begin{proof}
Continuity on the compact parameter interval and the point-gap condition give
\[
m_\ast:=\min_{0\le t\le2\pi} \sigma_{\min}(A_{z_\ast}(t))>0.
\]
Thus \(A_{z_\ast}^\dagger A_{z_\ast}\ge m_\ast^2I_E\). Functional calculus on the positive-definite cone shows that \((A_{z_\ast}^\dagger A_{z_\ast})^{-1/2}\), and hence \(\mathcal Q_{z_\ast}\), is \(C^1\) and periodic. A direct calculation gives
\[
\mathcal Q_{z_\ast}^\dagger\mathcal Q_{z_\ast} =\mathcal Q_{z_\ast}\mathcal Q_{z_\ast}^\dagger=I_E.
\]

Let \(|A_{z_\ast}|:=(A_{z_\ast}^\dagger A_{z_\ast})^{1/2}\). The polar decomposition gives \(A_{z_\ast}=\mathcal Q_{z_\ast}|A_{z_\ast}|\) and \(A_{z_\ast}^\dagger\mathcal Q_{z_\ast}=|A_{z_\ast}|\). Therefore
\[
\mathbb H_{z_\ast}\Psi_- =\frac1{\sqrt2}
\begin{pmatrix}
A_{z_\ast}\\-A_{z_\ast}^\dagger\mathcal Q_{z_\ast}
\end{pmatrix}
=\Psi_-(-|A_{z_\ast}|),
\qquad
\Psi_-^\dagger\Psi_-=I_E.
\]
The square of \(\mathbb H_{z_\ast}\) is block diagonal with blocks \(A_{z_\ast}A_{z_\ast}^\dagger\) and \(A_{z_\ast}^\dagger A_{z_\ast}\). Hence, its spectrum lies in \((-\infty,-m_\ast]\cup[m_\ast,\infty)\), and the \(\dim E\) columns of \(\Psi_-\) span its negative spectral subspace. Differentiating the frame gives
\[
\Psi_-^\dagger\partial_t\Psi_- =\frac12\mathcal Q_{z_\ast}^\dagger\partial_t\mathcal Q_{z_\ast}.
\]

Finally,
\[
\det\mathcal Q_{z_\ast} =\det A_{z_\ast}\, \det(A_{z_\ast}^\dagger A_{z_\ast})^{-1/2} =\frac{\det A_{z_\ast}}{|\det A_{z_\ast}|}.
\]
Thus, \(\det\mathcal Q_{z_\ast}\) and \(\det A_{z_\ast}\) have the same winding. Since \(\mathcal Q_{z_\ast}^{-1}=\mathcal Q_{z_\ast}^\dagger\), Jacobi's formula yields
\[
\frac{(\det\mathcal Q_{z_\ast})'}{\det\mathcal Q_{z_\ast}} =\operatorname{tr}_E(\mathcal Q_{z_\ast}^\dagger\mathcal Q_{z_\ast}').
\]
Integrating proves \eqref{eq:point_gap_chiral_winding}. Substitution into the frame connection gives
\[
i\int_0^{2\pi}\operatorname{tr}_E(\Psi_-^\dagger\Psi_-')\,dt =\frac{i}{2}(2\pi i)W(M;z_\ast) =-\pi W(M;z_\ast),
\]
which proves \eqref{eq:point_gap_zak_winding_identity}.
\end{proof}

The theorem applies to every \(C^1\) Riesz-reduced loop considered above whenever its stated point-gap condition holds. Applied to the exact reduced loop
\(\widehat{\mathsf C}_\beta^\rho\) in Corollary~\ref{cor:general_capacitance_riesz_winding}, it gives
\[
\gamma_{\mathrm{Herm}}(z_{+,\mathrm{self}}^\rho)\equiv\pi, \qquad \gamma_{\mathrm{Herm}}(z_{-,\mathrm{self}}^\rho)\equiv(M-1)\pi\pmod{2\pi}.
\]
On the sum-zero sector, \(\mathcal Q_{z_{-,\mathrm{self}}^\rho}\) is the normalized scalar loop associated with \(\widehat\mu_-^\rho-z_{-,\mathrm{self}}^\rho\), whose winding is \(-1\). After choosing a parameter-independent basis of this sector, the Hermitianized family is the direct sum of \(M-1\) identical rank-one families. Each rank-one summand has Zak phase \(\pi\), while the phase of the full sum-zero sector is \((M-1)\pi\); for \(M=3\), the latter phase vanishes modulo \(2\pi\) although each rank-one summand has nontrivial phase.

In summary, \(\gamma_E^{\mathrm{BFZ}}\) is defined for a \(C^1\) invariant spectral subspace selected by a Riesz projection around an isolated spectral cluster, \(\gamma_j^{\mathrm{BFZ}}\) requires a globally isolated periodic simple band, and \(\gamma_{\mathrm{Herm}}\) is defined for every \(C^1\) point-gapped reduced loop. Only the phase of the chiral Hermitianization is fixed by the point-gap winding. The finite calculations below are diagnostics for the predicted scale behavior; identifying them with the exact Riesz-reduced winding would require a separate convergence analysis.

\section{Numerical experiments}
\label{sec:num}
The following experiments are intended as numerical illustrations of the theoretical result. We will discretize the BFZ boundary symbol directly and compute the winding of the resulting nodal matrices. Figure~\ref{fig:projected_simplex_prefractals} gives schematic displays of the smooth prefractal geometries for \(M=3\) and \(M=4\).

\begin{figure}[H]
\centering
\includegraphics[width=0.6\linewidth]{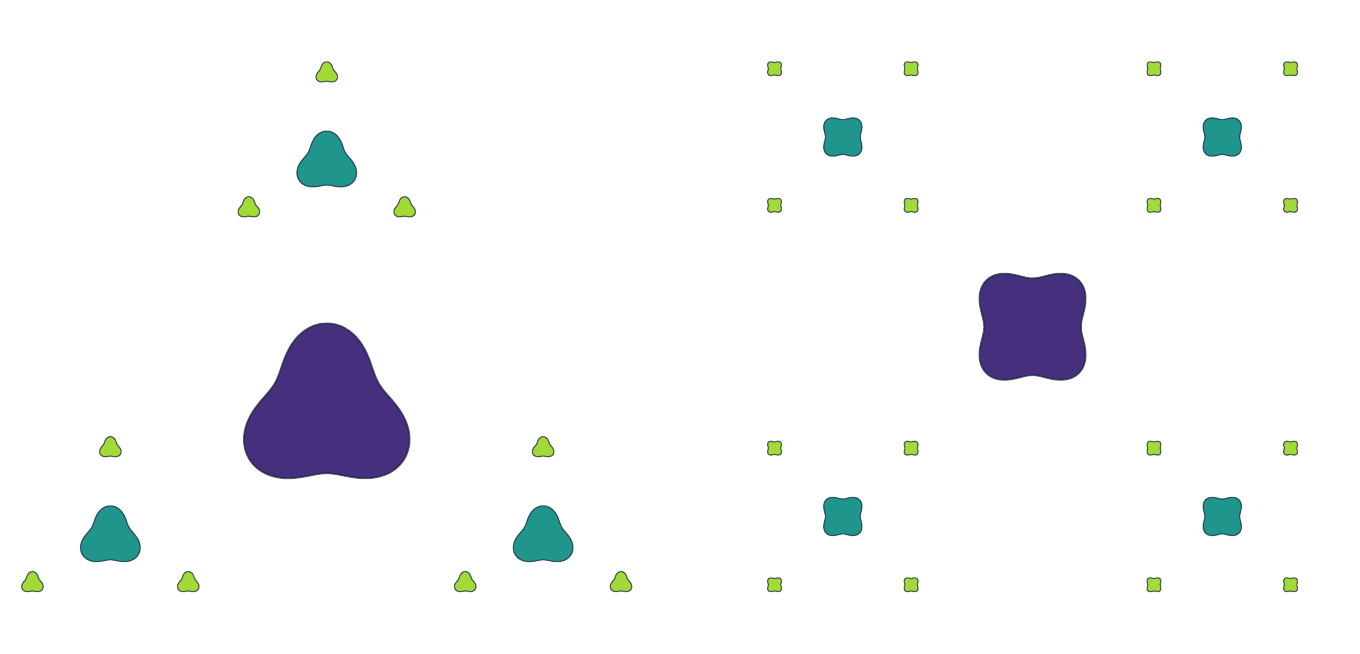}
\caption{Schematic display for prefractals of depth $2$ with contraction \(0.36\). The left and right plots contain \(1+3+3^2\) rounded-triangular and \(1+4+4^2\) rounded-square cells, respectively.}
\label{fig:projected_simplex_prefractals}
\end{figure}

\paragraph{Direct boundary-symbol winding.}
For \(M=3\) and \(4\), take the equilateral centers and tetrahedral centers as
\begin{align*}
&p_j^{(3)}  = \left(\cos\left(\frac{\pi}{2}+\frac{2\pi j}{3}\right),\sin\left(\frac{\pi}{2}+\frac{2\pi j}{3}\right),0\right), \quad j = 0,1,2\\
&\{p_j^{(4)}\}_{j=1}^4=\left\{\frac{\varepsilon}{\sqrt3}:\varepsilon\in\{\pm1\}^3,\ \varepsilon_1\varepsilon_2\varepsilon_3=1\right\}.
\end{align*}
Define
\[
h_3(\omega):=\operatorname{Re}\left(e^{-3\pi i/2}(\omega_1+i\omega_2)^3\right),\qquad h_4(\omega):=\frac52\left(\omega_1^4+\omega_2^4+\omega_3^4-\frac35\right),
\]
and
\[
\Sigma_M^\rho:=\left\{\rho\bigl(1+\varepsilon_Mh_M(\omega)\bigr)\omega:\omega\in\mathbb S^2\right\},\qquad \varepsilon_3=0.20,\quad \varepsilon_4=0.15.
\]
The scatterers \(p_j^{(M)}+\Sigma_M^\rho\) retain the corresponding \(S_3\)- and \(S_4\)-symmetries discussed in Section~\ref{subsec:scalar_scale_winding_reference}.

Using one collocation point \(x_a\) and one constant density unknown on each curved parameter panel \(P_b\), we assemble the BFZ blocks defined in \eqref{eq:Bn_def_weighted_sec} directly:
\[
(B_{0,h}^k)_{ab}:=\frac12\delta_{ab}  +\int_{P_b}\mathsf K_\eta^k(x_a,y)\,d\sigma(y),
\qquad
(B_{n,h}^k)_{ab}:=\lambda^n \int_{P_b}\mathsf K_\eta^k(x_a,\lambda^ny)\,d\sigma(y), \quad n\ne0.
\]
The same-panel singular integrals are evaluated by a four-triangle Duffy decomposition of order \(40\), while the remaining integrals use tensor Gaussian quadrature of order \(32\) with parameter meshes of \(100\) panels per component.

We take
\[
k=0.8,\qquad \rho=0.003,\qquad \eta=1,\qquad \lambda=64,\qquad \beta=-\frac13,
\]
and compute the symmetric Laurent truncation
\[
A_h^{(M)}(t):= \sum_{n=-5}^{5}e^{-int}\lambda^{\beta n}B_{n,h}^k, \qquad 0\le t\le2\pi,
\]
at \(Q=512\) equally spaced values \(t_q=2\pi q/Q\). The one-sided estimates in Proposition~\ref{prop:one_sided_Bn_bounds} imply norm convergence of the Laurent series. For the chosen truncation, the omitted Laurent tails have the matrix $2$-norm bounded by \(9.59\times10^{-7}\) for \(M=3\) and \(1.26\times10^{-6}\) for \(M=4\).

The same-scale cluster of \(B_{0,h}^k\) splits into the symmetric sector and the standard sector. We denote the corresponding eigenvalues by \(z_{+,h}\) and \(z_{-,h}\), respectively. Thus,
\[
 B_{0,h}^k\big|_{E_{+,h}}=z_{+,h}I,  \qquad B_{0,h}^k\big|_{E_{-,h}}=z_{-,h}I.
\]
After computing $z_{+,h}$ and $z_{-,h}$ numerically, for every sampled value $t_q$, the contour $\mathscr C_0:=\{\zeta\in\mathbb C:|\zeta|=1/6\}$ encloses exactly $M$ eigenvalues $\xi_{1,h}(t_q),\ldots,\xi_{M,h}(t_q)$, counted with algebraic multiplicity. We define the two cluster determinants by
\[
D_{\pm,h}(t_q):= \prod_{r=1}^{M}\bigl(\xi_{r,h}(t_q)-z_{\pm,h}\bigr)
\]
and compute their numerical windings from their unwrapped phases along the closed sampled path, with $t_Q=2\pi$ identified with $t_0=0$. The corresponding sampled full-matrix point gaps are
\[
g_{\pm,h}^{\mathrm{samp}}:=\min_{0\le q<Q}\sigma_{\min}\!\left(A_h^{(M)}(t_q)-z_{\pm,h}I\right).
\]

\begin{table}[H]
\centering
\footnotesize
\caption{Numerical cluster windings and the resulting levelwise local data. The first two entries are the windings computed from $D_{+,h}$ and $D_{-,h}$; the remaining two are obtained from the standard-sector winding using the multiplicity factors $M$ and $M^2$. The reported sampled gaps were computed using $100$ panels per component.}
\label{tab:numerical_full_boundary_winding}
\setlength{\tabcolsep}{4pt}
\begin{tabular}{lccc}
\toprule
 \(M\) & Local winding data & $g_{+,h}^{\mathrm{samp}}$ & $g_{-,h}^{\mathrm{samp}}$ \\ \midrule
 3 & \((-1,-2,-6,-18)\) & \(3.217\times10^{-6}\) & \(2.057\times10^{-6}\) \\
 4 & \((-1,-3,-12,-48)\) & \(4.754\times10^{-6}\) & \(1.880\times10^{-6}\) \\
\bottomrule
\end{tabular}
\end{table}

\paragraph{Full impedance-scattering scale cycle.}
The second experiment examines the phase accumulated by a shifted determinant of the normalized physical far-field response as the wavenumber varies from $k_0$ to $3k_0$. The far-field response factors through the inverse of the physical boundary operator and is therefore not a numerical approximation of the BFZ symbol or its Riesz reduction. For spherical inclusions, we test whether this accumulated phase approaches $-2\pi$ as the depth increases; for both geometries, we compare the endpoint phase defects defined below.

Here \(N\) denotes the depth of the finest prefractal generation. We use the equilateral three-map construction with contraction \(1/3\) and place \(3^N\) inclusions of reference radius \(r_N:=0.05\,3^{-N}\) at its depth-\(N\) centers. We consider spherical inclusions and the threefold-symmetric rounded inclusions
\[
\Sigma_{\triangle}^{r_N}:= \left\{r_N\left(1+0.18\operatorname{Re}\bigl((\omega_1+i\omega_2)^3\bigr)\right)(\omega_1,\omega_2,0.70\omega_3):\omega\in\mathbb S^2\right\}.
\]
Denote the corresponding union of inclusions by \(D_N\). For each plane-wave incidence, we solve
\[
(\Delta+k^2)u^s=0 \quad\text{in }\mathbb R^3\setminus\overline{D_N},
\qquad
\partial_\nu(u^i+u^s)+ik(u^i+u^s)=0 \quad\text{on }\partial D_N,
\]
with the Sommerfeld radiation condition. Writing \(u^s=S_{D_N}^k[\varphi_N]\), the density satisfies
\[
\left(-\frac12I+(K_{D_N}^k)^*+ikS_{D_N}^k\right)\varphi_N=-(\partial_\nu u^i+iku^i).
\]
We use the six coplanar incident and observation directions
\[
\mathbf d_j:=\left(\cos\frac{2\pi j}{6},\sin\frac{2\pi j}{6},0\right), \qquad \widehat x_j:=\left(\cos\left(\frac{2\pi j}{6}+\frac{\pi}{6}\right),  \sin\left(\frac{2\pi j}{6}+\frac{\pi}{6}\right),0\right), \qquad j=0,\ldots,5.
\]
For the $j$th incidence, let \(u_j^i(x):=e^{ik\mathbf d_j\cdot x}\), and define  \(\bigl(\mathsf F_N(k)\bigr)_{\ell j}:=u_N^\infty(\widehat x_\ell;\mathbf d_j,k)\), for \(0\le \ell,j\le5\), where $u_N^\infty$ denotes the corresponding far-field pattern.

To specify the normalization, let \(\mathsf F_{\mathrm{iso},N}(k)\) be the response matrix of one isolated inclusion of the same shape and size as the components of $D_N$, and set \(a_N(k):=\frac1{6}\sum_{j=0}^{5}\bigl(\mathsf F_{\mathrm{iso}, N}(k)\bigr)_{jj}.\) We define
\[
\widehat{\mathsf F}_N(k):= \frac{\mathsf F_N(k)}{3^N a_N(k)}.
\]
This normalization is fixed throughout the depth study and is intended to factor out the leading extensive contribution and the dominant single-inclusion size-frequency dependence. With
\[
k(s)=k_0\,3^s, \qquad k_0=2.14879, \qquad z_\ast=-0.16+0.55i, \qquad 0\le s\le1,
\]
define
\[
d_N(s):=\det\left(\widehat{\mathsf F}_N(k(s))-z_\ast I_6\right),
\qquad
\nu_N:=\frac{\theta_N(1)-\theta_N(0)}{2\pi},
\]
where \(\theta_N\) is an unwrapped phase of \(d_N\) along the sampled path. Thus, \(d_N\) is the shifted determinant whose phase increment is measured. For visualization only, set
\[
q_N(s):=\frac{d_N(s)}{d_N(0)},
\qquad
\rho_N:=|q_N(1)|,
\qquad
\Gamma_N(s):=\rho_N^{-s}q_N(s).
\]
Since \(\rho_N^s\) interpolates geometrically between the endpoint moduli, this normalization subtracts the linear endpoint trend from \(\log|q_N(s)|\) without changing the phase. Consequently,
\[
\Gamma_N(0)=1,
\qquad
\Gamma_N(1)=e^{2\pi i\nu_N},
\qquad
\delta_N:=|\Gamma_N(1)-1|,
\]
so \(\delta_N\) measures phase closure alone.

Each inclusion is discretized by $80$ triangular panels with a piecewise-constant density, and all displayed values use $32$ equal subintervals of $s$. For spherical inclusions, \(\nu_3=-1.004268\), \(\nu_4 = -1.000357\), \(\nu_5 = -0.999989\), and the corresponding endpoint phase defects are
\[
\delta_3=2.68\times10^{-2},
\qquad
\delta_4=2.24\times10^{-3},
\qquad
\delta_5=6.66\times10^{-5}.
\]
Figure~\ref{fig:full_impedance_scale_cycle} compares \(\delta_N\) for spherical and rounded triangular inclusions over \(2\le N\le5\). At $N=5$, the respective defects are $6.66\times10^{-5}$ and $2.40\times10^{-3}$. The corresponding sampled point gaps, defined analogously to those in the preceding experiment, are $6.643\times10^{-2}$ and $6.653\times10^{-2}$, respectively.

\begin{figure}[H]
\centering
\begin{subfigure}[t]{0.49\linewidth}
\centering
\includegraphics[width=\linewidth]{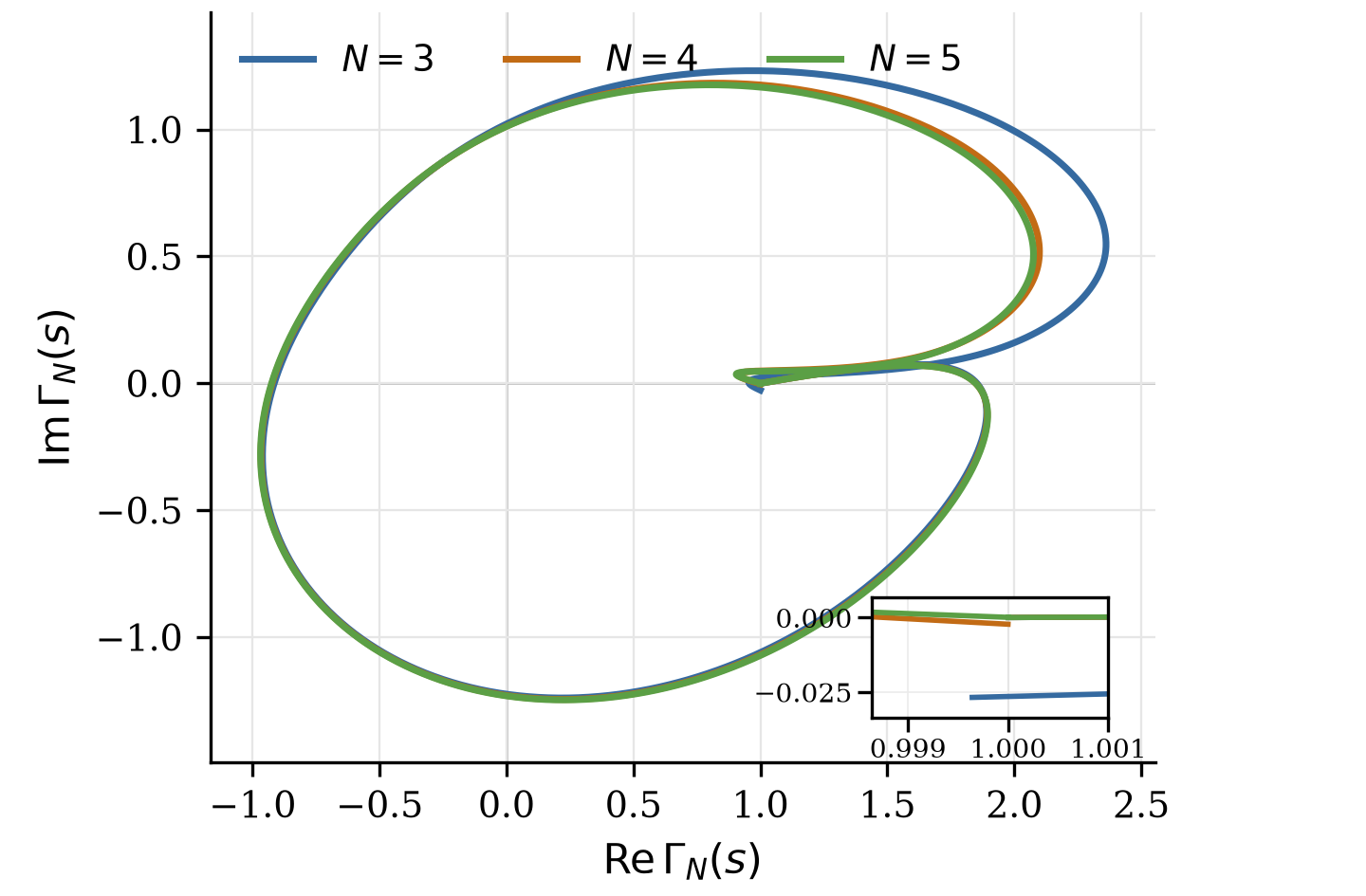}
\label{fig:full_impedance_sphere_path}
\end{subfigure}
\hfill
\begin{subfigure}[t]{0.49\linewidth}
\centering
\includegraphics[width=\linewidth]{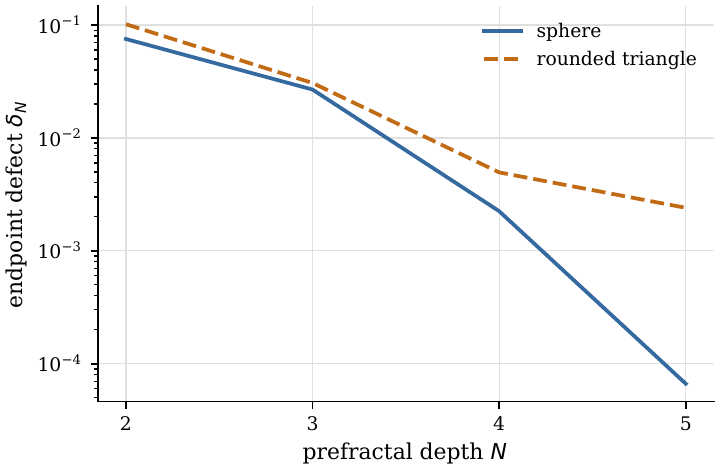}
\label{fig:full_impedance_depth_defect}
\end{subfigure}
\caption{Finite-depth illustration of phase closure. Left: the rescaled paths \(\Gamma_N\) for spherical inclusions at \(N=3,4,5\); the inset shows their endpoints near \(1\). Right: the endpoint defect \(\delta_N=|\Gamma_N(1)-1|\) for spherical and rounded triangular inclusions.}
\label{fig:full_impedance_scale_cycle}
\end{figure}
Since the finite-depth paths are not exactly closed, the quantities $\nu_N$ are phase increments measured in turns rather than integer windings. For spherical inclusions, their observed approach to $-1$ means that the accumulated phase approaches $-2\pi$. This physical diagnostic is not identified with the exact Riesz-reduced BFZ winding.
 
\section{Concluding remarks} \label{sec:conclusion}
In this paper, we have introduced a scale-periodic boundary-integral model motivated by exact dilation identities for Helmholtz scattering from a penetrable scatterer and carried out its topological analysis. After all scale layers are normalized to a common reference boundary, the interscale interaction becomes a block Laurent operator. Its one-sided coefficient estimates determine a natural exponentially weighted scale space, and the Bloch-Floquet-Zak transform fiberizes the resulting operator into a continuously differentiable loop over the dual scale circle. 

For reference cells composed of sufficiently small, well-separated components, the static equilibrium densities and their Calderón-dual constant functionals provide a computable finite-dimensional projected matrix. A two-stage Riesz construction identifies the corresponding invariant spectral cluster of the full BFZ boundary symbol. After recentering at the appropriate nonzero-wavenumber same-scale values, the projected family and the exact Riesz reduction differ by an error smaller than the relevant point gap. Their determinant windings therefore agree.

In regular-simplex configurations, the reduced space splits into its symmetric and standard sectors. Under the stated nonvanishing and scale-separation assumptions, both sector loops have winding number \(-1\). The full reduced determinant consequently has nontrivial windings about the symmetric-sector reference point and about the standard-sector reference point. These point-gap windings should be distinguished from biorthogonal eigenvector phases: they do not determine a bandwise Zak phase without a globally isolated simple band. They do, however, determine the Zak phase of the negative spectral subspace of the associated chiral Hermitianization.

\section*{Acknowledgments}
The authors thank Jade Leathrum for helpful discussions. 


\paragraph{Funding.}
Yat Tin Chow was supported by the National Science Foundation under Grant No.~DMS-2409903 and by the Office of Naval Research under Award No.~N00014-24-1-2661.




\paragraph{Use of generative AI.}
During the preparation of this manuscript, the authors used ChatGPT (OpenAI) for language editing and assistance with the development of code for the numerical experiments. All AI-assisted text, code, and numerical outputs were independently reviewed and verified by the authors. The authors take full responsibility for the content of the article.


\bibliographystyle{plain}
\bibliography{ref.bib}

@article{chen2023enhanced,
  title={Enhanced wave localization in multifractal scattering media},
  author={Chen, Yuyao and Sgrignuoli, Fabrizio and Zhu, Yilin and Shubitidze, Tornike and Dal Negro, Luca},
  journal={Physical Review B},
  volume={107},
  number={5},
  pages={054201},
  year={2023},
  publisher={APS}
}

@article{guerin1996scattering,
  title={Scattering on fractal measures},
  author={Guerin, Charles-Antoine and Holschneider, Matthias},
  journal={Journal of Physics A: Mathematical and General},
  volume={29},
  number={23},
  pages={7651--7667},
  year={1996}
}

@article{fractal-photonic2,
author={Zhang, Xiujuan and Zangeneh-Nejad, Farzad and Chen, Ze-Guo and  Lu, Ming-Hui
and Christensen, Johan}, 
year={2023}, 
title={A second wave of topological phenomena in photonics and acoustics},
journal={Nature},
volume={618}, 
pages={687--697},
}

@article{fractal-photonic1,
author = {Tobias Biesenthal  and Lukas J. Maczewsky  and Zhaoju Yang  and Mark Kremer  and Mordechai Segev  and Alexander Szameit  and Matthias Heinrich },
title = {Fractal photonic topological insulators},
journal = {Science},
volume = {376},
number = {6597},
pages = {1114--1119},
year = {2022},
}

@book{falconer2014fractal,
  author    = {Falconer, Kenneth},
  title     = {Fractal Geometry: Mathematical Foundations and Applications},
  edition   = {3rd},
  publisher = {John Wiley \& Sons},
  address   = {Chichester},
  year      = {2014},
  isbn      = {9781119942399}
}

@book{cbms,
  author    = {Ammari, Habib and Davies, Bryn and Hiltunen, Erik Orvehed},
  title     = {Mathematical Theories for Metamaterials: {From} Condensed Matter Theory to Subwavelength Physics},
  series    = {{NSF-CBMS} Regional Conference Series in the Mathematical Sciences},
  publisher = {American Mathematical Society},
  year     = {2026},
  volume  = {136},
}

@book{lapidus_van_frankenhuijsen2013,
  author    = {Lapidus, Michel L. and van Frankenhuijsen, Machiel},
  title     = {Fractal Geometry, Complex Dimensions and Zeta Functions: Geometry and Spectra of Fractal Strings},
  edition   = {2nd},
  series    = {Springer Monographs in Mathematics},
  publisher = {Springer},
  address   = {New York},
  year      = {2013},
  doi       = {10.1007/978-1-4614-2176-4}
}

@article{lapidus_radunovic_zubrinic2017distance,
  author  = {Lapidus, Michel L. and Radunovi{\'c}, Goran and {\v Z}ubrini{\'c}, Darko},
  title   = {Distance and tube zeta functions of fractals and arbitrary compact sets},
  journal = {Advances in Mathematics},
  volume  = {307},
  pages   = {1215--1267},
  year    = {2017},
  doi     = {10.1016/j.aim.2016.11.034}
}

@article{lapidus_radunovic_zubrinic2018tube,
  author        = {Lapidus, Michel L. and Radunovi{\'c}, Goran and {\v Z}ubrini{\'c}, Darko},
  title         = {Fractal tube formulas for compact sets and relative fractal drums: Oscillations, complex dimensions and fractality},
  journal       = {Journal of Fractal Geometry},
  volume        = {5},
  number        = {1},
  pages         = {1--119},
  year          = {2018},
  doi           = {10.4171/JFG/57},
  eprint        = {1604.08014},
  archivePrefix = {arXiv},
  primaryClass  = {math-ph}
}

@article{lapidus_pearse2010,
  author        = {Lapidus, Michel L. and Pearse, Erin P. J.},
  title         = {Tube formulas and complex dimensions of self-similar tilings},
  journal       = {Acta Applicandae Mathematicae},
  volume        = {112},
  number        = {1},
  pages         = {91--137},
  year          = {2010},
  doi           = {10.1007/s10440-010-9562-x},
  eprint        = {math/0605527},
  archivePrefix = {arXiv}
}

@article{hoffer2025tube,
  author        = {Hoffer, Will},
  title         = {Tube formulae for generalized von {K}och fractals through scaling functional equations},
  journal       = {Journal of Fractal Geometry},
  volume        = {12},
  number        = {1/2},
  pages         = {135--174},
  year          = {2025},
  doi           = {10.4171/JFG/155},
  eprint        = {2405.04712},
  archivePrefix = {arXiv},
  primaryClass  = {math.MG}
}

@article{kigami_lapidus1993,
  author  = {Kigami, Jun and Lapidus, Michel L.},
  title   = {{W}eyl's problem for the spectral distribution of {L}aplacians on P.C.F. self-similar fractals},
  journal = {Communications in Mathematical Physics},
  volume  = {158},
  number  = {1},
  pages   = {93--125},
  year    = {1993},
  doi     = {10.1007/BF02097233}
}

@article{lapidus1991fractal,
  author  = {Lapidus, Michel L.},
  title   = {Fractal drum, inverse spectral problems for elliptic operators and a partial resolution of the {W}eyl--{B}erry conjecture},
  journal = {Transactions of the American Mathematical Society},
  volume  = {325},
  number  = {2},
  pages   = {465--529},
  year    = {1991},
  doi     = {10.2307/2001638}
}

@article{lapidus1994analysis,
  author  = {Lapidus, Michel L.},
  title   = {Analysis on fractals, {L}aplacians on self-similar sets, noncommutative geometry and spectral dimensions},
  journal = {Topological Methods in Nonlinear Analysis},
  volume  = {4},
  number  = {1},
  pages   = {137--195},
  year    = {1994}
}

@incollection{teplyaev2004spectral,
  author    = {Teplyaev, Alexander},
  title     = {Spectral zeta functions of symmetric fractals},
  booktitle = {Fractal Geometry and Stochastics III},
  series    = {Progress in Probability},
  volume    = {57},
  pages     = {245--262},
  publisher = {Birkh{\"a}user},
  address   = {Basel},
  year      = {2004}
}

@article{brossard1986can,
  author  = {Brossard, Jean and Carmona, Ren{\'e}},
  title   = {Can one hear the dimension of a fractal?},
  journal = {Communications in Mathematical Physics},
  volume  = {104},
  number  = {1},
  pages   = {103--122},
  year    = {1986},
  doi     = {10.1007/BF01210795}
}

@phdthesis{wave_fractal_Anna,
  author = {Rozanova-Pierrat, Anna},
  title  = {Wave propagation and fractal boundary problems: Mathematical analysis and applications},
  school = {Universit{\'e} Paris-Saclay},
  type   = {Habilitation {\`a} diriger des recherches},
  year   = {2020},
  url    = {https://tel.archives-ouvertes.fr/tel-03060630/}
}

@article{yang2020photonic,
  author  = {Yang, Zhaoju and Lustig, Eran and Lumer, Yaakov and Segev, Mordechai},
  title   = {Photonic {F}loquet topological insulators in a fractal lattice},
  journal = {Light: Science \& Applications},
  volume  = {9},
  number  = {1},
  pages   = {128},
  year    = {2020},
  doi     = {10.1038/s41377-020-00354-z}
}

@article{canyellas2024topological,
  author  = {Canyellas, R. and Liu, Chen and Arouca, R. and Eek, L. and Wang, Guanyong and Yin, Yin and Guan, Dandan and Li, Yaoyi and Wang, Shiyong and Zheng, Hao and Liu, Canhua and Jia, Jinfeng and Morais Smith, C.},
  title   = {Topological edge and corner states in bismuth fractal nanostructures},
  journal = {Nature Physics},
  volume  = {20},
  number  = {9},
  pages   = {1421--1428},
  year    = {2024},
  doi     = {10.1038/s41567-024-02551-8}
}

@article{ammari_davies_hiltunen_yu2020,
  author        = {Ammari, Habib and Davies, Bryn and Hiltunen, Erik Orvehed and Yu, Sanghyeon},
  title         = {Topologically protected edge modes in one-dimensional chains of subwavelength resonators},
  journal       = {Journal de Math{\'e}matiques Pures et Appliqu{\'e}es},
  volume        = {144},
  pages         = {17--49},
  year          = {2020},
  doi           = {10.1016/j.matpur.2020.08.007},
  eprint        = {1906.10688},
  archivePrefix = {arXiv},
  primaryClass  = {math.AP}
}

@article{ebenfelt_khavinson_shapiro2002,
  author  = {Ebenfelt, Peter and Khavinson, Dmitry and Shapiro, Harold S.},
  title   = {A free boundary problem related to single-layer potentials},
  journal = {Annales Academiae Scientiarum Fennicae. Mathematica},
  volume  = {27},
  number  = {1},
  pages   = {21--46},
  year    = {2002},
  url     = {https://afm.journal.fi/article/view/135023}
}

@article{ammari_davies_hiltunen2022,
  author        = {Ammari, Habib and Davies, Bryn and Hiltunen, Erik Orvehed},
  title         = {Robust edge modes in dislocated systems of subwavelength resonators},
  journal       = {Journal of the London Mathematical Society},
  volume        = {106},
  number        = {3},
  pages         = {2075--2135},
  year          = {2022},
  doi           = {10.1112/jlms.12619},
  eprint        = {2001.10455},
  archivePrefix = {arXiv},
  primaryClass  = {math.AP}
}

@article{ammari_barandun_cao_feppon2023,
  author        = {Ammari, Habib and Barandun, Silvio and Cao, Jinghao and Feppon, Florian},
  title         = {Edge modes in subwavelength resonators in one dimension},
  journal       = {Multiscale Modeling \& Simulation},
  volume        = {21},
  number        = {3},
  pages         = {964--992},
  year          = {2023},
  doi           = {10.1137/23M1549419},
  eprint        = {2301.06747},
  archivePrefix = {arXiv},
  primaryClass  = {math.AP}
}

@article{leykam_bliokh_huang_chong_nori2017,
  author  = {Leykam, Daniel and Bliokh, Konstantin Y. and Huang, Chunli and Chong, Y. D. and Nori, Franco},
  title   = {Edge Modes, Degeneracies, and Topological Numbers in {N}on-{H}ermitian Systems},
  journal = {Physical Review Letters},
  volume  = {118},
  pages   = {040401},
  year    = {2017},
  doi     = {10.1103/PhysRevLett.118.040401}
}

@article{garrison_wright1988,
  author  = {Garrison, J. C. and Wright, E. M.},
  title   = {Complex geometrical phases for dissipative systems},
  journal = {Physics Letters A},
  volume  = {128},
  number  = {3--4},
  pages   = {177--181},
  year    = {1988},
  doi     = {10.1016/0375-9601(88)90905-X}
}

@article{gong_ashida_kawabata_et_al2018,
  author  = {Gong, Zongping and Ashida, Yuto and Kawabata, Kohei and Takasan, Kazuaki and Higashikawa, Sho and Ueda, Masahito},
  title   = {Topological Phases of {Non-Hermitian} Systems},
  journal = {Physical Review X},
  volume  = {8},
  number  = {3},
  pages   = {031079},
  year    = {2018},
  doi     = {10.1103/PhysRevX.8.031079}
}

@article{kawabata_shiozaki_ueda_sato2019,
  author  = {Kawabata, Kohei and Shiozaki, Ken and Ueda, Masahito and Sato, Masatoshi},
  title   = {Symmetry and Topology in {Non-Hermitian} Physics},
  journal = {Physical Review X},
  volume  = {9},
  number  = {4},
  pages   = {041015},
  year    = {2019},
  doi     = {10.1103/PhysRevX.9.041015}
}
\end{document}